\documentclass[10pt, reqno]{amsart}

\usepackage[utf8]{inputenc}
\usepackage[english]{babel}
\usepackage{csquotes} 
\usepackage[a4paper,twoside]{geometry}
\usepackage{graphicx}
\usepackage[backend=biber,style=alphabetic,bibencoding=utf8,language=auto,autolang=other,giveninits=true,doi=false,isbn=false,url=false,maxnames=10,sorting=nty]{biblatex}
\usepackage{caption} 
\usepackage[hyperfootnotes=false]{hyperref} 
\hypersetup{
	colorlinks = true,
	linkcolor = {blue},
	urlcolor = {red},
	citecolor = {blue}
}

\usepackage{amsmath}
\usepackage{amsthm}
\usepackage{amssymb}
\usepackage{esint} 
\usepackage{mathrsfs} 
\usepackage{faktor} 
\usepackage{dsfont} 
\usepackage[dvipsnames]{xcolor}
\usepackage{array}
\usepackage{hhline}
\usepackage{enumitem} 
\usepackage{comment} 
\usepackage[toc,page]{appendix} 
\usepackage{imakeidx} 
\usepackage{xparse} 
\usepackage{mathtools}
\usepackage{fancyhdr} 
\usepackage{ifthen} 
\usepackage{forloop} 
\usepackage{xstring}
\usepackage{tikz}
\usetikzlibrary{babel} 
\usetikzlibrary{cd} 
\usepackage{emptypage} 
\usepackage{listings} 
\usepackage{mathabx} 

\usepackage[nameinlink,capitalise,sort]{cleveref}
\crefname{equation}{}{} 
\crefname{enumi}{}{} 

\newcounter{results}[section] 

\theoremstyle{plain}
\newtheorem{theorem}[results]{Theorem}
\newtheorem{lemma}[results]{Lemma}
\newtheorem{proposition}[results]{Proposition}
\newtheorem{corollary}[results]{Corollary}

\newtheorem*{theorem*}{Theorem}
\newtheorem*{lemma*}{Lemma}
\newtheorem*{proposition*}{Proposition}
\newtheorem*{corollary*}{Corollary}
\newtheorem*{exercise*}{Exercise}
\newtheorem*{fact*}{Fact}
\newtheorem*{problem*}{Problem}
\newtheorem*{conjecture*}{Conjecture}

\theoremstyle{remark}
\newtheorem{remark}[results]{Remark}

\newtheorem*{remark*}{Remark}
\newtheorem*{question*}{Question}

\theoremstyle{definition}

\newtheorem*{definition*}{Definition}
\newtheorem*{example*}{Example}

\numberwithin{equation}{section}

\makeatletter
\newcommand{\newreptheorem}[2]{%
  \newtheorem*{rep@#1}{\rep@title}%
  \newenvironment{rep#1}[1]%
    {\def\rep@title{#2 \ref*{##1}}\begin{rep@#1}}%
    {\end{rep@#1}}%
}%
\makeatother
\theoremstyle{plain}
\newreptheorem{theorem}{Theorem}

\newcommand{\N}{\ensuremath{\mathbb N}} 
\newcommand{\Z}{\ensuremath{\mathbb Z}} 
\newcommand{\R}{\ensuremath{\mathbb R}} 

\renewcommand{\S}{\ensuremath{\mathbb S}} 

\makeatletter 
\DeclarePairedDelimiter{\@tmpabs}{\lvert}{\rvert}
\newcommand{\@absstar}[1]{{\@tmpabs*{#1}}}
\newcommand{\@absnostar}[2][]{{\@tmpabs[#1]{#2}}}
\newcommand{\abs}{\@ifstar\@absstar\@absnostar}
\makeatother

\makeatletter 
\DeclarePairedDelimiter{\@tmpnorm}{\lVert}{\rVert}
\newcommand{\@normstar}[1]{{\@tmpnorm*{#1}}}
\newcommand{\@normnostar}[2][]{{\@tmpnorm[#1]{#2}}}
\newcommand{\norm}{\@ifstar\@normstar\@normnostar}
\makeatother

\newcommand{\scalprod}[1]{\ensuremath{\langle #1\rangle}} 
\newcommand{\Haus}{\ensuremath{\mathscr H}} 

\renewcommand{\complement}{\ensuremath{\mathsf{c}}} 

\newcommand{\eps}{\ensuremath{\varepsilon}}
\newcommand{\emptyparam}{\ensuremath{\,\cdot\,}}
\newcommand{\defeq}{\ensuremath{\coloneqq}}

\newcommand{\lapl}{\ensuremath{\Delta}}

\newcommand{\meas}{\ensuremath{\mathscr{M}}} 

\DeclareMathOperator{\annulus}{An}

\begin{document}

\title[Non-degeneracy, stability, and symmetry for the fractional CKN inequality]{Non-degeneracy, stability and symmetry for the fractional Caffarelli--Kohn--Nirenberg inequality}

\author[N. De Nitti]{Nicola De Nitti}
\address[N. De Nitti]{EPFL, Institut de Mathématiques, Station 8, 1015 Lausanne, Switzerland.}
\email{nicola.denitti@epfl.ch}

\author[F. Glaudo]{Federico Glaudo}
\address[F. Glaudo]{Institute for Advanced Study, School of Mathematics, 1 Einstein Dr., Princeton NJ 08540, U.S.A. \and Princeton University, Fine Hall 310, Washington Road, Princeton NJ 08544, U.S.A.}
\email{fglaudo@ias.edu}

\author[T. König]{Tobias König}
\address[T. König]{Goethe-Universität Frankfurt, Institut für Mathematik, 
Robert-Mayer-Str. 10, 60325 Frankfurt am Main, Germany.}
\email{koenig@mathematik.uni-frankfurt.de}

\begin{abstract}
    The fractional Caffarelli--Kohn--Nirenberg inequality states that
\[        \int_{\R^n}\int_{\R^n} \frac{(u(x)-u(y))^2}{\abs{x}^\alpha\abs{x-y}^{n+2s}\abs{y}^\alpha} \,\mathrm d x  \,\mathrm d y 
        \ge \Lambda_{n, s, p, \alpha,\beta} 
        \norm*{u\abs{x}^{-\beta}}_{L^p}^2,
\]
for $0<s<\min\{1, n/2\}$, $2<p<2^*_s$, and $\alpha,\beta\in\R$ so that $\beta-\alpha = s - n\big(\frac12 - \frac1p\big)$ and  $-2s < \alpha <  \frac{n-2s}{2}$.

    Continuing the program started in Ao et al. (2022) \cite{AoDelaTorreGonzalez2022}, we establish the non-degeneracy and sharp quantitative stability of minimizers for $\alpha\ge 0$. Furthermore, we show that minimizers remain symmetric when $\alpha<0$ for $p$ very close to $2$. 
    
    Our results fit into the more ambitious goal of understanding the symmetry region of the minimizers of the fractional Caffarelli--Kohn--Nirenberg inequality.

    We develop a general framework to deal with fractional inequalities in $\R^n$, striving to provide statements with a minimal set of assumptions. Along the way, we discover a Hardy-type inequality for a general class of radial weights that might be of independent interest.
\end{abstract}

\maketitle

\section{Introduction and main results}
\label{sec:intro}

The \emph{Caffarelli--Kohn--Nirenberg} (CKN) inequality, first introduced in \cite{CKN82,CKN84}, states that 
\begin{align}\label{eq:ckn-classic}
\int_{\mathbb{R}^n}\abs{x}^{-2 \alpha}\abs{\nabla u}^2 \, \mathrm d x \ge  \Lambda_{n,p,\alpha, \beta} \left(\int_{\mathbb{R}^n}\abs{x}^{-\beta(p)}\abs{u}^{p} \, \mathrm d x\right)^{\frac{2}{p}}  
\end{align}
for any $u \in C^\infty_c(\R^n)$, 
$n \in \N$, $-\infty<\alpha< \frac{n-2}{2}$, $p\defeq \frac{2n}{n-2+2(\beta-\alpha)}$, and (if $n \geq 3$, for simplicity) $\alpha \leq \beta<\alpha+1$.

Particular cases of \cref{eq:ckn-classic} include the Hardy ($\alpha = 0$, $\beta = 1$) and
Sobolev ($\alpha = \beta = 0$) inequalities.

The CKN inequality \cref{eq:ckn-classic} has been vastly studied in the literature as a paradigmatic example for  the phenomenon of symmetry-breaking. That is, even if all the terms in \cref{eq:ckn-classic} are rotationally invariant, its minimizers are not rotationally invariant for certain values of the parameters $\alpha$ and $\beta$~\cite{CW01}. Since the minimizers $U = U(\alpha, \beta)$ of \cref{eq:ckn-classic} among radial functions are explicit, Felli and Schneider \cite{FS03} were able to compute the exact region of $(\alpha, \beta)$ such that the Hessian in $U$ of the quotient functional corresponding to \cref{eq:ckn-classic} has a negative eigenvalue. After some intermediate results \cite{Dolbeault2009, Dolbeault2012}, in the breakthrough paper \cite{Dolbeault2016}, Dolbeault, Esteban, and Loss were able to prove using a flow method that the ``Felli--Schneider region'' coincides exactly with the symmetry-breaking region. The quantitative stability of \cref{eq:ckn-classic}, in the sense of the classical result by Bianchi and Egnell \cite{BianchiEgnell}, has been investigated in the recent works \cite{WW22, WW23, FP23}.

\subsection{The fractional CKN inequality}
A natural fractional counterpart of \cref{eq:ckn-classic} is given by the inequality 
\begin{equation}\label{eq:fckn}\tag{CKN}
    \norm{w}^2_{D^s_\alpha(\R^n)}\defeq
    \int_{\mathbb R^n}\int_{\mathbb R^n} \frac{(u(x)-u(y))^2}{\abs{x}^{\alpha}\abs{x-y}^{n+2s}\abs{y}^\alpha} \,\mathrm d x  \,\mathrm d y 
    \ge \Lambda_{n, s, p, \alpha,\beta}     \norm*{u\abs{\emptyparam}^{-\beta}}_{L^p}^2 \qquad \text{for any $u \in D^s_\alpha(\mathbb R^n)$,} 
\end{equation}
{which was first studied (in a different, but equivalent formulation) in \cite{Ghoussoub2015} and more recently in \cite{AoDelaTorreGonzalez2022}.}
We shall refer to \cref{eq:fckn} as the \emph{fractional CKN inequality}. The constant $\Lambda_{n,s,p,\alpha,\beta}$ is assumed to be the optimal one. Other fractional variants of \cref{eq:ckn-classic} have appeared in \cite{Abdellaoui2017, Nguyen2018}, but we will be exclusively concerned with \cref{eq:fckn}.

Here, $D^s_\alpha(\R^n)$ denotes the closure of $C^\infty_c(\R^n)$ with respect to the norm on the left side of \cref{eq:fckn}. The parameters involved in \cref{eq:fckn} shall satisfy 
\begin{equation}\label{parameters-ckn}
\begin{aligned} 
&n\ge 1,\quad
0 < s < \min\Big\{1, \frac n2\Big\}, \quad
2 < p < 2^*_s \defeq \frac{2n}{n-2s}, \\ 
&-2s < \alpha < \frac n2 -s, \quad
\beta-\alpha = s - n\left(\frac12 - \frac1p\right).
\end{aligned}
\end{equation}
We stress that the restriction $\alpha > -2s$ is not present in the classical inequality \cref{eq:ckn-classic}. 
If we replace $u$ by $u(\lambda\emptyparam)\lambda^{n/2-\alpha-s}$, the left-hand side and the right-hand side of \cref{eq:fckn} do not change value (because $n/2-\alpha-s = n/p-\beta$).

Similarly to \cref{eq:ckn-classic}, the inequality \cref{eq:fckn} interpolates between the fractional Hardy and Sobolev inequalities. 

Ultimately, the goal would be to reproduce the above-mentioned results on \cref{eq:ckn-classic} for \cref{eq:fckn}, namely to obtain a complete characterization of the symmetry-breaking region. In doing so, however, one faces multiple fundamental difficulties: the minimizers of \cref{eq:fckn} among radial functions do not have a known explicit expression, flow methods are not available in the fractional setting, and ODE techniques break down. 

In \cite{AoDelaTorreGonzalez2022}, the analysis in this direction was started and some remarkable partial results were obtained. We summarize the most significant results from \cite{AoDelaTorreGonzalez2022} in connection with our paper:

\begin{enumerate}[label=(\Alph*)]
    \item A minimizer always exists; moreover, if $0 \le \alpha < \frac{n-2s}{2}$, it is radially symmetric \cite[Theorem 1.2 (ii), (iv)]{AoDelaTorreGonzalez2022}\footnote{In \cite[Theorem 1.2 (ii), (iv)]{AoDelaTorreGonzalez2022}, it is also claimed that for $\alpha \geq 0$ the minimizer $U$ is non-increasing in the radial variable. This claim is likely to be true (it is true for the classical CKN inequality \cref{eq:ckn-classic}), but not justified by the arguments given there: in \cite[Proposition 4.1]{AoDelaTorreGonzalez2022} it is only shown that $W(x) = \abs{x}^{-\alpha} U(x)$ must be radially decreasing. However, if $\alpha > 0$, this does not necessarily imply that $U$ is radially decreasing.}  {(or \cite[Theorem 1.1]{Ghoussoub2015})};  
    \item Minimizers are non-radial for certain valid choices of the parameters $s,p,\alpha,\beta$~\cite[Theorem 1.4]{AoDelaTorreGonzalez2022}; 
    \item \label{ADM nondeg} If the global minimizer is a radial function, then it is non-degenerate in the space of radial functions~ \cite[Theorem 1.5]{AoDelaTorreGonzalez2022};
    \item \label{ADM uniqueness} If the global minimizer is a radial function, then it is unique (up to scaling) \cite[Theorem 1.6]{AoDelaTorreGonzalez2022}.
\end{enumerate}
The last two items of the previous list are highly non-trivial since non-degeneracy of fractional-order equations is generally a hard question. For instance, in the context of fractional Schrödinger equations, it was only obtained in the seminal papers \cite{FL2013, FLS2016}.

\subsection{Non-degeneracy and stability for \texorpdfstring{$\alpha \geq 0$}{alpha non-negative}}
Despite these achievements, the method of \cite{AoDelaTorreGonzalez2022} to prove non-degeneracy is restricted to the class of radial functions.\footnote{Indeed, as explained in \cite[p. 7]{AoDelaTorreGonzalez2022}, their method crucially relies on comparison with the radial solution to a linearized equation. In higher angular momentum channels, no such solution is available. In fact, one precisely needs to prove that there exists none.} In this paper, we develop a different strategy to overcome this limitation, and prove the full non-degeneracy for $\alpha \geq 0$. 

To state our first main result, we fix a minimizer $U$ of \cref{eq:fckn}, normalized so that it satisfies the Euler--Lagrange equation
\begin{equation}
    \label{U equation} \tag{CKN-eq}
    \mathcal L_{s,\alpha} U(x) \defeq \mathrm{P.V.}_x \int_{\R^n} \frac{u(x)-u(y)}{\abs{x}^\alpha \abs{x-y}^{n+2s} \abs{y}^\alpha} \, \mathrm d y  = \frac{U^{p-1}}{\abs{x}^{\beta p}}.
\end{equation}
By the results of \cite{AoDelaTorreGonzalez2022}, such $U$ is radially symmetric and unique up to scaling if $\alpha \geq 0$.  We denote by 
\begin{equation}
    \label{U lambda definition}
    U_\lambda(x) \defeq \lambda^{\frac{n-2s}{2} - \alpha} U(\lambda x)
\end{equation}
its dilations, which also solve \cref{U equation}. 

By computing the second variation of \cref{eq:fckn} and using the fact that $U$ is a minimizer, we deduce  
\begin{equation}
    \label{U hessian}
    \norm{\varphi}_{D^s_\alpha(\R^n)}^2 \geq (p-1) \int_{\R^n} \varphi^2 \frac{U^{p-2}}{\abs{x}^{\beta p}} 
\end{equation} 
for every $\varphi \in D^s_\alpha(\R^n)$ such that $\int_{\R^n} \frac{U^{p-1}}{\abs{x}^{\beta p}} \varphi = 0$. 

The following result encodes the non-degeneracy of positive solutions to \eqref{U equation}. 

\begin{theorem}[Non-degeneracy for positive solutions]
\label{theorem ckn nondeg intro}
    Assume \cref{parameters-ckn} and, additionally, $\alpha\ge 0$.
    Let $U\in D_\alpha^s(\R^n)$ be a non-negative solution to \cref{U equation}.  
    If $\varphi \in D^s_\alpha(\R^n)$ solves 
    \[ \mathcal L_{s,\alpha} \varphi = (p-1) \frac{U^{p-2}}{\abs{x}^{\beta p}} \varphi, \]
    then $\varphi$ is a scalar multiple of $\partial_\lambda|_{\lambda = 1} U_\lambda$. 
\end{theorem}

We emphasize that in \cref{theorem ckn nondeg intro} we only assume $U$ to solve \cref{U equation}, not to minimize \cref{eq:fckn}. 

Notice that, if $U$ is a minimizer of \cref{eq:fckn}, then $U$ does not change sign because, otherwise, $\| |U| \|_{D^s_\alpha(\R^n)}$ would be strictly smaller than  $\| U \|_{D^s_\alpha(\R^n)}$. Hence, a suitable scalar multiple of $U$ satisfies the assumption of \cref{theorem ckn nondeg intro}.

\begin{corollary}[Non-degeneracy for minimizers]
    \label{corollary ckn nondeg intro}
    Assume \cref{parameters-ckn} and, additionally, $\alpha\ge 0$.
    Let $U\in D_\alpha^s(\R^n)$ be a minimizer of \cref{eq:fckn} normalized to satisfy \cref{U equation}. 
    If $\varphi \in D^s_\alpha(\R^n)$ satisfies $\int_{\R^n} \frac{U^{p-1}}{\abs{x}^{\beta p}} \varphi = 0$ and equality holds in \cref{U hessian}, then $\varphi$ is a scalar multiple of $\partial_\lambda|_{\lambda = 1} U_\lambda$. 
\end{corollary}

The non-degeneracy of $U$ is the crucial ingredient needed to prove the sharp quantitative stability of \cref{eq:fckn}, following the classical strategy pioneered by Bianchi and Egnell~\cite{BianchiEgnell}. Thus, as a consequence of \cref{theorem ckn nondeg intro}, we obtain the following stability theorem for the fractional CKN inequality, extending the recent result of \cite{WW22} to the fractional case. 

\begin{theorem}[Stability]
    \label{theorem ckn stability intro}
    Assume \cref{parameters-ckn} and, additionally, $\alpha\ge 0$.
    Let $U\in D_\alpha^s(\R^n)$ be a minimizer of \cref{eq:fckn}.
    There exists $\kappa > 0$ such that, for all $u \in D^s_\alpha(\R^n)$, it holds 
    \[
      \norm{u}_{D^s_\alpha(\R^n)}^2 
    - \Lambda_{n,s,p,\alpha,\beta}
    \norm*{u\abs{\emptyparam}^{-\beta}}_{L^p}^2 \geq \kappa \inf_{c \in \R, \lambda > 0} \norm{u - c U_\lambda}_{D^s_\alpha(\R^n)}^2,
    \]
    where $U_\lambda$ is given by \cref{U lambda definition}.
\end{theorem}

For functions $u$ supported on a domain $\Omega$ of finite measure $|\Omega| < \infty$, we can deduce from \cref{theorem ckn stability intro} a remainder term version of inequality \cref{eq:fckn}. The remainder term is given in terms of the \emph{weak $L^r$-norm}, that is, for $r \in (1, \infty)$, 
\[ \norm{u}_{L^{r, \infty}}  \defeq  \sup_{A \subset \R^n, |A| > 0} |A|^{\frac{1}{r} 
- 1} \int_A |u|. \]

\begin{corollary}[\cref{eq:fckn} inequality with remainder]
    \label{corollary remainder CKN}
Assume \cref{parameters-ckn} and, additionally, $\alpha \geq 0$. 
There exists $c =  c(n,s,\alpha,p) > 0$ such that, for all $\Omega \subseteq \R^n$, with $\abs{\Omega} < \infty$, and for all $u \in D^s_\alpha(\R^n)$ supported in $\Omega$, 
    \[ \norm{u}_{D^s_\alpha(\R^n)}^2  -  \Lambda_{n,s,p,\alpha,\beta} \norm*{u |\emptyparam|^{-\beta}}_{L^p(\Omega)}^2 \geq  c |\Omega|^{-\frac{n-2s-2\alpha}{n}} \norm{u|\emptyparam|^{-\alpha}}_{L^{\frac{n}{n-2s-\alpha}, \infty}}^2.   \]
\end{corollary}

For the Sobolev inequality ($s=1$, $\alpha = \beta = 0$), this refinement with remainder is due to Brezis and Lieb (see \cite{BrezisLieb1985}). In \cite{BianchiEgnell}, Bianchi and Egnell  gave a simplified proof by deducing it from their stability estimate. For the classical ($s = 1$) CKN inequality \cref{eq:ckn-classic} it was proved in \cite{Radulescu2002} (for $\alpha = 0$) and in \cite{Wang2003}. 

For fractional $s \in (0, \frac{n}{2})$, in \cite{Chen2013}, Chen, Frank, and Weth  gave a proof in the case  $\alpha=\beta=0$. Our proof is an adaptation (and slight simplification) of their argument.

\subsection{Symmetry for \texorpdfstring{$\alpha < 0$}{alpha negative}}

We complement our previous analysis by a symmetry result for $\alpha < 0$ and $p$ close to $2$. It is the first positive symmetry result for \cref{eq:fckn} in the parameter range $\alpha <0$. 

\begin{theorem}[Symmetry]
    \label{thm:perturbative-symmetry-ckn}
    For every $n\ge 1$, $0 < s < \min\{1,n/2\}$, and $-2s < \alpha_0 < 0$, there exists $\eps = \eps(n,s,\alpha_0) >0$ such that the following statement holds.
    If $\alpha \in (\alpha_0, 0)$ and  $p \in (2, 2 + \varepsilon)$, then every minimizer of \cref{eq:fckn} is a radial function.
\end{theorem} 

\begin{remark}[Constraints on the parameters]
We restrict ourselves to the case $\alpha <0$ because the case $\alpha \ge 0$ is contained in \cite[Theorem 1.2 (ii), (iv)]{AoDelaTorreGonzalez2022}. 
By the symmetry-breaking result from \cite[Theorem 1.4 (ii)]{AoDelaTorreGonzalez2022}, the symmetry range in \cref{thm:perturbative-symmetry-ckn} cannot be uniform in $\eps$ for $\alpha$ close to $-2s$; in other words, $\eps$ goes to $0$ as $\alpha_0$ approaches $-2s$.
 \end{remark}

\subsection{Comments on the proofs}
Differently from most works on the CKN inequality \cref{eq:ckn-classic} and in particular differently from \cite{AoDelaTorreGonzalez2022}, we work entirely on $\R^n$ and do not pass to cylindrical variables. 
Instead, we reformulate \cref{eq:fckn} by setting $w(x) \defeq u(x)\abs{x}^{-\alpha}$, obtaining the inequality
\begin{equation}
    \label{hardy ineq intro}
      \norm{w}_{{\dot{H}^s}}^2 + C(\alpha)\norm{w\abs{\emptyparam}^{-s}}_{L^2}^2 
\ge \tilde\Lambda
\norm*{w\abs{\emptyparam}^{-(\beta - \alpha)}}_{L^p}^2 .
\end{equation}
This formulation was first found in \cite[eq. (4.3)]{FrankLiebSeiringer2008}. 
It was studied in \cite{Ghoussoub2015,Dipierro2016} and also appears in \cite{AoDelaTorreGonzalez2022}.
We refer to \cref{sec:hardy-formulation,sec:notation} for more explanations about \cref{hardy ineq intro} and the space $\dot H^s(\R^n)$, respectively.

We strive to provide general statements with the minimal reasonable set of assumptions.
This approach has two positive byproducts:
\begin{itemize}
    \item We state and prove a number of statements (e.g., the maximum principles~ \cref{thm:global-max-principle,thm:strong-max-principle}, the fractional mean value property~\cref{prop:fractional_mean_value_property}, the strict negativity of the derivative of radial decreasing functions~\cref{thm:radial-decreasing-negative-derivative}, and the general Hardy-type inequality~\cref{thm:general-hardy}) that may be useful for other problems. We strive to state such statements with the minimal reasonable set of assumptions.
    \item Our proofs are free of computations, which are often a burden to take care of when working with the CKN inequality. The drawback is that the proofs are more technical and abstract than usual in this area (e.g., we use rather heavily the theory of distributions).
\end{itemize}

Our proof of \cref{theorem ckn nondeg intro} is inspired by the paper \cite{MusinaNazarov2021}. There, non-degeneracy of minimizers of the fractional Hardy--Sobolev inequality (corresponding to the value $\alpha = 0$ in our framework) is obtained.
Our adaptation of their argument turned out to be completely different from the original one: we avoid involved computations as well as the use of the Caffarelli--Silvestre extension. 
Along the road, we obtain a general Hardy-type inequality, \cref{thm:general-hardy}, for functions orthogonal to radial ones. 
Namely, for any $\varphi \in \dot H^s(\R^n)$ orthogonal to radial functions and every radial $U$ such that both $U$ and $(-\Delta)^s U$ are radially decreasing, it holds 
\begin{equation}
    \label{general hardy ineq intro}
     \norm{ \varphi }^2_{\dot H^s(\R^n)} \geq \int_{\R^n} \varphi^2 \frac{((-\Delta)^s U)'}{U'}.
\end{equation} 
This general inequality might be of some independent interest. We expect it to be implicitly known when $s =1$, but we were unable to find a reference  even in this case. 
Interestingly, an inequality similar in spirit to \cref{general hardy ineq intro} has appeared independently in the very recent work \cite{FallWeth2023}. Like in our paper, \cite[Lemma 2.1]{FallWeth2023} is one of the main novel ingredients there used to prove the non-degeneracy of a fractional PDE. 

Even though the inequalities \cref{general hardy ineq intro} and \cite[Lemma 2.1]{FallWeth2023} are not equivalent (for example, because they concern different classes of functions), they are definitely related. It would be interesting to further clarify and systematize the role played by such inequalities in the context of non-degeneracy and related issues.

For $s=1$, a result corresponding to \cref{thm:perturbative-symmetry-ckn} was obtained in \cite[Theorem 1.1]{Dolbeault2009}. We adapt the argument of \cite{Dolbeault2009} to the fractional setting, taking care of a major difficulty that arises because  the radial minimizers are not explicit. The key point is that an a priori control, independent of $p$, is needed on the (global) minimizers; in the classical case, such control can be obtained by comparison with the radial minimizers (see \cite[Lemma 2.2]{Dolbeault2009}). We bypass the comparison entirely with a rather technical bootstrap argument (see \cref{lem:asymptotic-growth}). 
Moreover, we streamline the proof of \cite{Dolbeault2009} (namely, we avoid working in cylindrical coordinates and we have virtually no computations).

\subsection{Structure of the paper}

In \Cref{sec:hardy-formulation}, we introduce the reformulation \cref{hardy ineq intro} of the fractional CKN inequality and we restate our main results in this new setting. It is these equivalent versions that we are actually going to prove. 

\Cref{sec:notation} is devoted to preliminaries. We recall the definitions of fractional Laplacian and fractional Sobolev spaces. Furthermore we state and prove a number of basic facts about them that will be used throughout the paper. We also collect some useful notation and results about radial functions and distributions.

In \Cref{sec:max}, we prove two maximum principles for the fractional Laplacian.
To this purpose, we need to generalize the theory developed in \cite[Section 2.2]{Silvestre2007} to distributions that are not locally integrable and we need some technical one-dimensional lemmas that are contained in \cref{sec:smoothing}.
In \Cref{sec:radial-negative}, we show that a radially weakly decreasing non-negative function such that $(-\lapl)^sU$ is radially weakly decreasing in $\R^n\setminus\{0\}$ is either constant or $U'<0$ in $\R^n\setminus\{0\}$. In \Cref{sec:general-hardy}, we prove the general Hardy-type inequality \cref{general hardy ineq intro}. 

These results from \cref{sec:max,sec:radial-negative,sec:general-hardy} are the basis for \cref{sec:non-degeneracy-proof}, where we prove \cref{thm:non-degeneracy} (respectively \cref{theorem ckn nondeg intro}) and \cref{corollary nondeg hardy} (respectively \cref{corollary ckn nondeg intro}).
In \Cref{sec:stability}, we prove the quantitative stability result, \cref{theorem stability can} (respectively \cref{theorem ckn stability intro}). 
In \cref{sec:symmetry-proof}, we prove the symmetry result, \cref{thm:perturbative-symmetry-hardy} (respectively \cref{thm:perturbative-symmetry-ckn}).


\section{Main results in the Hardy formulation}
\label{sec:hardy-formulation}

Let us introduce the alternative formulation of the fractional Caffarelli--Kohn--Nirenberg inequality that we will work with, and the corresponding Euler--Lagrange equation. 
We refer to \Cref{ssec:fraclaplace} for the definition of the fractional Sobolev space $\dot{H}^s(\R^n)$, the constant $C_{n,s}$ and the fractional Laplace operator $(-\lapl)^s$. 

Setting $u(x) \eqqcolon \abs{x}^\alpha w(x)$, one has the identity 
\begin{equation}
    \label{hardy ckn trafo formula}
    C_{n,s}\abs{x}^\alpha\mathcal L_{s,\alpha}u-(-\lapl)^sw(x) = \frac{C(\alpha)}{\abs{x}^{2s}}w(x)
\end{equation}
(see \cref{U equation} for the definition of $\mathcal L_{s,\alpha}$), where
\begin{equation}
    \label{C(alpha) definition}
     C(\alpha) = C(n, s, \alpha) \defeq C_{n,s} \mathrm{P.V.}_{e_1}\int_{\R^n} \frac{1-\abs{z}^\alpha}{\abs{z}^\alpha\abs{e_1-z}^{n+2s}} \, \mathrm d z .
\end{equation}
Using this identity, we can prove
\begin{equation}
\label{hardy ckn integral identities}
    \frac{C_{n,s}}{2}\norm{u}_{D_\alpha^s(\R^n)}^2  = \norm{w}_{\dot H^s}^2 + C(\alpha)\norm{w\abs{\emptyparam}^{-s}}_{L^2}^2
    \quad
    \text{and}
    \quad
    \norm{u\abs{\emptyparam}^{-\beta}}_{L^p}^2
    =
    \norm{w\abs{\emptyparam}^{-t}}_{L^p}^2,
\end{equation}
for $t\defeq s - n\big(\frac12 - \frac1p\big)$.
As a byproduct of this computation, we note that $u \in D^s_\alpha(\R^n)$ if and only if $w \in \dot H^s(\R^n)$. 

Thus, we deduce that \cref{eq:fckn} is equivalent to
\begin{align}
    \label{eq:hardy} \tag{H-CKN}
    \norm{w}_{{\dot{H}^s}}^2 + C(\alpha)\norm{w\abs{\emptyparam}^{-s}}_{L^2}^2 
    &\ge \tilde\Lambda_{n,s,p,\alpha}
    \norm*{w\abs{\emptyparam}^{-t}}_{L^p}^2 \qquad \text{for any }w \in \dot H^s(\R^n) ,
\end{align}
where $\tilde\Lambda_{n,s,p,\alpha}\defeq \frac{C_{n,s}}{2}\Lambda_{n,s,p,\alpha,\beta}$. 

Due to the presence of the Hardy term $\norm{w\abs{\emptyparam}^{-s}}_{L^2}^2$, we refer to \cref{eq:hardy} as the \emph{Hardy formulation} of the fractional CKN inequality. 

By the results of \cite{AoDelaTorreGonzalez2022}, there exists a $W \in \dot H^s(\R^n)$ which minimizes \cref{eq:hardy} and satisfies distributionally (and pointwise in $\R^n\setminus\{0\}$) the Euler--Lagrange equation 
\begin{equation}
\label{eq:hardy-eq} \tag{H-CKN-eq}
    (-\Delta)^s W + C(\alpha) \frac{W}{\abs{x}^{2s}} = \frac{W^{p-1}}{\abs{x}^{tp}}. 
\end{equation}
 To obtain \cref{eq:hardy-eq} in this clean form without multiplicative constants, we are normalizing $W$ so that $\tilde\Lambda = \norm{W\abs{\emptyparam}^{-t}}^{{p-2}}_{L^p}$. 
 Moreover, if $\alpha \geq 0$, then $W$ is (up to scaling) unique and radially non-increasing.
 
 For a fixed $W$, we denote by 
 \begin{equation}
     \label{W lambda}
     W_\lambda(x) \defeq \lambda^{\frac{n-2s}{2}} W(\lambda x) 
 \end{equation}
 its dilations, which also solve \cref{eq:hardy-eq}. Computing the second variation in $W$ of the quotient associated to \cref{eq:hardy} yields that, for every $\varphi \in \dot H^s(\R^n)$ such that $\int_{\R^n} \varphi \frac{W^{p-1}}{\abs{x}^{tp}} = 0$,   
 \begin{equation}
     \label{W Hessian}
     \norm{\varphi}_{\dot H^s}^2 + C(\alpha) \norm{\varphi \abs{\emptyparam}^{-s} }_{L^2}^2 \geq (p-1) \int_{\R^n} \varphi^2 \frac{W^{p-2}}{\abs{x}^{tp}} \mathrm d x. 
 \end{equation}

For clarity and future reference, let us repeat here the exact and complete assumptions on the parameters. We assume
\begin{equation}
    \label{parameters-hardy}
    \begin{aligned}
   & n \geq 1, \quad 0 < s < \min\left\{1, \frac n2\right\}, \quad 2 <  p < 2^*_s \defeq \frac{2n}{n-2s}, \\ 
    &-2s <\alpha < \frac n2-s, \quad  t \defeq s-n\left(\frac12-\frac1p\right) .
    \end{aligned}
\end{equation}

\subsection{Properties of \texorpdfstring{$C(\alpha)$}{the constant C(alpha)}}

Before going on, let us establish some basic properties of the function $\alpha\mapsto C(\alpha)$.\footnote{See also  \cite[Appendix A]{AoDelaTorreGonzalez2022}. We warn the reader that the constant denoted by $C(\alpha)$ in \cite{AoDelaTorreGonzalez2022} is different from ours; in fact our $C(\alpha)$ corresponds to the constant $\kappa_{\alpha, \gamma}^{n, -\alpha}$ from \cite[eq. (9.2)]{AoDelaTorreGonzalez2022}.} 
Observe that $C(\alpha)$ is well-defined (and finite) as soon as $0<s<1$ and $-2s<\alpha<n$, which is implied by our assumptions \cref{parameters-hardy}. 

For any function $\varphi:\R^n\to\R$, we have $\int_{\R^n}\varphi(z) \,\mathrm d z  = \int_{B_1} \varphi(z) + \varphi(z/\abs{z}^2)\abs{z}^{-2n} \,\mathrm d z$ and, furthermore, $\abs*{\frac{z}{\abs{z}^2}-e_1}=\frac{\abs{z-e_1}}{\abs{z}}$. Applying these two identities, we obtain
\begin{align}
\nonumber
    C(\alpha) &= {C_{n,s}} \int_{B_1} \frac{1-\abs{z}^\alpha}{\abs{z}^\alpha\abs{e_1-z}^{n+2s}} 
    +\frac{1-\abs{z}^{-\alpha}}{\abs{z}^{-\alpha}\abs{e_1-z}^{n+2s}\abs{z}^{-n-2s}}\abs{z}^{-2n}\,\mathrm d z
    \\
    &=
    {C_{n,s}} \int_{B_1} \frac{1-\abs{z}^\alpha}{\abs{e_1-z}^{n+2s}} \big(\abs{z}^{-\alpha}-\abs{z}^{2s-n}\big)\,\mathrm d z. 
    \label{C alpha B1}
\end{align}
As a consequence (consistently with \cite[Corollary 9.2]{AoDelaTorreGonzalez2022}),
\begin{itemize}
    \item $C(\alpha)>0$ for $-2s<\alpha < 0$,
    \item $C(\alpha)=0$ for $\alpha=0$,
    \item $C(\alpha)<0$ for $0<\alpha<\frac n2-s$.
\end{itemize}
Let us emphasize that the change of sign of $C(\alpha)$ at $\alpha=0$ is the reason why the problem at hand becomes more tractable for $\alpha\ge 0$.
Moreover, the function $\alpha \mapsto C(\alpha)$ is strictly decreasing on $(-2s, \frac n2-s)$ because (as can be seen differentiating \cref{C alpha B1})
\begin{equation*}
    \frac{\mathrm d}{\mathrm d\alpha} C(\alpha) =
    C_{n,s}\int_{B_1} \frac{\abs{z}^{\alpha+2s-n}-\abs{z}^{-\alpha}}{\abs{e_1-z}^{n+2s}}\log\abs{z}\,\mathrm dz < 0.
\end{equation*}

Let $C_\mathrm{Hardy}(s)$ denote the best constant in the fractional Hardy inequality 
    \[ \norm{w}_{\dot H^s(\R^n)}^2 \geq C_\mathrm{Hardy}(s) \norm{w\abs{\emptyparam}^{-s}}_{L^2}^2. \]
Since $\norm{w}_{{\dot{H}^s}}^2 + C(\alpha)\norm{w\abs{\emptyparam}^{-s}}_{L^2}^2 \geq 0$ for all $w\in\dot H^s(\R^n)$ by \cref{eq:hardy}, it must hold $C(\alpha) \ge -C_\mathrm{Hardy}(s)$ for all $-2s<\alpha < \frac n2-s$. Thanks to the strict monotonicity of $C(\alpha)$, this implies
\begin{equation}\label{eq:compare_calpha_chardy}
    C(\alpha) > -C_\mathrm{Hardy}(s) \quad\text{for all $-2s<\alpha < \frac n2-s$.}
\end{equation}

\subsection{Reformulation of the main results}

We can now equivalently reformulate our main results from \cref{sec:intro} in terms of $w(x) \defeq \abs{x}^{-\alpha} u(x)$. The following theorems are what we will actually prove. 
Since it is straightforward to deduce \cref{theorem ckn nondeg intro,corollary ckn nondeg intro,theorem ckn stability intro,thm:perturbative-symmetry-ckn} from \cref{thm:non-degeneracy,corollary nondeg hardy,theorem stability can,thm:perturbative-symmetry-hardy} by using \cref{hardy ckn trafo formula,hardy ckn integral identities}, we omit the proofs of the former.

\begin{theorem}[Non-degeneracy for positive solutions, Hardy formulation]
\label{thm:non-degeneracy}
      Assume \cref{parameters-hardy} and, additionally, $\alpha\ge 0$. 
    Let $W\in \dot H^s(\R^n) \setminus \{0\}$ be a non-negative solution of \cref{eq:hardy-eq}. 
    If $\varphi \in\dot H^s(\R^n)$ solves 
    \begin{equation}
    \label{lin-eq-thm}
         (-\Delta)^s \varphi + C(\alpha) \frac{\varphi}{\abs{x}^{2s}} = (p-1) \frac{W^{p-2}}{\abs{x}^{t p}} \varphi \qquad \text{ in } \R^n \setminus \{0\}, 
    \end{equation}
    then $\varphi$ is a scalar multiple of $\partial_\lambda|_{\lambda = 1} W_\lambda$.
\end{theorem}

We emphasize again that in Theorem \ref{thm:non-degeneracy} we only assume $W$ to solve \eqref{eq:hardy-eq}, not to minimize \eqref{eq:hardy}. If $W$ is a minimizer of \cref{eq:hardy}, then $W$ does not change sign because, otherwise, $\| |W| \|_{\dot H^s}$ would be strictly smaller than  $\| W \|_{\dot H^s}$. Hence, a suitable scalar multiple of $W$ satisfies the assumption of \cref{thm:non-degeneracy}.

\begin{corollary}[Non-degeneracy for minimizers, Hardy formulation]
\label{corollary nondeg hardy}
           Assume \cref{parameters-hardy} and, additionally, $\alpha\ge 0$. 
    Let $W\in \dot H^s(\R^n)$ be a minimizer of \cref{eq:hardy} normalized to satisfy \cref{eq:hardy-eq}.
    If $\varphi \in \dot H^s(\R^n)$ satisfies $\int_{\R^n} \frac{W^{p-1}}{\abs{x}^{t p}} \varphi = 0$ and equality holds in \cref{W Hessian}, then $\varphi$ is a scalar multiple of $\partial_\lambda|_{\lambda = 1} W_\lambda$. 
\end{corollary}

\begin{theorem}[Stability, Hardy formulation]
\label{theorem stability can}
    Assume \cref{parameters-hardy} and, additionally, $\alpha\ge 0$. 
    Let $W\in \dot H^s(\R^n)$ be a minimizer of \cref{eq:hardy}. 
    There exists $\kappa > 0$ such that for all $w \in \dot H^s(\R^n)$, it holds 
    \begin{equation}
        \label{inequality stability thm}
         \norm{w}_{\dot H^s(\R^n)}^2 + C(\alpha) \norm{w \abs{\emptyparam}^{-s}}_{L^2}^2  
    - \tilde\Lambda_{n,s,p,\alpha} 
    \norm*{w\abs{\emptyparam}^{-t}}_{L^p}^2 
    \geq 
    \kappa \inf_{c \in \R, \lambda > 0} \norm{w - c W_\lambda}_{\dot H^s(\R^n)}^2. 
    \end{equation}
\end{theorem} 

\begin{corollary}[\cref{eq:hardy} inequality with remainder]
    \label{corollary remainder Hardy}
    Assume \cref{parameters-hardy} and, additionally, $\alpha \geq 0$. There exists $\tilde c = \tilde c(n,s,\alpha,p) > 0$ such that, for all $\Omega \subset \R^n$, with $|\Omega| < \infty$, and for all $w \in \dot H^s(\R^n)$ supported in $\Omega$, 
    \begin{equation}
        \label{remainder term ineq}
        \norm{w}_{\dot H^s}^2 + C(\alpha)  \norm*{w |\emptyparam|^{-s}}_{L^2}^2 - \tilde \Lambda_{n,s,p,\alpha} \norm*{w |\emptyparam|^{-t}}_{L^p(\Omega) }^2 \geq \tilde c |\Omega|^{-\frac{n-2s-2\alpha}{n}} \norm{w}_{L^{\frac{n}{n-2s-\alpha}, \infty}}^2.  
    \end{equation} 
    In fact, $\tilde c = \frac{C_{n,s}}{2} c$, where $c$ is the constant from \cref{corollary remainder CKN}. 
\end{corollary}

\begin{theorem}[Symmetry, Hardy formulation] \label{thm:perturbative-symmetry-hardy}
    For every $n\ge 1$, $0<s<\min\{1,n/2\}$, and $-2s < \alpha_0 < 0$, there exists $\eps = \eps(n,s,\alpha_0) >0$ such that the following statement holds.

    If $\alpha\in(\alpha_0, 0)$ and $p\in(2, 2+\eps)$, then every minimizer of \cref{eq:hardy} is radial.
\end{theorem}


\section{Notation and preliminaries}
\label{sec:notation}

We write $B_r(x)$ for the open ball in Euclidean space with radius $r$ and center $x$; we abbreviate $B_r \defeq B_r(0)$. When no ambiguity is possible, we use the shortened notation $\norm{\emptyparam}_{L^p} = \norm{\emptyparam}_{L^p(\R^n)}$ for the $L^p$-norm of a function, and we omit the integration variable in integral expressions. Finally, for a set $M$ and functions $f,g : M \to \R_+$, we write $f(m) \lesssim g(m)$ if there exists a constant $C > 0$, independent of $m$, such that $f(m) \leq C g(m)$ for all $m \in M$ (and accordingly for $\gtrsim$). We write $f \sim g$ if both $f \lesssim g$ and $f \gtrsim g$ hold. 

\subsection{Notation for distributions}
\label{subsec:notation-distributions}
Let us denote by $C^\infty_c(\R^n)$ and $\mathcal D'(\R^n)$ respectively the space of compactly supported test functions and its dual, i.e., the space of distributions. Analogously, let $\mathcal S$ and $\mathcal S'$ denote the Schwartz space of rapidly decaying functions and its dual, i.e., the space of tempered distributions.

Let us introduce the notation for the \emph{Cauchy principal value} (see \cite[Chapter 3, Section 8]{MR4703940}).
Given $x_0\in\R^n$, a signed Radon measure $\mu$ in $\R^n\setminus\{x_0\}$ and a measurable function $\varphi:\R^n\to\R$, we define
\begin{equation*}
   \mathrm{P.V.}_{x_0}\int_{\R^n} \varphi\, \mathrm{d} \mu \defeq 
    \lim_{r\to 0}
    \int_{B_r(x_0)^\complement} \varphi\, \mathrm{d} \mu. 
\end{equation*}
Sometimes we drop the subscript of $\mathrm{P.V.}$ if it is clear from the context what is the point $x_0$ (usually it is the point where the function or measure we are integrating is singular).
Moreover, we denote by $ \mathrm{P.V.}_{x_0}(\mu)$ the distribution that acts as
\begin{equation*}
    \scalprod{ \mathrm{P.V.}_{x_0}(\mu), \varphi} \defeq  \mathrm{P.V.}_{x_0}\int_{\R^n} \varphi\, \mathrm{d} \mu,
\end{equation*}
for all $\varphi\in C^\infty_c(\R^n)$.

Let us recall the definition and main properties of the convolution in the framework of distributions (see, e.g., \cite[Chapter IV]{MR1996773}).
Given a distribution $u\in \mathcal D'(\R^n)$ and a smooth compactly supported function $\varphi\in C^\infty_c(\R^n)$, the convolution $u\ast \varphi$ is defined as the smooth function
\begin{equation}\label{eq:def_convolution}
    (u\ast\varphi)(x) \defeq \scalprod{u, \varphi(x-\emptyparam)}
\end{equation}
for all $x\in\R^n$.
For any function $\eta\in C^\infty_c(\R^n)$, it holds
\begin{equation}\label{eq:def_convolution_duality}
    \scalprod{u\ast\varphi, \eta} = \scalprod{u, \eta\ast\varphi(-\emptyparam)}.
\end{equation}
Observe that the latter identity could be taken as a definition of convolution by duality.
Finally, let us recall that $\partial_i(u\ast \varphi) = (\partial_i u)\ast\varphi = u\ast(\partial_i\varphi)$ for any $1\le i\le n$.

\subsection{The fractional Laplacian and its inverse}
\label{ssec:fraclaplace}
We present some basic notions in the theory of the fractional Laplacian and its inverse; see \cite{Landkof1972,Stinga2019} or \cite[Section 2]{Silvestre2007} for a thorough presentation of the subject.

For $0<s<1$, let $\mathcal S_s$ be the subspace of $C^\infty(\R^n)$-functions such that $\norm{(1+\abs{x}^{n+2s})D^k\varphi}_{L^\infty}<\infty$ for all $k\ge 0$. We endow such space with the topology induced by the family of seminorms $\norm{(1+\abs{x}^{n+2s})D^k\varphi}_{L^\infty}$. We denote by $\mathcal S_s'$ the dual of $\mathcal S_s$. Observe that $C^\infty_c(\R^n)\subseteq \mathcal S\subseteq \mathcal S_s$ and $\mathcal S'_s\subseteq \mathcal S'\subseteq \mathcal D'(\R^n)$. Let us remark that $\mathcal S_s$ is closed under differentiation. 

For a function $\varphi\in \mathcal S_s$, we define its fractional Laplacian through the integral formula
\begin{equation}\label{eq:laplacian-pointwise}
    (-\lapl)^s \varphi(x) = C_{n,s}\,  \mathrm{P.V.}_{x}\int_{\R^n} \frac{\varphi(x)-\varphi(y)}{\abs{x-y}^{n+2s}}\, \mathrm{d} y ,
\end{equation}
where the constant $C_{n,s} > 0$ is chosen so that the symbol of the operator $(-\lapl)^s$ is $\abs{\xi}^{2s}$ (see \cite[Theorem 1]{Stinga2019}). With this normalization, we have $(-\lapl)^s\circ (-\lapl)^t = (-\lapl)^{s+t}$. Observe also that $(-\lapl)^s$ is self-adjoint with respect to the $L^2(\R^n)$-scalar product.

The operator $(-\lapl)^s$ maps $\mathcal S$ into $\mathcal S_s$  (see \cite{Silvestre2007}); hence, since $(-\lapl)^s$ is self-adjoint, we can extend its definition to any distribution $u\in \mathcal S'_s$ via duality:
\begin{equation*}
    \scalprod{(-\lapl)^s u, \varphi} = \scalprod{u, (-\lapl)^s\varphi} \quad
    \text{for all $\varphi\in \mathcal S$.}
\end{equation*}
In particular, $(-\lapl)^s$ maps $\mathcal S'_s$ into $\mathcal S'$. Let us remark that the pointwise definition \cref{eq:laplacian-pointwise} holds for $\varphi\in C^{2s+\delta}_{\mathrm{loc}}(\R^n)$ {with $(1 + |\cdot|^{n+2s})^{-1} \varphi \in L^1(\R^n)$} ~\cite[Proposition 2.4]{Silvestre2007}.

Let us the define the inverse operator $(-\lapl)^{-s}$. For a function $\varphi\in\mathcal S$, its inverse fractional Laplacian is defined as
\begin{equation}\label{eq:representation_formula_inverse}
    (-\lapl)^{-s}\varphi(x) = C_{n,-s}\int_{\R^n} \frac{\varphi(y)}{\abs{x-y}^{n-2s}}\, \mathrm{d} y,
\end{equation}
where the constant $C_{n,-s}$ is chosen so that $(-\lapl)^s(-\lapl)^{-s}\varphi=\varphi$ (see \cite[Theorem 5 and 6]{Stinga2019}).

Let $\mathcal S_{-s}$ be defined analogously to $\mathcal S_s$ (substituting $s$ with $-s$) and let $\mathcal S_{-s}'$ be its dual. Observe that $C^\infty_c(\R^n)\subseteq \mathcal S\subseteq \mathcal S_{-s}$ and $\mathcal S'_{-s}\subseteq \mathcal S'\subseteq \mathcal D'(\R^n)$.

The inverse fractional Laplacian $(-\lapl)^{-s}$ maps $\mathcal S$ into $\mathcal S_{-s}$. And therefore, as we did for the fractional Laplacian, we can extend its definition by duality to $\mathcal S_{-s}'$. In particular, $(-\lapl)^{-s}$ maps $\mathcal S_{-s}$ into $\mathcal S'$.

\subsection{Regularity properties of \texorpdfstring{$(-\lapl)^s$}{the fractional Laplacian}}
Let us recall the definition of H\"older spaces.
Fix an open set $\Omega\subseteq\R^n$.
Given $\gamma > 0$, write it as $\gamma = [\gamma] + \{\gamma\}$, where $[\gamma]$ is an integer and $0 < \{\gamma\}\le 1$ (so, for example, $[1]=0$ and $\{1\}=1$).
For $\gamma\ge 0$, let $C^{\gamma}(\Omega)$ denote the space of functions $u:\Omega\to\R$ such that
\begin{equation*}
    \infty > [u]_{C^\gamma(\Omega)} 
    = [D^{[\gamma]}u]_{C^{\{\gamma\}}(\Omega)}
    \defeq
    \sup_{x,y\in\Omega} \frac{\abs{D^{[\gamma]}u(x) - D^{[\gamma]}u(y)}}{\abs{x-y}^{\{\gamma\}}}.
\end{equation*}
Notice that, with this notation, $C^1$ corresponds to Lipschitz functions. We endow $C^\gamma(\Omega)$ with the topology given by the seminorm $[\emptyparam]_{C^\gamma(\Omega)}$.

We define $C^\gamma_{\mathrm{loc}}(\Omega)$ as the space of functions that belong to $C^\gamma(\Omega')$ for all $\Omega'\Subset \Omega$, endowed with the family of seminorms $([\emptyparam]_{C^\gamma(\Omega')})_{\Omega'\Subset\Omega}$.

Let us begin by studying the behavior of $(-\lapl)^s u$ away from the support of $u$.

\begin{lemma}\label{lem:laplacian-where-zero}
    Let $n\ge 1$ be a positive integer, let $s\in (0,1)$, and let $\Omega\subseteq\R^n$ be an open set.
    
    If a distribution $u\in\mathcal S'_s$ is null on $\Omega$, then $(-\lapl)^s\varphi$ is smooth in $\Omega$. 
    Moreover, when restricted on distributions that are supported on $\Omega^\complement$, the operator $(-\lapl)^s$ is continuous from $\mathcal S_s'(\R^n)$ into $C^\infty(\Omega)$.
    
    Equivalently, if $(u_k)_{k\in\N}, u\subseteq\mathcal S_s'(\R^n)$ are supported on $\Omega^\complement$ and $u_k\to u$ in the $\mathcal S_s'(\R^n)$-topology, then $u_k\to u$ in the $C^\infty(\Omega)$-topology.
    Moreover, for any $u\in\mathcal S'_s(\R^n)$ supported on $\Omega^\complement$, for all $x\in \Omega$, we have\footnote{The scalar product makes sense because the result does not change if we replace $\abs{\emptyparam}^{-(n+2s)}$ with a function $\eta\in \mathcal S_s$ that coincides with $\abs{\emptyparam}^{-(n+2s)}$ outside of a small ball centered at the origin.}
    \begin{equation*}
        (-\lapl)^s u(x) = -C_{n,s}\scalprod{u, \abs{\emptyparam-x}^{-(n+2s)}}.
    \end{equation*}
\end{lemma}
\begin{proof}
    Fix $u\in\mathcal S_s(\R^n)$ supported outside of $\Omega$, let $\Omega'$ be a smaller open set $\Omega'\Subset\Omega$ and take $\varphi\in C^{\infty}_c(\Omega')$. 
    Observe that
    \begin{equation*}
        \scalprod{(-\lapl)^s u, \varphi}
        =
        \scalprod{u, (-\lapl)^s\varphi}. 
    \end{equation*}
    For $x\in \Omega^\complement$, by definition of the fractional Laplacian and exploiting that $\varphi(x)=0$, we have
    \begin{equation*}
        (-\lapl)^s\varphi(x) = 
        -C_{n,s}\varphi\ast \abs{\emptyparam}^{-(n+2s)}.
    \end{equation*}
    Observe that we can remove the singularities of $\abs{\emptyparam}^{-(n+2s)}$ because $\varphi$ is supported far from $x$. Hence, there exists a function $\eta\in\mathcal S_s$, depending only on $\Omega$ and $\Omega'$, such that
    \begin{equation*}
        (-\lapl)^s\varphi(x) = 
        -\varphi\ast \eta.
    \end{equation*}
    Hence, since $u$ is supported on the complement of $\Omega$, we have
    \begin{equation}\label{eq:local43}
        \scalprod{(-\lapl)^s u, \varphi} = -\scalprod{u, \varphi\ast \eta}.
    \end{equation}
    Thanks to this formula, for any $k\ge 0$, we have
    \begin{align*}
        \norm{D^k(-\lapl)^su}_{L^{\infty}(\Omega')}
        &=
        \sup_{\varphi\in C^\infty_c(\Omega'), \norm{\varphi}_{L^1}\le 1}
        \scalprod{D^k(-\lapl)^su, \varphi}
        =
        \sup_{\varphi\in C^\infty_c(\Omega'), \norm{\varphi}_{L^1} \le 1}
        \scalprod{(-\lapl)^su, D^k\varphi}
        \\
        &=
        \sup_{\varphi\in C^\infty_c(\Omega'), \norm{\varphi}_{L^1}\le 1}
        \scalprod{u, \varphi\ast D^k\eta}.
    \end{align*}
    Observe that the map $\varphi\to \varphi\ast D^k\eta$ is continuous from $L^1(\Omega')$ to $\mathcal S_s(\Omega^\complement)$ and thus the supremum appearing in the last chain of identities is finite. Hence, $(-\lapl)^s u$ is smooth in $\Omega$. This argument also proves the part of the statement about the continuity of $(-\lapl)^s$.

    It remains to compute the pointwise value of $(-\lapl)^su$ at an arbitrary $x\in\Omega'$. To this end, consider a sequence $(\varphi_k)_{k\in\N}\subseteq C^\infty_c(\Omega')$ so that $0\le \varphi_k$, $\varphi_k$ is supported in $B_{\frac1k}(x)$, $\int\varphi_k=1$. This is a sequence of smooth functions converging to $\delta_x$. We have
    \begin{equation*}
        (-\lapl)^su(x) =
        \lim_{k\to\infty} \scalprod{(-\lapl)^su, \varphi_k}
        =
        \lim_{k\to\infty} -\scalprod{u, \varphi_k\ast\eta}.
    \end{equation*}
    At this point, observe that $\varphi_k\ast\eta\to \eta(\emptyparam-x)$ in $\mathcal S_s$ and thus the last limit coincides with $-\scalprod{u, \eta(\emptyparam-x)}$ which is the desired result (because we can replace $\eta$ with $C_{n,s}\abs{\emptyparam}^{-(n+2s)}$ and the value does not change).
\end{proof}

The following result is a refinement of \cite[Theorem 2]{Stinga2019} and \cite[Proposition 2.7]{Silvestre2007}.

\begin{proposition}\label{prop:laplacian-holder-regularity}
    Let $n\ge 1$ be a positive integer, let $s\in (0,1)$, let $\Omega\subseteq\R^n$ be an open set, and let $\gamma>2s$.
    
    The operator $(-\lapl)^s$ is continuous from $\mathcal S'_s(\R^n)\cap C_{\mathrm{loc}}^{\gamma}(\Omega)$ into $C_{\mathrm{loc}}^{\gamma-2s}(\Omega)$, i.e., if $(u_k)_{k\in\N},u\subseteq \mathcal S'_s(\R^n)\cap C_{\mathrm{loc}}^{\gamma}(\Omega)$ satisfy $u_k\to u$ both in the $S'_s(\R^n)$-topology and in the $C_{\mathrm{loc}}^{\gamma}(\Omega)$-topology, then $((-\lapl)^s u_k), (-\lapl)^su \subseteq C_{\mathrm{loc}}^{\gamma-2s}$ and $(-\lapl)^s u_k\to (-\lapl)^su$ with respect to the $C_{\mathrm{loc}}^{\gamma-2s}(\Omega)$-topology.
\end{proposition}
\begin{proof}
    Fix a bounded open set $\Omega'\Subset\Omega$ and a smooth $\eta\in C^\infty_c(\R^n)$ so that $\eta=1$ on $\Omega'$ and $\eta=0$ on $\Omega^{\complement}$. Split $u_k$ as $u_k=u^{(1)}_k + u^{(2)}_k$, where $u^{(1)}_k=u_k\eta$ and $u^{(2)}_k=u_k(1-\eta)$. Observe that $u_k^{(1)}$ belongs to $C^\gamma(\R^n)$ and is supported in $\Omega$, while $u_k^{(2)}\in\mathcal S'_s(\R^n)$ and is supported in the complement of $\Omega'$. The result follows thanks to \cref{lem:laplacian-where-zero} and \cite[Proposition 2.7]{Silvestre2007}.
\end{proof}

\subsection{Regularity properties of \texorpdfstring{$(-\lapl)^{-s}$}{the inverse fractional Laplacian}}
Let us begin with a statement analogous to \cref{lem:laplacian-where-zero}.

\begin{lemma}\label{lem:inverse-laplacian-where-zero}
    Let $n\ge 1$ be a positive integer, let $s\in (0,1)$, and let $\Omega\subseteq\R^n$ be an open set. 
    
    If a distribution $u\in\mathcal S'_{-s}$ is null on $\Omega$, then $(-\lapl)^{-s}u$ is smooth in $\Omega$. 
\end{lemma}
\begin{proof}
    It can be shown repeating, almost verbatim, the proof of \cref{lem:laplacian-where-zero}.
\end{proof}

The following proposition is a summary of the \emph{local} elliptic regularity enjoyed by the fractional Laplacian.
\begin{proposition}\label{prop:fractional-elliptic-regularity}
    Let $n\ge 1$ be a positive integer, let $0<s<\min\{1, \frac n2\}$, and let  $\Omega\subseteq\R^n$ be an open set.

    Assume that $(-\lapl)^su=f$ with $u\in\mathcal S'_s$ and $f\in \mathcal S_{-s}'$.
    \begin{enumerate}
        \item If $f\in L^p(\Omega)$ with $1<p<\frac n{2s}$, then $u\in L_{\mathrm{loc}}^{\frac{np}{n-2sp}}(\Omega)$.
        \item If $f\in L^p(\Omega)$ with $p>\frac n{2s}$ (and $p<\frac n{2s-1}$ if $2s>1$), then $u\in C_{\mathrm{loc}}^{2s-\frac np}(\Omega)$.
        \item If $f\in C^\gamma(\Omega)$ with $\gamma\ge 0$, then $u\in C_{\mathrm{loc}}^{\gamma+2s}(\Omega)$.
    \end{enumerate}
\end{proposition}
\begin{proof}
    We show the proposition in two steps of increasing generality.
    The three parts of the statement are proven together.
    
    \textbf{Step 1.} \emph{Special case $\Omega=\R^n$ and $f\in L^1(\R^n)$}. 
    Under these stronger assumptions, we observe that $u=C_{n,-s}f\ast \abs{x}^{2s-n} + g$, where $g$ is an affine function (see \cite[Theorem 1.1, Corollary 1.4]{Fall2016}).
    
    Then the classical regularity theory concerning the convolution with the Riesz kernel yields the desired statements, and the conclusion are not even local (so we have $L$ instead of $L_{\mathrm{loc}}$ and $C$ instead of $C_{\mathrm{loc}}$). 
    For \textit{(1)} and \textit{(2)} of the statement, see \cite[Proposition 2.1]{MusinaNazarov2021} or directly \cite[Chapter V]{Stein1970}. 
    For \textit{(3)}, see \cite[Proposition 2.8]{Silvestre2007} (which can be iterated by differentiating the identity $(-\lapl)^su=f$). 

    \textbf{Step 2.} \emph{Full generality}.
    Let $\eta\in C^\infty_c(\R^n)$ be a cut-off function with $0\le \eta\le 1$ supported in $\Omega$ and such that $\eta=1$ in $\Omega'\Subset\Omega$. Let $u_1\defeq (-\lapl)^{-s}(\eta f)$ and $u_2=u-u_1$. We have $(-\lapl)^su_1=\eta f$ while $(-\lapl)^su_2=0$ in $\Omega'$. Applying \cref{lem:inverse-laplacian-where-zero}, we obtain that $u_2$ is smooth in $\Omega'$. On the other hand, \textbf{Step 1} tells us that $u_1$ has the desired regularity (or integrability) in $\R^n$. Hence, we obtain the desired control on $u=u_1+u_2$ in $\Omega'$. Since $\Omega'\Subset\Omega$ is arbitrary, this concludes the proof.
\end{proof}

\subsection{Fractional Sobolev Spaces}
Let us briefly introduce the fractional Sobolev space with exponent $2$. For an in-depth presentation of this topic, we refer the reader to \cite{DiNezzaPalatucciValdinoci2012}.

Let $\dot H^s(\R^n)$ be the closure of $C^\infty_c(\R^n)$ with respect to the norm
\begin{equation}
\label{Hs norm normalization}
    \norm{\varphi}_{\dot H^s}^2 
    \defeq 
    \frac{C_{n,s}}2 \int_{\R^n}\int_{\R^n} \frac{(\varphi(x)-\varphi(y))^2}{\abs{x-y}^{n+2s}}\, \mathrm{d} x\, \mathrm{d} y .
\end{equation}

Let us remark that, with our normalization of the fractional Laplacian and of the fractional Sobolev norm, it holds
\begin{equation*}
    \norm{\varphi}_{\dot H^s}^2
    = \scalprod{(-\lapl)^{s/2}\varphi, (-\lapl)^{s/2}\varphi}
    = \scalprod{(-\lapl)^s\varphi, \varphi}.
\end{equation*}
The fractional Sobolev inequality states that the space $\dot H^s(\R^n)$ embeds into $L^{2^*_s}(\R^n)$, where $2^*_s \defeq \frac{2n}{n-2s}$.

Let us conclude this subsection with a simple yet useful fractional integration by parts formula. Observe that in the classical case, the equivalent formula is
\begin{equation*}
    \norm{fg}_{\dot H^1}^2 - \int_{\R^n} f^2g(-\lapl)g = \int_{\R^n} \abs{\nabla f}^2g^2
\end{equation*}
and, in particular, the right-hand side is always non-negative. In the fractional setting, the right-hand side might be negative, but in the cases we care about it will turn out to be non-negative.

\begin{lemma}[Fractional integration by parts] \label{lem:product_rule_fractional_laplacian}
    Let $n\ge 1$ be a positive integer and let $s\in (0,1)$.
    For any $f, g\in C^\infty_c(\R^n)$, it holds
    \begin{equation*}
        \norm{fg}^2_{{\dot{H}^s}} - \int_{\R^n} f^2 g(-\lapl)^sg 
        =
        \frac{C_{n,s}}2 
        \int_{\R^n} \int_{\R^n} 
        \frac{g(x)g(y)(f(x)-f(y))^2}{\abs{x-y}^{n+2s}}
        \, \mathrm{d} y\, \mathrm{d} x
        .
    \end{equation*}
\end{lemma}
\begin{proof}
    We have
    \begin{align*}
        \norm{fg}^2_{{\dot{H}^s}} &-  \int_{\R^n} f^2 g(-\lapl)^sg  \\
        &=
        \int_{\R^n} f(x)g(x)(-\lapl)^s(fg)(x) - f^2(x)g(x)(-\lapl)^sg(x)\, \mathrm{d} x
        \\
        &=
        C_{n,s}\int_{\R^n}  \mathrm{P.V.}_{x}\int_{\R^n} 
        f(x)g(x)\frac{f(x)g(x)-f(y)g(y)}{\abs{x-y}^{n+2s}} 
        - f^2(x)g(x)\frac{g(x)-g(y)}{\abs{x-y}^{n+2s}}\, \mathrm{d} y\, \mathrm{d} x
        \\
        &=
        C_{n,s}\int_{\R^n}  \mathrm{P.V.}_{x} \int_{\R^n} 
        \frac{f(x)^2g(x)g(y)-f(x)f(y)g(x)g(y)}{\abs{x-y}^{n+2s}}\, \mathrm{d} y\, \mathrm{d} x. 
    \end{align*}
    Hence, swapping $x$ and $y$, we get
    \begin{align*}
        \norm{fg}^2_{{\dot{H}^s}}& - \int_{\R^n} f^2 g(-\lapl)^sg  \\
        &=
        \frac{C_{n,s}}2 
        \int_{\R^n}  \mathrm{P.V.}_{x}\int_{\R^n}  
        \frac{f(x)^2g(x)g(y) + f(y)^2g(y)g(x)-2f(x)f(y)g(x)g(y)}{\abs{x-y}^{n+2s}}
        \, \mathrm{d} y\, \mathrm{d} x \\
        &= 
        \frac{C_{n,s}}2 
        \int_{\R^n} \int_{\R^n} 
        \frac{g(x)g(y)(f(x)-f(y))^2}{\abs{x-y}^{n+2s}}
        \, \mathrm{d} y\, \mathrm{d} x ,
    \end{align*}
    which proves the lemma. 
\end{proof}

\subsection{Radial functions and distributions}
A function $U:\R^n\to\R\cup\{\pm\infty\}$ is radial if $U(x)=U(y)$ whenever $\abs{x}=\abs{y}$. It is equivalent to ask that $U=U\circ L$ for all linear isometries $L\in \mathcal O(\R^n)$.

Analogously, a distribution $U\in \mathcal D'(\R^n)$ is radial if $\scalprod{U, \varphi}=\scalprod{U, \varphi\circ L}$ for any $\varphi\in C^\infty_c(\R^n)$ and any linear isometry $L\in\mathcal O(\R^n)$. Observe that a radial function in $L^1_{\mathrm{loc}}(\R^n)$ is a radial distribution.
Let us remark that the definition of radial distribution (or function) can be extended also to any open domain $\Omega\subseteq\R^n$ that is rotationally invariant.
Observe that if $U\in \mathcal S_s'(\R^n)$ is a radial distribution, then also $(-\lapl)^s U$ is a radial distribution.

For a (smooth enough) radial function $U$, we may define its radial derivative as $U'(x)\defeq \partial_r U(x) = \nabla U(x)\cdot \frac{x}{\abs{x}}$. Observe that $U'$ is a radial function in $\R^n\setminus\{0\}$ (and it is not defined in the origin) and furthermore it holds $\nabla U(x) = U'(x)\frac{x}{\abs{x}}$.

Analogously, for a radial distribution $U\in \mathcal D'(\R^n)$, we define its radial derivative as $U'(x)\defeq \partial_r U(x) = D U(x)\cdot \frac{x}{\abs{x}}$ (we denote with $DU$ the distributional derivative of $U$). Since $\frac{x}{\abs{x}}$ is smooth away from the origin, then $U'$ is a radial distribution in $\R^n\setminus\{0\}$. The formula $D U(x) = U'(x)\frac{x}{\abs{x}}$ holds also in this setting and can be shown by approximation since we know that it holds when $U$ is smooth.
We will also use that, if $U_k\to U$ in the sense of distributions, then $U_k'\to U'$ in the sense of distributions in the domain $\R^n\setminus\{0\}$.

We say that a radial function $U$ is weakly decreasing if $U(x)\ge U(y)$ whenever $\abs{x}\le\abs{y}$. Analogously, a radial distribution is weakly decreasing if $U'\le 0$ in $\R^n\setminus\{0\}$. Let us remark that a weakly decreasing radial distribution is necessarily a function in $\R^n\setminus\{0\}$ (observe that $U'\le 0$, thus $U'$ is a measure and so is $DU$).

The following lemma provides a representation of the distributional derivative of a radial function that has a \emph{wild behavior} at the origin. Let us recall, following~\cite{AmbrosioFuscoPallara2000}, that $\mathrm{BV}_{\mathrm{loc}}(\Omega)$ denotes the functions of local bounded variation in $\Omega$, i.e., the family of functions in $L^1_{\mathrm{loc}}(\Omega)$ such that their differential is a vectorial Radon measure in $\Omega$.

\begin{lemma}\label{lem:der_radial_distr}
    For $n\ge 1$, let $u\in \mathrm{BV}_{\mathrm{loc}}(\R^n\setminus\{0\})\cap L^1_{\mathrm{loc}}(\R^n)$ be a radial function such that $\lim_{x\to 0} u(x)\abs{x}^n=0$.
    Then\footnote{Here, $Du$ represents the distributional derivative of $u$, while $\nabla u$ is the vector measure that coincides with the distributional derivative outside of the origin.} $Du = \mathrm{P.V.}_{0_{\R^n}}(\nabla u)$.
\end{lemma}
\begin{proof}
    Let $\eta\in C^\infty_c(B_1)$ be a radial function satisfying $0\le \eta\le 1$ and $\eta\equiv 1$ in $B_{\frac12}$. Let $\eta_\eps(x) \defeq \eta( \eps^{-1}x)$. For any $\varphi\in C^\infty_c(\R^n)$, we have
    \begin{equation*}
        \int_{B_\eps^\complement} \varphi \,\mathrm d\nabla u
        +
        \int_{B_\eps\setminus B_{\eps/2}} \varphi \, \mathrm d\nabla((1-\eta_\eps)u)
        =
        \scalprod{(1-\eta_\eps)Du, \varphi}
        \xrightarrow{\eps\to 0}
        \scalprod{Du, \varphi}.
    \end{equation*}
    Let $u_\eps\defeq (1-\eta_\eps)u$. The desired result would follow if we were able to prove that
    \begin{equation}\label{eq:local98}
        \int_{B_\eps\setminus B_{\eps/2}} \varphi\, \mathrm d \nabla u_\eps \xrightarrow{\eps\to 0} 0.
    \end{equation}
    We show that there exists a constant $C=C(n)$ such that, for any radial function $v\in \mathrm{BV}_{\mathrm{loc}}(\R^n)$, it holds
    \begin{equation}\label{eq:local320}
        \abs*{
        \int_{B_R\setminus B_r}\varphi\,\mathrm d\nabla v
        }
        \le C\norm{\nabla\varphi}_{L^\infty}
        \Big(
        \norm{v}_{L^1(B_R)} + 
        \norm{v\abs{x}^n}_{L^\infty(B_R)}
        \Big)
    \end{equation}
    for all $0<r<R$. Applying this with $v=u_\eps$, we obtain \cref{eq:local98} (because of the assumptions on $u$).

    To prove \cref{eq:local320} we can assume that $v$ is smooth and compactly supported (the inequality for more general $v$ can be recovered by approximation). Integrating by parts, we have
    \begin{equation*}
        \int_{B_R\setminus B_r}\varphi \nabla v
        =
        v(R)\int_{\partial B_R}(\varphi-\varphi(0)) \frac{x}{\abs{x}}\,\mathrm d\Haus^{n-1}
        -v(r)\int_{\partial B_r}(\varphi-\varphi(0))\frac{x}{\abs{x}}\,\mathrm d\Haus^{n-1}
        +
        \int_{B_R\setminus B_r} \nabla\varphi v
    \end{equation*}
    and \cref{eq:local320} follows (we used that $v$ is radial to replace $\varphi$ with $\varphi-\varphi(0)$ in the first two terms in the right-hand side).
\end{proof}

\begin{remark}
   The assumptions $u\in \mathrm{BV}_{\mathrm{loc}}(\R^n\setminus\{0\})$ and $\lim_{x\to 0} u(x)\abs{x}^n=0$ of \cref{lem:der_radial_distr} holds if $u\in L^1_{\mathrm{loc}}(\R^n)$ is radially decreasing. Indeed, for any $r>0$, $\int_{B_r\setminus B_{r/2}} u\gtrsim u(r)r^n$ and thus $u(r)r^n\to 0$ as $r\to 0$ by the (local) uniform integrability of $u$. 
\end{remark}

\section{Two maximum principles for the fractional Laplacian}
\label{sec:max}

The goal of this section is to prove two different maximum principles for the fractional Laplacian.
The two statements \cref{thm:global-max-principle,thm:strong-max-principle} offer two different perspectives. The first one is a global comparison principle on $\R^n$ that allows us to exploit the integral formula for the fractional Laplacian in situations where the function is not necessarily good enough to apply it. The second one is a local strong maximum principle that is valid for a large class of distributions.

\begin{theorem}[Global comparison principle]\label{thm:global-max-principle}
    Let $n\ge 1$ be a positive integer, let $0<s<\min\{1, \frac n2\}$, and let $u\in L^1_{\mathrm{loc}}(\R^n)$ be a function such that $\fint_{B_R}\abs{u}\to0$ as $R\to\infty$.

    If $(-\lapl)^s u \le \mu$ (in the distributional sense), where $\mu\in\meas(\R^n)$ is a non-negative measure, then
    \begin{equation} \label{eq:local57}
        u(x) \le C_{n, -s} \int_{\R^n} \frac{1}{\abs{x-y}^{n-2s}} \, \mathrm{d} \mu(y)
        \quad\text{for almost every $x\in\R^n$}
        .
    \end{equation}
\end{theorem}

\begin{theorem}[Local Strong Maximum Principle]\label{thm:strong-max-principle}
    Let $n\ge 1$ be a positive integer, let $0<s<\min\{1,\tfrac n2\}$, let $\Omega\subseteq\R^n$ be an open set, and let $u\in \mathcal S'_s(\R^n)$. 
    
    Fix $\bar x\in\Omega$ and assume that:
    \begin{enumerate}
        \item $u\ge 0$ in $\Omega$,
        \item $(-\lapl)^su\ge 0$ in $\Omega$,
        \item There is a function $\psi\in C^\infty(\R^n)$ satisfying $0\le \psi\le \abs{x-\bar x}^{-(n+2s)}$ and $\psi=\abs{x-\bar x}^{-(n+2s)}$ in $\Omega^\complement$ such that $\scalprod{u, \psi}>0$.\footnote{Observe that it makes sense to compute $\scalprod{u, \psi}$ as $u\in\mathcal S'_s(\R^n)$ and $\psi\in\mathcal S_s(\R^n)$.}
    \end{enumerate}
    Then $u$ coincides in $\Omega$ with a lower semicontinuous function that is strictly positive at $\bar x$.
\end{theorem}

Some remarks on the statements are due.

Under the same assumptions of \cref{thm:global-max-principle}, observe that if $(-\lapl)^su=\mu$ and $\mu$ is smooth and compactly supported, then \cref{eq:local57} is an identity (see \cite[Theorem 5]{Stinga2019}). The statement of \cref{thm:global-max-principle} is not an immediate consequence of the representation formula for the inverse fractional Laplacian \cref{eq:representation_formula_inverse} because we do not assume anything about the regularity of $u$ or $(-\lapl)^su$.

In the literature various strong maximum principles for the fractional Laplacian similar to \cref{thm:strong-max-principle} have appeared, e.g., \cite[Theorem 1]{ChenLiLi2017}, \cite[Corollary 4.2]{MusinaNazarov2019}, \cite[Proposition 3.1]{Dipierro2016}.
There are two differences between our statement and the ones appearing in the literature:
\begin{itemize}
    \item We do not assume symmetry of $u$. Indeed, we replace such an assumption with condition \textit{(3)} (which is implied by symmetry with respect to a hyperplane). 
    \item We do not assume $u$ to be a function, that is, we prove the statement for distributions with bare minimum regularity assumptions. As will be clear later on, this seemingly minor technical difference makes the proof substantially more delicate. We cannot avoid it because our main application of \cref{thm:strong-max-principle} involves the radial derivative of a minimizer of \cref{eq:hardy}, which may not be integrable at the origin.
\end{itemize}

Observe that conditions \textit{(1)} and \textit{(2)} in \cref{thm:strong-max-principle} are expected also for the classical Laplacian, while condition \textit{(3)} is necessary because of the non-locality of the operator.

\subsection{Supersolutions for the fractional Laplacian}
To prove \cref{thm:strong-max-principle}, we need to generalize the theory for locally integrable functions developed in \cite[Section 2.2]{Silvestre2007} to the case of distributions that are not necessarily locally integrable.

Let $\Phi = \Phi_{n,s} \defeq C_{n,-s} \abs{x}^{-(n-2s)}$ be the fundamental solution of $(-\lapl)^s$ in $\R^n$.
We construct a smooth function $\Gamma=\Gamma_{n,s}$ by appropriately smoothing the singularity of $\Phi$ so that $\Gamma$ is still a supersolution for $(-\lapl)^s$.
In \cite[Section 2.2]{Silvestre2007}, an analogous function $\Gamma$ is constructed but it is only $C^{1,1}$ instead of smooth. This lack of higher regularity would introduce a large number of issues in our argument.

\begin{proposition}\label{prop:Gamma_def}
    Let $n\ge 1$ be a positive integer and let $0<s < \min\{1, \tfrac n2\}$.
    There exists a decreasing radial function $\Gamma=\Gamma_{n,s}\in C^\infty(\R^n)$ such that:
    \begin{itemize}
        \item $\Gamma(x)>0$ for all $x\in\R^n$,
        \item $\Gamma'(x)<0$ for all $x\in\R^n\setminus\{0\}$,
        \item $\Gamma=\Phi_{n,s}$ in $B_1^\complement$,
        \item $(-\lapl)^s\Gamma$ is a strictly positive smooth function such that $\int_{\R^n}(-\lapl)^s\Gamma=1$,
        \item It holds $\Gamma_{\lambda_1}\ge \Gamma_{\lambda_2}$ for any $0<\lambda_1<\lambda_2$, where $\Gamma_\lambda(x)\defeq \Gamma(\frac x\lambda)\lambda^{-(n-2s)}$.
    \end{itemize}
\end{proposition}
\begin{proof}
    Let $\Gamma$ be the radial function generated by the profile function $\psi$ whose existence is guaranteed by \cref{lem:crucial_smoothing_argument} when $\varphi(\abs{x})=\Phi(x)$.
    By construction, $\Gamma$ is smooth, radial, strictly decreasing with $\Gamma'<0$, and it coincides with $\Phi$ in $B_1^\complement$.

    Let us show that, for any $x_0\in B_1\setminus\{0\}$, one can find $\tau>0$ such that $\Phi(\emptyparam+\tau x_0)-\Gamma$ has its global minimum at $x_0$. 
    We apply \textit{(2)} of \cref{lem:smoothing_argument_consequence} to the function $\psi$ (which is the radial profile of $\Gamma$). 
    Let $r_0=|x_0|$ and let $r_1$ be the value mentioned in \textit{(2)} of \cref{lem:smoothing_argument_consequence}, let $\tau\defeq\frac{r_1-r_0}{r_0}$. 
    Observe that, when restricted on $\{r\frac{x_0}{\abs{x_0}}:\, r>0\}$, we have that $\Phi(\emptyparam+\tau x_0)-\Gamma$ has a global minimum at $x_0$. 
    Given $x\in\R^n$, we have $\abs{x+\tau x_0} \le \abs{\frac{\abs{x}}{\abs{x_0}}x_0+\tau x_0 }$. Therefore,
    \begin{equation*}
        \Phi(x+\tau x_0)-\Gamma(x) \ge 
        \Phi\Big(\frac{\abs{x}}{\abs{x_0}}x_0+\tau x_0\Big) - \Gamma\Big(\frac{\abs{x}}{\abs{x_0}}x_0\Big)
    \end{equation*}
    and thus the desired statement follows.

    At this point, one can repeat verbatim the proof of \cite[Proposition 2.11]{Silvestre2007} to obtain that $(-\lapl)^s\Gamma$ is a strictly positive smooth function with integral equal to $1$.

    To show that $\Gamma_{\lambda_1}\ge \Gamma_{\lambda_2}$ when $0<\lambda_1<\lambda_2$, we have to use property \textit{(6)} of \cref{lem:crucial_smoothing_argument}. Observe that, since $\varphi(\abs{x})=\Phi(x)$, then $\frac{r\varphi'(r)}{\varphi(r)}=-(n-2s)$ for all $r>0$.
    By \textit{(6)} of \cref{lem:crucial_smoothing_argument}, in the interval $[\bar r, 1]$, we have $\frac{\psi'}{\psi}\ge \tfrac{\varphi'}{\varphi}$.
    Moreover, since $\psi''<0$ in $(0,\bar r]$ (and $\psi>0$, $\psi'<0$) we get that $\frac{r\psi'(r)}{\psi(r)}$ is decreasing on $(0,\bar r]$. 
    Joining these observations, we deduce that $\frac{\abs{x}\Gamma'(x)}{\Gamma(x)}\ge -(n-2s)$ on $\R^n\setminus\{0\}$. This is equivalent to $\frac{d}{d\lambda}\Gamma_\lambda \le 0$, which implies the desired statement.
\end{proof}

Let $\Gamma=\Gamma_{n,s}:\R^d\to(0,\infty)$ be a function satisfying the assumptions mentioned in \cref{prop:Gamma_def}.
Let $\gamma=\gamma_{n,s}\defeq (-\lapl)^s\Gamma$. For any $x\not\in B_1$, we have
\begin{equation}\label{eq:formula_gamma}
\begin{aligned}
    \gamma(x) 
    &= (-\lapl)^s\Gamma(x) = C_{n,s}\,\text{P.V.}_{x}\int_{\R^n}\frac{\Gamma(x)-\Gamma(y)}{\abs{x-y}^{n+2s}}\, \mathrm{d} y
    \\
    &= 
    C_{n,s}\,\text{P.V.}_x\int_{\R^n}\frac{\Phi(x)-\Phi(y)}{\abs{x-y}^{n+2s}}\, \mathrm{d} y
    + C_{n,s}\,\text{P.V.}_x \int_{\R^n}\frac{\Phi(y)-\Gamma(y)}{\abs{x-y}^{n+2s}}\, \mathrm{d} y
    \\
    &=
    C_{n,s}\,\int_{\R^n}\frac{\Phi(y)-\Gamma(y)}{\abs{x-y}^{n+2s}}\, \mathrm{d} y= C_{n,s}(\Phi-\Gamma)\ast\abs{\emptyparam}^{-(n+2s)}(x)
\end{aligned}
\end{equation}

Let us prove the following strengthening of \cite[Proposition 2.12]{Silvestre2007}.

\begin{lemma}\label{lem:gamma_Ss}
    Let $n\ge 1$ be a positive integer and let $0<s < \min\{1, \tfrac n2\}$.
    The function $\gamma_{n,s}$ belongs to $\mathcal S_s$.
\end{lemma}
\begin{proof}
    By \cref{prop:Gamma_def} we know that $\gamma$ is smooth. Thus, it is sufficient to verify that 
    \begin{equation*}
        \sup_{x\in B_2^\complement}\abs{D^k\gamma}(x)\abs{x}^{n+2s} < +\infty
    \end{equation*}
    for all $k\ge 1$. Thanks to \cref{eq:formula_gamma}, it would be sufficient to show
    \begin{equation*}
        \sup_{x\in B_2^\complement} (\Phi-\Gamma)\ast \abs{D^k\big(\abs{\emptyparam}^{-(n+2s)}\big)}(x)\abs{x}^{-{n+2s}} < \infty,
    \end{equation*}
    which holds true since $\Phi-\Gamma$ is an $L^1$-function supported in $B_1$ and $\abs{D^k\big(\abs{\emptyparam}^{-(n+2s)}\big)}\le C_k \abs{x}^{-{n+2s}}$ for all $x\in B_1^\complement$ (where the constant $C_k$ does not depend on $x$, but can depend on $k, n, s$).
\end{proof}

For $\lambda>1$, define $\Gamma_\lambda(x) \defeq \Gamma(x/\lambda)\lambda^{-(n-2s)}$ and $\gamma_\lambda\defeq \gamma(x/\lambda)\lambda^{-n}$. Observe that $(-\lapl)^s\Gamma_\lambda = \gamma_\lambda$.

Let us study the behavior, away from the origin, of $\gamma_\lambda\lambda^{-2s}$ as $\lambda\to0$. The following lemma does not have a counterpart in \cite{Silvestre2007}.
\begin{lemma}\label{lem:strange-distributional-convergence}
    Let $n\ge 1$ be a positive integer and let $0<s < \min\{1, \tfrac n2\}$.
    
    There is a constant $c=c(n,s)>0$\footnote{The value of $c$ depends also on the precise choice of $\Gamma$.} such that $c\gamma_\lambda\lambda^{-2s}\to \abs{x}^{-(n+2s)}$ with respect to the $\mathcal S_s$-convergence in $\R^n\setminus\{0\}$ as $\lambda\to0$.
    More precisely, for any $\eta\in C^\infty_c(\R^n)$ that is equal to $1$ in a neighborhood of the origin, we have $c\gamma_\lambda\lambda^{-2s}(1-\eta)\xrightarrow{\mathcal S_s} \abs{x}^{-(n+2s)}(1-\eta)$ as $\lambda\to 0$.
\end{lemma}
\begin{proof}
    Thanks to \cref{eq:formula_gamma}, for any $x\in B_\lambda^\complement$, we obtain
    \begin{equation*}
        \gamma_\lambda(x) = C_{n,s}\lambda^{2s}\int_{B_1}\frac{\Phi(y)-\Gamma(y)}{\abs{x-\lambda y}^{n+2s}}\, \mathrm{d} y.
    \end{equation*}
    Therefore, by choosing $c>0$ appropriately, there is an $L^1$-function $\rho$ supported in $B_1$  such that $\norm{\rho}_{L^1} = 1$ and
    \begin{equation*}
        c\gamma_\lambda(x)\lambda^{-2s} = 
        \abs{\emptyparam}^{-(n+2s)}\ast \rho_\lambda,
    \end{equation*}
    where $\rho_\lambda\defeq \lambda^{-n} \rho(\frac{\emptyparam}{\lambda})$. Hence, we need to show that $\abs{\emptyparam}^{-(n+2s)}\ast \rho_\lambda$ converges to $\abs{\emptyparam}^{-(n+2s)}$ as $\lambda\to 0$ away from the origin, with respect to the $\mathcal S_s$-topology. Since we are interested only in the convergence away from the origin, we can replace $\abs{\emptyparam}^{-(n+2s)}$ with a function $\varphi\in\mathcal S_s$ that coincides with it in $B_1^\complement$. The convergence $\varphi\ast \rho_\lambda\to \varphi$ as $\lambda\to 0$ with respect to $\mathcal S_s$-convergence holds as $\varphi\in\mathcal S_s$ and this concludes the proof.
\end{proof}

It is now time to establish a generalization of \cite[Proposition 2.13]{Silvestre2007}.

Let us recall that the convolution between a distribution and compactly supported smooth function is well defined~\cref{eq:def_convolution} and satisfies \cref{eq:def_convolution_duality}.

Given $\varphi\in\mathcal S_s$, the map $x\mapsto \varphi(x-\emptyparam)$ is continuous from $\R^n$ into $\mathcal S_s$. Thus, the formula \cref{eq:def_convolution} makes sense also when $u\in\mathcal S'_s$ and $\varphi\in\mathcal S_s$. One can check that the resulting convolution belongs to $C^\infty(\R^n)$.
Moreover, by a standard approximation of $\eta$ with linear combinations of delta-distributions, one can show the validity of \cref{eq:def_convolution_duality} when $u\in \mathcal S_s'$ and $\varphi\in\mathcal S_s$.\footnote{We provide a sketch of the proof of \cref{eq:def_convolution_duality} in this setting. Given $\ell>0$, let $P_\ell=\{\ell\big(z+[0,1)^d\big):z\in\Z^d\}$ be a partition of $\R^d$ into cubes with side $\ell$. Define $\eta_\ell\defeq \sum_{Q\in P_\ell} \delta_{\text{center}(Q)}\int_Q\eta$. Observe that $\eta_\ell\to \eta$ in the distributional sense as $\ell\to0$.
Since $\eta_\ell$ is a combination of delta-distributions, we have
\begin{equation*}
    \scalprod{u\ast\varphi, \eta_\ell}
    =
    \scalprod{u, \eta_\ell\ast \varphi(-\emptyparam)}.
\end{equation*}
To pass this identity to the limit (as $\ell\to 0$) and obtain \cref{eq:def_convolution_duality}, we employ the following two facts:
\begin{itemize}
    \item For any $w\in C^1_{\mathrm{loc}}(\R^n)$, we have $\scalprod{w, \eta_\ell}\to \scalprod{w, \eta}$ as $\ell\to 0$. 
    \item For any $\varphi\in \mathcal S_s$, $\varphi\ast\eta_\ell \to \varphi\ast\eta$ with respect to the $\mathcal S_s$-topology.
\end{itemize}
}

\begin{lemma}\label{lem:gamma_convolution}
    Let $n\ge 1$ be a positive integer and let $0<s < \min\{1, \tfrac n2\}$.

    For any $\varphi\in\mathcal C^\infty_c(\R^n)$, we have $\varphi\ast\gamma_\lambda \to \varphi$, as $\lambda\to 0$, with respect to the $\mathcal S_s$-topology.
    Moreover, for any $u\in\mathcal S'_s$, the convolution $u\ast\gamma_\lambda$ is a smooth function and converges to $u$, in the sense of distributions, as $\lambda\to0$.
\end{lemma}
\begin{proof}
    Let us begin by proving the first part of the statement.
    Without loss of generality, we may assume that $\varphi$ is supported in $B_1$.
    Arguing as in \cite[Proposition 2.13]{Silvestre2007}, one can show that $\varphi\ast \gamma_\lambda \to \varphi$ in $L^\infty(B_2)$ as $\lambda\to 0$.
    Moreover, for any $x\in B_2^\complement$, we have
    \begin{equation*}
        \abs{\varphi\ast\gamma_\lambda-\varphi}(x)
        = \abs*{\int_{x+B_1}\varphi(x-y)\gamma_\lambda(y)\, \mathrm{d} y}
        \lesssim \norm{\varphi}_{L^\infty}\cdot \norm{\gamma_\lambda}_{L^\infty(x+B_1)}
        \lesssim \norm{\varphi}_{L^{\infty}}\lambda^{2s}\abs{x}^{-(n+2s)},
    \end{equation*}
    where in the last step we have applied \cref{lem:strange-distributional-convergence}.
    Hence, since we have handled both the region $B_2$ and its complement, we deduce that 
    \begin{equation*}
        \norm{(\varphi\ast\gamma_\lambda-\varphi)(x)(1+\abs{x}^{n+2s})}_{L^\infty} \to 0
        \quad
        \text{as $\lambda\to 0$.}
    \end{equation*}
    By repeating the same argument for the derivatives of $\varphi$ (observing that $D^k(\varphi\ast\gamma_\lambda)=(D^k\varphi)\ast\gamma_\lambda$) we obtain the first part of the statement.

    Let us now focus our attention to the second part of the statement. Since $u\in\mathcal S'_s$ and $\gamma_\lambda\in\mathcal S_s$, we have already observed that the convolution $u\ast\gamma_\lambda$ is smooth.
    Moreover, the first part of the statement together with the formula \cref{eq:def_convolution_duality} imply that $u\ast \gamma_\lambda\to u$ in the distributional sense as $\lambda\to 0$.
\end{proof}

We are ready to state our generalization of \cite[Proposition 2.15]{Silvestre2007}.

\begin{proposition}[Fractional Mean Value Property] \label{prop:fractional_mean_value_property}
    Let $n\ge 1$ be a positive integer and let $0<s < \min\{1, \tfrac n2\}$.
    
    Let $u\in\mathcal S'_s$ be a distribution such that $(-\lapl)^s u\ge 0$ in an open set $\Omega\subseteq\R^n$.
    Then $u$ coincides (as a distribution) with a lower semicontinuous function such that, for any $x_0\in\Omega$,
    \begin{equation*}
        u(x_0) \ge \scalprod{u, \gamma_\lambda(x_0-\emptyparam)},
    \end{equation*}
    for all $\lambda < \mathrm{dist}(x_0, \partial\Omega)$. 
\end{proposition}
\begin{proof}
    Given $0<\lambda_1<\lambda_2$, thanks to \cref{prop:Gamma_def} we know that $\Gamma_{\lambda_1}- \Gamma_{\lambda_2}\ge 0$ is a smooth function supported in $B_{\lambda_2}$. Hence, in the open set $\Omega_{\lambda_2}\defeq\{x\in\Omega:\, \mathrm{dist}(x, \partial\Omega)>\lambda_2\}$, we have
    \begin{equation*}
        0 \le (-\lapl)^su \ast \big(\Gamma_{\lambda_1}-\Gamma_{\lambda_2}\big)
        =
        u \ast (-\lapl)^s\big(\Gamma_{\lambda_1}-\Gamma_{\lambda_2}\big)
        =
        u\ast \gamma_{\lambda_1} - u\ast\gamma_{\lambda_2}.
    \end{equation*}
    Therefore we obtain $u\ast\gamma_{\lambda_2} \le u\ast\gamma_{\lambda_1}$ in $\Omega_{\lambda_2}$. Let $\tilde u = \sup_{\lambda>0} u\ast\gamma_\lambda$; we have shown that $u\ast\gamma_\lambda\nearrow \tilde u$ in any $\Omega'\Subset\Omega$ as $\lambda\to 0$. By dominated convergence on the negative part and monotone convergence on the positive part, we deduce that the convergence holds also in the distributional sense in $\Omega$. 
    Therefore, thanks to \cref{lem:gamma_convolution}, we deduce that $u=\tilde u$ in $\Omega$; in particular $u$ is lower semicontinuous as it is an increasing limit of continuous functions.
    The inequality $u(x_0)\ge \scalprod{u, \gamma_\lambda(x_0-\emptyparam)}$ follows from $\tilde u(x_0) \ge (u\ast \gamma_\lambda)(x_0)$.
\end{proof}

\subsection{Proofs of the comparison principles}
We can now show the global comparison principle. 
\begin{proof}[Proof of \cref{thm:global-max-principle}]
    The strategy of the proof is to test the inequality $(-\lapl)^s u \le \mu$ against the fundamental solution of $(-\lapl)^s$, but such a function is not an admissible test function.
    To overcome this difficulty we argue as in the proof of \cite[Proposition 2.15]{Silvestre2007}. 

    Take $0<r<R<\infty$ and observe that $\Gamma_r-\Gamma_R$ is a non-negative smooth function supported in $B_R$.
    By testing $(-\lapl)^s u\le \mu$ against $\Gamma_r-\Gamma_R$ we obtain
    \begin{equation}\label{eq:local52}
        \scalprod{u, \gamma_r - \gamma_R} \le  \scalprod{\mu, \Gamma_r-\Gamma_R} \le \scalprod{\mu, \Gamma_r}.
    \end{equation}
    
    Let us show that $\scalprod{u, \gamma_R}\to0$ as $R\to\infty$.
    We have 
    \begin{equation*}
        \scalprod{u, \gamma_R} = \int_{\R^n} u(Rx)\gamma(x) \, \mathrm{d} x.
    \end{equation*}
    Thanks to \cref{lem:gamma_Ss} we know the decay of $\gamma$, and therefore we have
    \begin{equation*}
        \abs*{\int_{\R^n} u(Rx)\gamma(x) \, \mathrm{d} x}
        \lesssim
        \fint_{B_R} \abs{u} + R^{2s}\int_{B_R^{\complement}} \frac{\abs{u(x)}}{\abs{x}^{n+2s}}\, \mathrm{d} x
    \end{equation*}
    and both quantities on the right-hand side go to $0$ as $R\to\infty$ because of the decay assumption on $u$.\footnote{To show the decay of the second term in the right-hand side, we employ the following dyadic annuli decomposition:
    \begin{align*}
        \int_{B_R^{\complement}} \frac{\abs{u(x)}}{\abs{x}^{n+2s}}\, \mathrm{d} x
        &=
        \sum_{k\ge 1}
        \int_{B_{2^kR}\setminus B_{2^{k-1}R}} \frac{\abs{u(x)}}{\abs{x}^{n+2s}}\, \mathrm{d} x
        \le
        \sum_{k\ge 1}\int_{B_{2^kR}\setminus B_{2^{k-1}R}} \frac{\abs{u}}{\abs{2^{k-1}R}^{n+2s}}
        \\
        &\le
        \sum_{k\ge 1} \frac{1}{{\abs{2^{k-1}R}^{n+2s}}}\int_{B_{2^kR}} \abs{u}
        \lesssim
        \sum_{k\ge 1} \frac{1}{{\abs{2^{k-1}R}^{2s}}}\fint_{B_{2^kR}} \abs{u}
        \lesssim 
        R^{-2s} \sup_{R'>R} \fint_{B_{R'}}\abs{u}.
    \end{align*}
    }

    Hence, by letting $R\to\infty$, \cref{eq:local52} implies
    \begin{equation*}
        \scalprod{u, \gamma_r} \le \scalprod{\mu, \Gamma_r}.
    \end{equation*}
    By repeating the same argument adding a spatial shift, we can show that, for all $x\in\R^n$,
    \begin{equation*}
        u\ast \gamma_r(x) \le \int_{\R^n} \Gamma_r(y-x)\, \mathrm{d} \mu(y).
    \end{equation*}
    Thanks to \cref{lem:gamma_convolution}, by letting $r\to 0$ in the previous inequality, we obtain
    \begin{equation*}
        u(x) \le \int_{\R^n} \Phi(y-x)\, \mathrm{d} \mu(y)
    \end{equation*}
    for almost every $x\in\R^n$. 
\end{proof}

Now we prove the strong maximum principle.
\begin{proof}[Proof of \cref{thm:strong-max-principle}]
    Without loss of generality we may assume $\bar x = 0_{\R^n}$.
    By \cref{prop:fractional_mean_value_property}, we have that $u$ is a lower semicontinuous function in $\Omega$ and
    \begin{equation}
        u(0_{\R^n}) \ge \scalprod{u, \gamma_\lambda}
    \end{equation}
    for any $\lambda>0$ sufficiently small. We will prove that the right-hand side is strictly positive and so $u(0_{\R^n})>0$ as desired.

    Let $\eta\in C^\infty_c(\Omega)$ be a function such that $0\le \eta\le 1$ and $\eta=1$ in a neighborhood of $0_{\R^n}$. 
    Then $\scalprod{u, \gamma_\lambda} = 
        \scalprod{u, \gamma_\lambda\eta} + \scalprod{u, \gamma_\lambda(1-\eta)}$.
    The term $\scalprod{u, \gamma_\lambda\eta}$ is non-negative because $u\ge 0$ in $\Omega$ and $\gamma_\lambda\eta$ is non-negative and supported in $\Omega$. Hence, it is sufficient to prove that $\scalprod{u, \gamma_\lambda(1-\eta)}>0$. 
    We have (here $c$ is the constant appearing in \cref{lem:strange-distributional-convergence})
    \begin{align*}
        c\lambda^{-2s}\scalprod{u, \gamma_\lambda(1-\eta)} 
        &=
        \scalprod{u, \big(c\gamma_\lambda\lambda^{-2s} - \abs{x}^{-(n+2s)}\big)(1-\eta)}
        \\&\quad +
        \scalprod{u, (\abs{x}^{-(n+2s)}-\psi)(1-\eta)}
        \\&\quad+
        \scalprod{u, \psi(1-\eta)}
    \end{align*}
    Observe that $\eta$ can be chosen so that the last term is strictly positive, because $\scalprod{u,\psi}>0$ by assumption. 
    The second term is non-negative because $u\ge 0$ in $\Omega$ and $(\abs{x}^{-(n+2s)}-\psi)(1-\eta)$ is non-negative and supported in $\Omega$.
    The first term goes to $0$ as $\lambda\to0$ as a consequence of \cref{lem:strange-distributional-convergence} because $u\in \mathcal S'_s$.
    Hence, by appropriately choosing $\eta$ and $\lambda>0$, we get the inequality $\scalprod{u, \gamma_\lambda(1-\eta)}>0$ which concludes the proof.
\end{proof}


\section{Radially decreasing implies negative derivative}
\label{sec:radial-negative}

We show the following general result, which we will then apply to the minimizers of \cref{eq:hardy} when $\alpha\ge 0$.

\begin{theorem}\label{thm:radial-decreasing-negative-derivative}
    Let $n\ge 1$ be a positive integer, let $0<s<\min\{1, \tfrac n2\}$, and let $U\in L^1_{\mathrm{loc}}(\R^n)$ be a radially weakly decreasing non-negative function such that $(-\lapl)^sU$ is radially weakly decreasing.
    Then either $U$ is constant or $U'$ is upper semicontinuous and \emph{strictly} negative in $\R^n\setminus\{0\}$.
\end{theorem}

The proof is short but technically involved (it will require both \cref{lem:der_radial_distr} and \cref{thm:strong-max-principle}).
On the other hand, the idea is rather simple: applying a fractional maximum principle to a directional derivative of $U$. The technical difficulties arise because the derivative of $U$ is not an integrable function around the origin.

\begin{proof}[Proof of \cref{thm:radial-decreasing-negative-derivative}]
    Since $U$ belongs to $\mathcal S'_s$, also $\partial_1 U$ belongs to $\mathcal S'_s$.
    Therefore, in the distributional sense, it holds that $(-\lapl)^s (\partial_1 U) = \partial_1 (-\lapl)^s U$. 
    
    We would like to apply \cref{thm:strong-max-principle} to the function $-\partial_1 U$ at the point $\bar x=e_1$.
    Assumptions \textit{(1)} and \textit{(2)} are verified because $U$ and $(-\lapl)^sU$ are radially weakly decreasing. Let us check that assumption \textit{(3)} also holds.
    Thanks to \cref{lem:der_radial_distr}, we have
    \begin{equation}\label{eq:local187}
        \scalprod{-\partial_1 U, \psi} = -\lim_{r\to 0}\int_{B_r(0)^\complement} U'\psi \frac{x_1}{\abs{x}},
    \end{equation}
    for any function $\psi$ satisfying the constraints described in \textit{(2)}. Since $U'$ is radial, the integral on the right-hand side can be written as
    \begin{equation}\label{eq:local188}
        \int_{B_r(0)^\complement\cap\{x_1>0\}} [\psi(x_1, x_2, \dots, x_n) - \psi(-x_1, x_2, \dots, x_n)]U'(x)\frac{x_1}{\abs{x}}\, \mathrm{d} x.
    \end{equation}
    Observe that in the region of integration, $U'(x)\frac{x_1}{\abs{x}}\le 0$ and it is not $0$ everywhere because we assume that $U$ is not constant.
    Let $\psi \in C^\infty(\R^n)$ be a smoothed-out version of $\min\{\abs{x-e_1}^{-(n+2s)}, L\}$ for some $L>0$. By choosing $L$ large enough, we have $\psi(x_1, x_2, \dots, x_n) - \psi(-x_1, x_2, \dots, x_n)>0$ for all $x$ with $x_1>0$. Thus, combining \cref{eq:local187,eq:local188}, we obtain that \textit{(3)} holds.

    Now, by \cref{thm:strong-max-principle}, we know that $U'(1)=\partial_1 U(e_1)<0$ (and $\partial_1 U$ is upper semicontinuous around $e_1$). By an analogous argument, we deduce that the radial derivative of $U$ is upper semicontinuous and strictly negative at all points different from the origin.    
\end{proof}


\section{A general Hardy-type inequality for functions orthogonal to radial ones}
\label{sec:general-hardy}

In this section, we establish the following Hardy-type inequality for functions that are orthogonal to the family of radial functions.

\begin{theorem}[General Hardy-type inequality]\label{thm:general-hardy}
    Let $n\ge 1$ be a positive integer, let $0<s<\min\{1, \tfrac n2\}$, and
    let $U\in L^1_{\mathrm{loc}}(\R^n)\cap C_{\mathrm{loc}}^{1+2s}(\R^n\setminus\{0\})$ be a radial non-constant non-negative weakly decreasing function such that also $(-\lapl)^sU$ is weakly decreasing.
    
    Then $\rho_U\defeq \frac{((-\lapl)^sU)'}{U'}$ is a non-negative continuous function on $\R^n\setminus\{0\}$ and for any $\varphi\in \dot H^s(\R^n)$ such that $\int_{B_R}\varphi=0$ for all $R>0$, we have
    \begin{equation}\label{eq:local92}
        \norm{\varphi}_{{\dot{H}^s}}^2 \ge
        \int_{\R^n} \varphi^2 \rho_U.
    \end{equation}
\end{theorem}

Notice that $\rho_U$ is well-defined because $U' < 0$ as a consequence of \cref{thm:radial-decreasing-negative-derivative}.

\begin{remark}
\label{remark s=1}
While it is likely to exist, we could not find the statement of \cref{thm:general-hardy} for the classical Laplacian ($s=1$) in the literature.
Let us outline an alternative proof of \cref{thm:general-hardy} in the case $s=1$ (assuming all the necessary regularity and integrability conditions are satisfied).

Let $V>0$. For any $\varphi$,  integration by parts together with the product rule gives 
    \begin{equation}
        \label{s=1 ibp}
        \int_{\R^n} |\nabla \varphi|^2 - \int_{\R^n} \varphi^2 \frac{(-\Delta)V}{V} = \int_{\R^n} \left|\nabla \varphi - \frac{\varphi}{V} \nabla V\right|^2 = \int_{\R^n} |\nabla w|^2 V^2 ,
    \end{equation}
    where $w \defeq \frac{\varphi}{V}$. (This computation is, e.g., carried out in Lemma 1.5 of the lecture notes \cite{Frank2011}.)

Now suppose that additionally $V$ is radial and $\int_{B_R} \varphi = 0$ for all $R >0$. Then also $\int_{B_R} w = 0$ and hence $\int_{\mathbb S^{n-1}} |\nabla_\theta w(R \theta)|^2 \, \mathrm d \theta \geq (n-1)  \int_{\mathbb S^{n-1}} w(R \theta) \, \mathrm d \theta$ for all $R > 0$. 

By passing to polar coordinates (see, e.g., \cite[Lemma 2.4]{Ekholm2006} for a similar computation), this implies that 
\begin{equation}
    \label{ekholm-frank}
    \int_{\R^n} |\nabla w|^2 V^2 \geq (n-1) \int_{\R^n} \frac{w^2}{r^2}V^2 = (n-1) \int_{\R^n} \frac{\varphi^2}{|x|^2}. 
\end{equation}
Now let $U$ be radial such that $U' < 0$. Taking $V = -U'$, and observing that $(-\Delta)(U') = (-\Delta U)' - \frac{n-1}{r^2} U'$,  \eqref{s=1 ibp} and \eqref{ekholm-frank} give
\begin{equation}
    \label{s=1 hardy}
    \int_{R^n} |\nabla \varphi|^2 \geq \int_{\R^n} \varphi^2 \frac{(-\Delta U)'}{U'}, 
\end{equation} 
which is the counterpart of \cref{eq:local92} for $s = 1$. 

We note that, as a byproduct of \cref{s=1 ibp}, the related Hardy-type inequality 
\[ \int_{\R^n} |\nabla \varphi|^2 \geq \int_{\R^n} \varphi^2 \frac{(-\Delta)V}{V} \]
 actually holds for \emph{any} $\varphi$ (not necessarily orthogonal to radial functions) and \emph{any} $V$ (not necessarily radial). 
\end{remark} 

\begin{remark}\label{rem:orthogonality-necessary}
    The assumption of orthogonality to radial functions cannot be dropped. To see why, let $U\in \dot H^s(\R^n)$ be the minimizer of the fractional Sobolev inequality $\dot H^s\hookrightarrow L^{2^*_s}$ (for $s<\frac12$ if $n=1$). Then, $U$ is a positive radially decreasing function that, up to normalization, satisfies $(-\lapl)^s U = U^p$ for $p=2^*_s-1$.

    Since $\rho_U = \frac{((-\lapl)^sU)'}{U'} = p U^{p-1}$, we then have 
    \begin{equation*}
        \norm{U}_{\dot H^s}^2
        =
        \int_{\R^n} U(-\lapl)^sU = 
        \int_{\R^n} U^{p+1} = \frac1p\int_{\R^n} U^2 \rho_U.
    \end{equation*}
    Since $p>1$, the latter identity would be in contradiction with \cref{eq:local92} if we were allowed to choose $\varphi=U$ (but we cannot because $U$ is not orthogonal to radial functions being radial itself).
\end{remark}

In the proof we employ a sequence of approximation procedures that might obscure the main idea; the readers interested only in the crucial non-technical ideas should focus on \textbf{Steps 2 and 3} of the proof of \cref{lem:far-from-origin-general-hardy}.

We need the following intermediate result.
\begin{lemma}\label{lem:far-from-origin-general-hardy}
    Let $n\ge 1$ be a positive integer, let $0<s<\min\{1, \tfrac n2\}$, and 
    let $U\in L^1_{\mathrm{loc}}(\R^n)\cap C_{\mathrm{loc}}^{1+2s}(\R^n\setminus\{0\})$ be a radial non-negative function such that $U'<0$ in $\R^n\setminus\{0\}$. 

    Then $\rho_U\defeq \frac{((-\lapl)^sU)'}{U'}$ is a continuous function in $\R^n\setminus\{0\}$ and for any $\varphi\in C^{\infty}_c(\R^n\setminus\{0\})$ such that $\int_{B_R}\varphi=0$ for all $R>0$, we have
    \begin{equation*}
        \norm{\varphi}_{{\dot{H}^s}}^2 \ge
        \int_{\R^n} \varphi^2\rho_U .
    \end{equation*}
\end{lemma}
\begin{proof}
  \textbf{Step 1.} \emph{Regularization of $U$}.
    By convolution and multiplication with a compactly supported function, we can find a sequence $(U_k)_{k\in\N}\subseteq C^\infty_c(\R^n)$ such that 
    \begin{enumerate}
        \item For all $k\ge 1$, $U_k$ is a non-negative radial function such that $U_k'\le 0$.
        \item $U_k\to U$ in the $C_{\mathrm{loc}}^{1+2s}(\R^n\setminus\{0\})$-topology.
        \item $U_k\to U$ as distributions.
    \end{enumerate}
    In particular, the distribution $\frac{((-\lapl)^sU_k)'}{U_k'}$ converges to $\frac{((-\lapl)^sU)'}{U'}$ uniformly on compact sets of $\R^n\setminus\{0\}$ (recall \cref{prop:laplacian-holder-regularity}).
    
    Thanks to the above-mentioned approximation procedure, we can assume that $U$ is compactly supported and the assumption $U'<0$ is valid on the support of $\varphi$. The additional regularity of $U$ allows us to perform all the computations without worrying about technical details.

    \textbf{Step 2.} \emph{Spectral decomposition of $\varphi$.}
    Let us decompose $\varphi(x) = \sum_{k\ge 1} \varphi_k(\abs{x})A_k(x)$, where $A_k:\S^{n-1}\to\R$ is an eigenfunction of the spherical Laplacian extended $0$-homogeneously to $\R^n$, i.e., $-\lapl_{\S^{n-1}} A_k = \lambda_k A_k$ with $0=\lambda_0 < \lambda_1 \le \lambda_2 \le\cdots$.
    
    We can drop the first term, which would be the radial term $\varphi_0(\abs{x})$ (since $A_0$ is constant), because of the assumption of null average on balls. Let us normalize $A_k$ so that $\fint_{\S^{n-1}} A_k^2 = 1$. With this normalization, for any radial function $\psi\in L^2(\R^n)$, we have $\norm{\psi A_k}_{L^2} = \norm{\psi}_{L^2}$ and also
    \begin{equation}\label{eq:local91}
        \int_{\R^n}\psi\varphi^2 = \sum_{k\ge 1} \int_{\R^n} \psi\varphi_k^2.
    \end{equation}

    For any radial function $\psi:\R^n\to\R$, we have the identity
    \begin{equation*}
        (-\lapl)\big(\psi A_k(x)\big) = \big((-\lapl + \lambda_k \abs{x}^{-2})\psi\big)A_k(x),
    \end{equation*}
    thus we deduce
    \begin{equation}
    \label{eq:formula-for-fractional-laplacian}
        (-\lapl)^s(\psi A_k) = ((-\lapl+\lambda_k \abs{x}^{-2})^s\psi\big)A_k.
    \end{equation}
    In particular, we have
    \begin{equation*}
        \norm{\varphi}_{{\dot{H}^s}}^2 = 
        \sum_{k\ge 1} \norm{\varphi_k A_k}_{{\dot{H}^s}}^2.
    \end{equation*}
    Recalling \cref{eq:formula-for-fractional-laplacian}, since $\lambda_k\ge \lambda_1$, Loewner's theorem (see \cite{Simon2019}) guarantees that $\norm{\varphi_k A_k}_{{\dot{H}^s}} \ge \norm{\varphi_k A_1}_{{\dot{H}^s}}$ for all $k\ge 2$ (observe that $t\mapsto t^s$ satisfies the hypotheses of Loewner's theorem because $0<s<1$) . Hence, we obtain
    \begin{equation}\label{eq:local92uff}
        \norm{\varphi}_{{\dot{H}^s}}^2 \ge 
        \sum_{k\ge 1} \norm{\varphi_k A_1}_{{\dot{H}^s}}^2.
    \end{equation}

    \textbf{Step 3.} \emph{Main estimate.}
    We are going to show that, for any radial function $\psi\in C^\infty_c(\R^n\setminus\{0\})$,
    \begin{equation}\label{eq:magic-strong-non-degeneracy}
        \norm{\psi A_1}_{{\dot{H}^s}}^2 \ge
        \int_{\R^n}\frac{((-\lapl)^sU)'}{U'}\psi^2.
    \end{equation}
    Notice that joining \cref{eq:local91,eq:local92uff,eq:magic-strong-non-degeneracy} we obtain the statement of the lemma. 

    We observe that (up to rotation) $A_1(x) = \sqrt{n} \frac{x_1}{\abs{x}}$.
    Let $V\defeq \partial_1 U = U'(\abs{x})\frac{x_1}{\abs{x}}$. Let $\eta:\R^n\to\R$ be the radial function $\eta \defeq \frac{\psi}{U'}$. (The idea to consider $\eta$ is inspired by a similar argument for $\alpha = 0$ in \cite{MusinaNazarov2021}. This is similar in spirit with the argument discussed in \cref{remark s=1}.)  Notice that $\eta$ is well-defined by \cref{thm:radial-decreasing-negative-derivative}. 
    Applying \cref{lem:product_rule_fractional_laplacian}, we obtain
    \begin{equation}\label{eq:local88}
        \norm{\psi A_1}_{{\dot{H}^s}}^2 = n\norm{\eta V}_{{\dot{H}^s}}^2
        = n \int_{\R^n} \eta^2 V(-\lapl)^s V
        +
        \frac n2 C_{n,s} \int_{\R^n}\int_{\R^n} \frac{V(x)V(y)(\eta(x)-\eta(y))^2}{\abs{x-y}^{n+2s}}\, \mathrm{d} x\, \mathrm{d} y.
    \end{equation}
    Let $x^-\defeq (-x_1, x_2,\dots, x_n)$. Since $\eta$ is radial and $V(x^-)=-V(x)$,  the second term at the right-hand side of the last inequality can be rewritten as
    \begin{align*}
        \int_{\R^n}\int_{\R^n} &\frac{V(x)V(y)(\eta(x)-\eta(y))^2}{\abs{x-y}^{n+2s}}\, \mathrm{d} x\, \mathrm{d} y \\
        &=
        2
        \int_{\{x_1>0\}}
        \int_{\{y_1>0\}}
        V(x)V(y)(\eta(x)-\eta(y))^2\Big(
        \frac{1}{\abs{x-y}^{n+2s}}
        -
        \frac{1}{\abs{x^--y}^{n+2s}}
        \Big)\, \mathrm{d} x\, \mathrm{d} y,
    \end{align*}
    which is non-negative because $V\le 0$ in $\{x_1>0\}$ and $\abs{x-y}<\abs{x^{-}-y}$ when $x_1>0$ and $y_1>0$.
    Therefore, using the commutation of $(-\lapl)^s$ with $\partial_1$, \cref{eq:local88} implies
    \begin{equation*}
        \norm{\psi A_1}_{{\dot{H}^s}}^2
        \ge
        n \int_{\R^n} \eta^2 V(-\lapl)^s V
        =
        n \int_{\R^n} \frac{\psi^2}{U'^2} U'\frac{x_1}{\abs{x}} ((-\lapl)^s U)'\frac{x_1}{\abs{x}}
        =
        \int_{\R^n} \frac{((-\lapl)^sU)'}{U'}\psi^2
        ,
    \end{equation*}
    that is exactly \cref{eq:magic-strong-non-degeneracy}.
\end{proof}

We are now ready to prove \cref{thm:general-hardy}. The proof is fundamentally an additional layer of approximation over \cref{lem:far-from-origin-general-hardy}.

\begin{proof}[Proof of \cref{thm:general-hardy}]
    By \cref{prop:laplacian-holder-regularity}, we have that $U$ and $(-\lapl)^s U$ are both differentiable functions outside the origin. Hence, their radial derivatives are continuous functions in $\R^n\setminus\{0\}$. Thanks to \cref{thm:radial-decreasing-negative-derivative}, we also know that $U'<0$ and therefore $\rho_U$ is a non-negative smooth function in $\R^n\setminus\{0\}$.

    The validity of \cref{eq:local92} was established in \cref{lem:far-from-origin-general-hardy} under the additional assumption $\varphi\in C^\infty_c(\R^n\setminus\{0\}$). To extend it to any $\varphi\in \dot H^s(\R^n)$ we adopt an approximation procedure.
    Given $\varphi\in \dot H^s(\R^n)$, since $2s<n$, we can find a sequence $(\varphi_k)_{k\in\N}\subseteq C^\infty_c(\R^n\setminus\{0\})$ such that $\varphi_k\to \varphi$ in $\dot H^s(\R^n)$ (the $\dot H^s$-capacity of a point is $0$) and furthermore, for all $k\in\N$, also $\varphi_k$ is orthogonal to all radial functions. Such a sequence can be obtained by convolution of $\varphi$ with a smooth radial kernel. 
    In particular $\varphi_k$ converge in the distributional sense to $\varphi$ and thus it holds (observe that all integrals make sense---even though they might be infinite---since the functions involved are positive)
    \begin{equation*}
        \liminf_{k\to\infty} \int_{\R^n} \varphi_k^2\rho_U \ge \int_{\R^n} \varphi^2\rho_U.
    \end{equation*}
    Therefore, we get
    \begin{equation*}
        \norm{\varphi}_{\dot H^s}
        =
        \lim_{k\to\infty} \norm{\varphi_k}_{\dot H^s}
        \ge
        \liminf_{k\to\infty} \int_{\R^n} \varphi_k^2\rho_U
        \ge
        \int_{\R^n} \varphi^2\rho_U
    \end{equation*}
    as desired.
\end{proof}


\section{Non-degeneracy of positive solutions for \texorpdfstring{$\alpha\ge 0$}{alpha non-negative}}
\label{sec:non-degeneracy-proof}

In this section, we prove \cref{thm:non-degeneracy} and \cref{corollary nondeg hardy}. 
Let us begin by establishing some useful qualitative properties of $W$.

\begin{proposition}
\label{prop:W-regularity}
    Assume \cref{parameters-hardy}. Let $W \in \dot H^s(\R^n) \setminus \{0\}$ be a non-negative solution to \eqref{eq:hardy-eq}.
   Then $W$ is a strictly positive smooth function in $\R^n\setminus\{0\}$. 
    Moreover, if we additionally assume $\alpha\ge 0$, then $W$ is radially decreasing with $W'<0$ and 
    \[ W(x) \sim |x|^{-\alpha} \quad \text{ as } x \to 0, \qquad W(x) \sim |x|^{-n + 2s + \alpha} \quad \text{ as } |x| \to \infty. \]
\end{proposition}

\begin{proof}
    We begin by employing a standard bootstrap argument to show the smoothness of $W$ away from the origin.
    We have $W\in L^{2^*_s}(\R^n)$ since $\dot H^s(\R^n)$ embeds in $L^{2^*_s}(\R^n)$.
    Observe that if $W\in L^q_{\mathrm{loc}}(\R^n\setminus\{0\})$ for some $q\ge 1$, then $(-\lapl)^sW\in L^{q/(p-1)}_{\mathrm{loc}}(\R^n\setminus\{0\})$. And thus, provided $q$ is not too large, \cref{prop:fractional-elliptic-regularity}-\textit{(1)} tells us that $W\in L^{\frac{nq}{n(p-1)-2sq}}_{\mathrm{loc}}(\R^n\setminus\{0\})$. One can check (see the proof of \cref{lem:asymptotic-growth}, where this check is performed with care) that by iterating this argument (starting with $q=2^*_s$) finitely many times, we will eventually obtain that $W\in L^q_{\mathrm{loc}}(\R^n\setminus\{0\})$ for some $q>\frac{n}{2s}$.
    Then, we apply \cref{prop:fractional-elliptic-regularity}-\textit{(2)-(3)} to deduce that $W$ is smooth in $\R^n\setminus\{0\}$. 

    At this point, we know that $W$ is non-negative and smooth away from the origin. If, by contradiction, $W(\bar x)=0$ for some $\bar x\in\R^n\setminus\{0\}$, then by \cref{eq:hardy-eq} we would have $(-\lapl)^sW(\bar x)=0$ which is in contradiction with the formula \cref{eq:laplacian-pointwise} (which can be applied because $W$ is smooth at $\bar x$ and decays sufficiently fast at infinity, see \cite[Proposition 2.4]{Silvestre2007}).

  If we assume $\alpha\ge 0$, then $C(\alpha) \leq 0$ and hence the function $f(r,z) = -C(\alpha) r^{-2s} z + r^{-tp} z^{p-1}$ is radially non-increasing for every $z \geq 0$. Observe that \cref{eq:hardy-eq} is equivalent to $(-\lapl)^sW = f(\abs{x}, W)$. Thus \cite[Theorem 1.1]{Montoro2018} is applicable\footnote{An inspection of the proof shows that \cite[Theorem 1.1]{Montoro2018} remains true if the assumption $u \in L^1(\R^N)$ is replaced by the weaker assumption $\int_{\R^N} \frac{|u(x)|}{1 + |x|^{N+2s}} \, \mathrm{d} x < \infty$. See \cite[(2.10)]{Montoro2018}, which is the only place in the proof where this assumption is used. Observe that $\int_{\R^n} \frac{\abs{W(x)}}{1 + |x|^{n+2s}} \, \mathrm d x < \infty$ because $W\in\dot H^s(\R^n)$. } and yields that $W$ is radial and  strictly radially decreasing. 
    Finally, we can apply \cref{thm:radial-decreasing-negative-derivative} to get that $W'<0$ (note that this does not follow from the fact that $W$ is strictly radially decreasing).

By \cite[Lemma 3.10]{Abdellaoui2016}, we have that $W(x) \gtrsim |x|^{-\alpha}$ on $B_1$.\footnote{Note that in the notation of \cite{Abdellaoui2016}, $\gamma$ corresponds to our $\alpha$, and $\lambda$ corresponds to our $C(\alpha)$. To see this, one may look at the relationship between \cite[eq. (6)]{Abdellaoui2016} and \cite[eq. (21)]{Abdellaoui2016}, which corresponds to our formula \cref{hardy ckn trafo formula} linking $\alpha$ and $C(\alpha)$. The reference to \cite[eq. (19)]{Abdellaoui2016} in the statement of \cite[Lemma 3.10]{Abdellaoui2016} is erroneous.}

Moreover, as explained in Section \ref{sec:hardy-formulation}, the function $U(x) = |x|^{\alpha} W(x)$ satisfies \cref{U equation}. Then $U \in L^\infty(\R^n)$ by Lemma \ref{lemma Dipierro} below.  In particular, we have $U(x) \lesssim 1$ on $B_1$, and hence $W(x) \lesssim |x|^{-\alpha}$ on $B_1$. 

Altogether, we thus have $W(x) \sim |x|^{-\alpha}$ on $B_1$. Now consider the Kelvin-type transform $\tilde W(x) \defeq |x|^{-n + 2s} W(x/|x|^2)$. Since (see, e.g., \cite[Proposition 3.2]{MR3888401})
\[ (-\Delta)^s \tilde W(x) = |x|^{-n-2s} ((-\Delta)^s W)(\frac{x}{|x|^2}), \]
a direct computation shows that the function $\tilde W$ also satisfies \cref{eq:hardy-eq}. Thus the above implies $\tilde W(x) \sim |x|^{-\alpha}$ on $B_1$. This is equivalent to $W(x) \sim |x|^{-n + 2s + \alpha}$ on $\R^n \setminus B_1$, so the proof is complete. 
\end{proof}

Here is the lemma we have used in the previous proof. 

\begin{lemma}
    \label{lemma Dipierro}
    Assume \eqref{parameters-ckn}, and additionally $\alpha \geq 0$. Let $u \in D^s_\alpha(\R^n)$ be a non-negative weak solution to 
    \cref{U equation}. Then $u \in L^\infty(\R^n)$. 
\end{lemma}

\begin{proof}
    When $\alpha = 0$, the statement of the lemma is contained in \cite[Theorem 1.1]{MusinaNazarov2021}. When $\alpha > 0$, we deduce the lemma by repeating verbatim the proof of \cite[Proposition 4.5]{Dipierro2016}, after replacing (in the statement and proof of \cite[Proposition 4.5]{Dipierro2016}) the exponent $2^*_s$ by $p$ and the weight $|\cdot|^{-\alpha 2^*_s}$ by $|\cdot|^{-\beta p}$ (with $\beta$ satisfying \cref{parameters-ckn}). 
\end{proof}

Having established the regularity of $W$, we can show its non-degeneracy.
We perform a decomposition of $\varphi$ in radial and non-radial components. For the radial component, we apply \cite[Theorem 1.5]{AoDelaTorreGonzalez2022}; for the non-radial one, we use \cref{thm:general-hardy}.

\begin{proof}
    [Proof of Theorem \ref{thm:non-degeneracy}]
    Let $W \geq 0$ solve \cref{eq:hardy-eq} and   let $\varphi$ be a solution to \cref{lin-eq-thm}. Since $\alpha \geq 0$ by assumption, Proposition \ref{prop:W-regularity} gives that $W$ is radial with $W >0$ and $W' < 0$ on $\R^n \setminus \{0\}$.   

     Let us define the function $f(r, z) \defeq r^{-tp}z^{p-1}-C(\alpha)r^{-2s}z$. We observe that \cref{eq:hardy-eq} is equivalent to $(-\lapl)^sW = f(\abs{x}, W)$ and that \cref{lin-eq-thm} is equivalent to 
     \begin{equation}
         \label{lin-eq proof}
         (-\lapl)^s \varphi = \partial_2 f(\abs{x}, W) \varphi. 
     \end{equation}
          Write $\varphi = \varphi_0 + \tilde\varphi$, where $\varphi_0$ is radial and $\tilde\varphi$ is $L^2$-orthogonal to any radial function (or equivalently, $\int_{B_r}\tilde\varphi=0$ for all $r>0$). 
     Since $W$ is radial, and since $(-\Delta)^s$ preserves the classes of radial functions and of functions orthogonal to radial functions, both $\varphi_0$ and $\tilde \varphi$ are solutions to \cref{lin-eq proof}.

         Thanks to \cite[Theorem 1.5]{AoDelaTorreGonzalez2022}, since $\varphi_0$ is radial, we must have $\varphi_0 = \partial_\lambda|_{\lambda = 1} W_\lambda$. 

    To deal with $\tilde\varphi$, we can invoke \cref{thm:general-hardy}. So, we have
    \begin{equation*}
        \norm{\tilde\varphi}_{\dot H^s}^2 \ge \int_{\R^n} \tilde\varphi^2 \frac{((-\lapl)^sW)'}{W'} .
    \end{equation*}
    On the other hand, integrating \cref{lin-eq proof} for $\tilde \varphi$ against $\tilde\varphi$ gives 
        \begin{equation*}
        \norm{\tilde\varphi}_{\dot H^s}^2 = \int_{\R^n} \tilde\varphi^2 \partial_2 f(\abs{x}, W).
    \end{equation*}
    
    By differentiating $(-\lapl)^sW = f(\abs{x}, W)$ and recalling that $\partial_1 f(\abs{x}, W)<0$, we obtain $\frac{((-\lapl)^sW)'}{W'} > \partial_2 f(\abs{x}, W)$. Hence, unless $\tilde \varphi \equiv 0$, we find 
    \[  \norm{\tilde\varphi}_{\dot H^s}^2 \ge \int_{\R^n} \tilde\varphi^2 \frac{((-\lapl)^sW)'}{W'} > \int_{\R^n} \tilde\varphi^2 \partial_2 f(\abs{x}, W) =     \norm{\tilde\varphi}_{\dot H^s}^2 \]
    This is a contradiction. Thus we must have $\tilde \varphi \equiv 0$ and the proof is complete. 
\end{proof}

\begin{proof}
    [Proof of Corollary \ref{corollary nondeg hardy}]
       If $\varphi_0=\partial_\lambda|_{\lambda = 1} W_\lambda$, we already know that equality is achieved. 
       
       Assume conversely that equality is achieved in \cref{W Hessian} for some function $\varphi \in \dot H^s(\R^n)$ satisfying $\int_{\R^n} \frac{W^{p-1}}{|x|^{tp}} \varphi = 0$. Let us consider (as in the proof of \cref{thm:non-degeneracy}) the function $f(r, z) \defeq r^{-tp}z^{p-1}-C(\alpha)r^{-2s}z$. Recall that \cref{eq:hardy-eq} is equivalent to $(-\lapl)^sW=f(\abs{x},W)$.
       The fact that $\varphi$ is a minimizer for \cref{W Hessian} implies that
       $(-\lapl)^s\varphi-\partial_2 f(\abs{x}, W)\varphi = \lambda W^{p-1}\abs{x}^{-tp}$ for an appropriate Lagrange multiplier $\lambda\in\R$. Moreover, it holds that $\partial_2 f(\abs{x}, W)W - f(\abs{x}, W) = (p-2)W^{p-1}\abs{x}^{-tp} \perp_{L^2} \varphi$.
       Hence, we have
       \begin{equation*}
           \lambda \int_{\R^n} W^p\abs{x}^{-tp}\, \mathrm d x
           =
           \scalprod{W, (-\lapl)^s\varphi-\partial_2 f(\abs{x}, W)\varphi}_{L^2}=
           \scalprod{(-\lapl)^s W-\partial_2f(\abs{x}, W)W, \varphi}_{L^2} = 0,
       \end{equation*}
       thus $\lambda=0$.
       Hence, \cref{thm:non-degeneracy} yields $\varphi = c \partial_\lambda|_{\lambda = 1} W_\lambda$ as desired. 
\end{proof}


\section{Sharp quantitative stability}
\label{sec:stability}

In this section, we prove \cref{theorem stability can}. It is convenient to denote by 
\begin{equation}
    \label{U manifold}
    \mathcal W \defeq \{ c W_\lambda \, : \, c \in \R \setminus \{0\}, \, \lambda > 0 \}  
\end{equation} 
the set of minimizers of the inequality \cref{eq:hardy}. (Recall that here we fix $W$ minimizing \cref{eq:hardy} and satisfying \cref{eq:hardy-eq}, and write $W_\lambda(x) = \lambda^\frac{n-2s}{2} W(\lambda x)$ for its dilations.)

It is crucial to work with the norm
\begin{equation*}
\norm{w}_\ast^2 \defeq    \norm{w}_{{\dot{H}^s}}^2 + C(\alpha)\norm{w\abs{\emptyparam}^{-s}}_{L^2}^2, 
\end{equation*}
induced by the scalar product
\begin{equation*}
    \scalprod{w_1, w_2}_\ast\defeq \scalprod{(-\lapl)^sw_1, w_2} + C(\alpha)\int_{\R^n} w_1w_2\abs{x}^{-2s}.
\end{equation*}
This is a scalar product on $\dot H^s(\R^n)$ equivalent to the standard scalar product $\scalprod{w, w}_{\dot H^s} = \norm{w}_{\dot H^s}^2$ from \cref{Hs norm normalization}, for every $\alpha < \frac{n-2s}{2}$. That is, there is $c = c(n, s, \alpha)$ such that 
\begin{equation}
    \label{equivalence Hs scalarprod}
      c^{-1} \norm{w}_{\dot H^s} \leq \norm{w}_\ast \leq c \norm{w}_{\dot H^s}. 
\end{equation}
Indeed, by Hardy's inequality $C_\mathrm{Hardy}(s) \|w |\cdot|^{-s}\|_{L^2}^2 \leq \|w\|_{\dot H^s}^2$, the upper bound is immediate, and the lower bound  follows from \cref{eq:compare_calpha_chardy}. 

Let us begin by establishing a compact embedding result that often comes up when addressing stability questions.
\begin{lemma}
\label{lemma compact embedding}
    Assume \cref{parameters-hardy}. 
    Let $W\in \dot H^s(\R^n) \cap C^0(\R^n \setminus \{0\})$ be strictly positive outside of the origin. 
    Then, the embedding $\dot{H}^s(\R^n)\hookrightarrow L^2(\R^n, W^{p-2}\abs{\emptyparam}^{-tp})$ is compact.
\end{lemma}
\begin{proof}
We follow \cite[Appendix A]{FG20}.
Let $(\varphi_k)_{k\in\N}$ be a bounded sequence in $\dot{H}^s(\R^n)$. For any open set $\Omega\subseteq\R^n$, Hölder's inequality  yields
    \begin{equation*}
        \int_{\Omega} \varphi_k^2 W^{p-2}\abs{x}^{-tp} 
        \leq
        \norm{\varphi_k \abs{\emptyparam}^{-t}}_{L^p(\R^n)}^2
        \norm{W\abs{\emptyparam}^{-t}}_{L^p(\Omega)}^{p-2} \lesssim \norm{ \varphi_k}_{\dot{H}^s}^2  \norm{W\abs{\emptyparam}^{-t}}_{L^p(\Omega)}^{p-2} \lesssim \norm{W\abs{\emptyparam}^{-t}}_{L^p(\Omega)}^{p-2}.
    \end{equation*}
    We know that $W\in {\dot{H}^s}(\R^n)$ and thus, by the fractional Hardy--Sobolev inequality, $W\abs{x}^{-t}\in L^p(\R^n)$. Hence, 
    \begin{equation}
    \label{compact-proof-tail}
\begin{aligned}
        \int_{B_R^\complement} \varphi_k^2 W^{p-2}\abs{x}^{-tp} &\to 0 &&\text{as $R\to \infty$,}\\
        \int_{B_r} \varphi_k^2 W^{p-2}\abs{x}^{-tp} &\to 0 &&\text{as $r\to 0$,}
\end{aligned}
\end{equation}
    uniformly in $k$. 

    For every $0 < r < R < \infty$, by assumption, on $B_R \setminus B_r$, the weight $W^{p-2} \abs{x}^{-tp}$ is bounded away from zero and $\infty$. Hence, $L^2(B_R \setminus B_r) \simeq L^2(B_R \setminus B_r, W^{p-2} \abs{\emptyparam}^{-tp})$. In particular, $\dot H^s(\R^n)$ embeds compactly into $L^2(B_R \setminus B_r, W^{p-2} \abs{\emptyparam}^{-tp})$ by the standard compact embedding theorem for fractional Sobolev spaces, see, e.g., \cite[Theorem 7.1]{DiNezzaPalatucciValdinoci2012}.

By a diagonal extraction, we find $\varphi \in L^2(\R^n, W^{p-2}\abs{\emptyparam}^{-tp})$ such that $\varphi_k \to \varphi$ strongly in $L^2(B_R \setminus B_r, W^{p-2}\abs{\emptyparam}^{-tp})$ for every $r, R$. Together with \cref{compact-proof-tail}, it follows that $\varphi_k \to \varphi$ strongly in $L^2(\R^n, W^{p-2}\abs{\emptyparam}^{-tp})$. Hence, the embedding $\dot{H}^s(\R^n)\hookrightarrow L^2(\R^n, W^{p-2}\abs{\emptyparam}^{-tp})$ is compact. 
\end{proof}

The next lemma is central to the proof of \cref{theorem stability can}. Here, we use the fact that $W$ is a non-degenerate minimizer in order to study the spectrum of the linearized operator at $W$. 

\begin{lemma}
\label{lemma spectrum nondeg}
    Assume \cref{parameters-hardy} and, additionally, $\alpha\ge 0$.
    
    Let $W\in\dot H^s(\R^n)$ be a non-negative minimizer of \cref{eq:hardy} which satisfies \cref{eq:hardy-eq}.
    Then the operator
    \begin{equation*}
        \mathcal L_W \defeq     \frac{(-\lapl)^s + C(\alpha)\abs{x}^{-2s}}{W^{p-2}\abs{x}^{-tp}}
    \end{equation*}
    is the inverse of a positive self-adjoint compact operator on $L^2(\R^n, W^{p-2}\abs{\emptyparam}^{-tp})$. In particular, its spectrum is discrete.    

    Denoting by $\mu_0 < \mu_1 <\mu_2 < \ldots$ the sequence of its eigenvalues and by $(E_i)_{i \in \N_0}$ the corresponding eigenspaces, we have 
    \begin{equation}
        \label{mu i E i }
        \mu_0 = 1, \quad E_0 = \mathrm{span}\{W\}, \qquad \mu_1 = p-1, \quad E_1 = \mathrm{span}\{\partial_\lambda|_{\lambda = 1} W_\lambda\}. 
    \end{equation} 
    Moreover, the subspaces $E_i$ are mutually orthogonal with respect to $\scalprod{\emptyparam, \emptyparam}_\ast$. 
\end{lemma}

\begin{proof}
Let $f \in L^2(\R^n, W^{p-2}\abs{\emptyparam}^{-tp})$. Then, for every $\varphi \in \dot H^s(\R^n)$, using Cauchy--Schwarz followed by \cref{lemma compact embedding,equivalence Hs scalarprod}, we can estimate 
\begin{equation}
    \label{scalprod f varphi}
    \scalprod{f, \varphi}_{ L^2(\R^n, W^{p-2}\abs{\emptyparam}^{-tp})} \lesssim  \norm{f}_{ L^2(\R^n, W^{p-2}\abs{\emptyparam}^{-tp})} \norm{\varphi}_{\ast}.  
\end{equation} 
 Thus $\scalprod{f, \emptyparam}_{ L^2(\R^n, W^{p-2}\abs{\emptyparam}^{-tp})}$ is a continuous linear form on $(\dot H^s(\R^n), \scalprod{\emptyparam, \emptyparam}_*)$. By the Riesz representation theorem, there exists a unique $g =: T(f) \in \dot H^s(\R^n)$ such that $\scalprod{f, \varphi}_{ L^2(\R^n, W^{p-2}\abs{\emptyparam}^{-tp})} = \scalprod{g, \varphi}_\ast$. This is the weak formulation of $\mathcal L_W g = f$,
 so $T:L^2(\R^n, W^{p-2}\abs{\emptyparam}^{-tp})\to \dot H^s(\R^n)$ is inverted by $\mathcal L_W$. 
 
 Moreover, by \cref{scalprod f varphi}, the map $f \mapsto T(f)$ is continuous from $L^2(\R^n, W^{p-2}\abs{\emptyparam}^{-tp})$ into $\dot H^s(\R^n)$. By Proposition \ref{prop:W-regularity}, $W$ satisfies the assumptions of \cref{lemma compact embedding}. Hence, the embedding $\tau :\dot{H}^s(\R^n)\to L^2(\R^n, W^{p-2}\abs{\emptyparam}^{-tp})$ is compact. It follows that the composition $\tilde{T} \defeq \tau \circ T$ is a compact self-adjoint operator on $L^2(\R^n, W^{p-2}\abs{\emptyparam}^{-tp})$. 

In particular, the spectrum of $\tilde{T}$ consists of a countable number of real eigenvalues $(\lambda_k)$ which accumulate at $0$. It follows that the spectrum of  $ \mathcal L_W$ consists precisely of the numbers $\mu_k = \lambda_k^{-1}$, for $k \in \N$. 

The claimed orthogonality of the eigenspaces follows directly from $\mathcal L_W$ being self-adjoint on $L^2(\R^n, W^{p-2}\abs{\emptyparam}^{-tp})$  and the fact that, for $w_1, w_2\in\dot H^s(\R^n)$, we have
\begin{equation*}
    \scalprod{w_1, w_2}_\ast =
    \scalprod{\mathcal L_W w_1, w_2}_{L^2(\R^n, W^{p-2}\abs{x}^{-tp})}.
\end{equation*}

Let us now prove \cref{mu i E i }. First, we note that $W$ is an eigenfunction of $\mathcal L_W$ with eigenvalue 1. Due to the Rayleigh principle 
\[ \mu_0 = \inf_{\varphi \in L^2(\R^n, W^{p-2}\abs{\emptyparam}^{-tp})} \frac{\norm{\varphi}_\ast^2}{\norm{\varphi}^2_{ L^2(\R^n, W^{p-2}\abs{\emptyparam}^{-tp})}}\]
and the inequality $\norm{\abs{\varphi}}_\ast \leq \norm{\varphi}_\ast$, which is strict if $\varphi$ changes sign, every $\varphi \in E_0$ does not change sign. This implies that $\mu_0 = 1$ and $E_0 = \mathrm{span}\{W\}$. 

To prove the claims about $\mu_1$ and $E_1$, we will use our findings about non-degeneracy from \cref{sec:non-degeneracy-proof}. In view of the min-max characterization 
\begin{equation}
    \label{minmax mu1}
    \mu_1 = \inf_{\varphi \in L^2(\R^n, W^{p-2}\abs{\emptyparam}^{-tp}), \, \varphi \perp E_0} \frac{\norm{\varphi}_\ast^2}{\norm{\varphi}^2_{ L^2(\R^n, W^{p-2}\abs{\emptyparam}^{-tp})}} 
\end{equation} 
and the fact that $E_0 = \mathrm{span}\{W\}$, \cref{W Hessian} implies that $\mu_1 \geq p-1$. Still considering \cref{minmax mu1}, the statement of \cref{thm:non-degeneracy} can then be simply rephrased into saying that $\mu_1 = p-1$ and $E_1 = \mathrm{span}\{\partial_\lambda|_{\lambda = 1} W_\lambda\}$. 
\end{proof}

Our proof of \cref{theorem stability can} follows the classical strategy of Bianchi and Egnell~\cite{BianchiEgnell}. Let us study the deficit of the inequality \cref{eq:hardy} close to the manifold  $\mathcal W$ of its minimizers.

\begin{lemma} 
    \label{lemma BE local}
    Assume \cref{parameters-hardy} and, additionally, $\alpha\ge 0$.
        
    Let $\mathcal W$ be given by \cref{U manifold}. There is $\kappa=\kappa(n,s,p,\alpha) > 0$ such that the following statement holds.
    If $w_k \to w_\infty\in\mathcal W$, strongly in $\dot H^s(\R^n)$, then for all $k$ large enough, we have 
    \begin{equation}
        \label{inequality BE local}
         \norm{w_k}_\ast^2 - \tilde\Lambda \norm{w_k \abs{\emptyparam}^{-t}}_{L^p}^2 \geq \kappa \inf_{\tilde W\in \mathcal W} \norm{w_k-\tilde W}_{{\dot{H}^s}}^2.  
    \end{equation}
\end{lemma}

\begin{proof}
    It is not hard to check that $\inf_{\tilde W\in \mathcal W} \norm{w-\tilde W}_{\ast}$ is achieved for every $w \in \dot H^s(\R^n)$.\footnote{Indeed, we can complete a square and use the equation for $W_\lambda$ to find that 
    \begin{align*}
    \inf_{\tilde W \in \mathcal W} \norm{w - \tilde W}_\ast^2  &= \inf_{c, \lambda} \norm{w - c W_\lambda }_\ast^2  =  \norm{w}_\ast^2 - \norm{W}_\ast^{-2} \left( \sup_{\lambda > 0} \int_{\R^n} w W_\lambda \abs{x}^{-tp}  \right)^2 \end{align*}
      To see that $\sup_{\lambda > 0} \int_{\R^n} w W_\lambda \abs{x}^{-tp}$ is attained, it suffices to check that  $\int_{\R^n} w W_\lambda \abs{x}^{-tp} \to 0$ as $\lambda \to 0$ and $\lambda \to \infty$. But this follows from the general Cauchy--Schwarz estimate  
         \begin{align*}
        \scalprod{w, W} = 
        \scalprod{w\chi_{B_r}, W_\lambda\chi_{B_r}}
        +
        \scalprod{w\chi_{B_r^\complement}, W_\lambda\chi_{B_r^\complement}}
        \le
        \norm{w\chi_{B_r}}\norm{W_\lambda\chi_{B_r}}
        +
        \norm{w\chi_{B_r^\complement}}\norm{W_\lambda\chi_{B_r^\complement}}
    \end{align*}
    applied with $\scalprod{v,w} \defeq \int_{\R^n} vw \abs{x}^{-tp}$ and $\scalprod{v,v}^{1/2} =: \norm{v}$. 
    If $\lambda \to \infty$, then $W_\lambda$ is concentrated around the origin, so by picking $r$ small we have $\norm{w\chi_{B_r}}$ and $\norm{W_\lambda\chi_{B_r^\complement}}$ as small as we like, while $\norm{w\chi_{B_r^\complement}}$ and $\norm{W_\lambda\chi_{B_r}}$ are bounded uniformly in $r$, $\lambda$. Hence, $\scalprod{u, W_\lambda} \to 0$ as $\lambda \to \infty$. If $\lambda \to 0$, a variation of this argument gives the same conclusion. 
       }

     Up to multiplication by a constant and rescaling of $w_k$, we may assume that 
    \begin{equation}
        \label{distance-min}
        \inf_{\tilde W\in \mathcal W} \norm{w_k-\tilde W}_{\ast} =\norm{w_k-W}_{\ast}  
    \end{equation} 
    for every $k$, where $W \geq 0$ is a fixed (i.e., $k$-independent) minimizer of \cref{eq:hardy} satisfying \cref{eq:hardy-eq}.
  
    Let us decompose $w_k = W + \rho_k$. As a consequence of \cref{distance-min}, we find that.
    \begin{equation}
        \label{ortho conditions}
        \scalprod{\rho_k, W}_\ast = \scalprod{\rho_k,  \partial_\lambda |_{\lambda = 1} W_\lambda}_\ast =0. 
    \end{equation} 
    A consequence of \cref{ortho conditions} is that 
    \[ \norm{w_k}_\ast^2 = \norm{W}_\ast^2 + \norm{\rho_k}_\ast^2. \]
Moreover, using the pointwise expansion, for $a > 0$ and $b \in \R$,
\[ \abs{a + b}^p = a^p + p a^{p-1} b + \frac{p(p-1)}{2}a^{p-2} b^2 + \mathcal O( a^{(p-3)^+} \abs{b}^{p-(p-3)^+} +  \abs{b}^p) \]
with $a = W$ and $b = \rho$, and bounding the error terms using Hölder's inequality and \cref{lemma compact embedding}, we can expand 
\begin{equation}\begin{aligned} \label{eq:4564}
    \int_{\R^n} \abs{w_k} ^p \abs{x}^{-tp}&=  \int_{\R^n} \abs{W + \rho_k}^p \abs{x}^{-tp} \\
    &= \int_{\R^n} W^p \abs{x}^{-tp} + p \int_{\R^n} W^{p-1} \rho_k \abs{x}^{-tp} + \frac{p(p-1)}{2} \int_{\R^n} W^{p-2} \rho_k^2 \abs{x}^{-tp}  + o(\norm{\rho_k}_\ast^2).
\end{aligned}\end{equation}
By \cref{eq:hardy-eq,ortho conditions}, we have
\[ \int_{\R^n} W^{p-1} \rho_k \abs{x}^{-tp} = \scalprod{W, \rho_k}_\ast  =0.  \]
Thus, by taking the $\frac 2p$-th power of \cref{eq:4564}, we obtain
\[ \norm{w_k \abs{\emptyparam}^{-t}}_{L^p}^2 = \norm{W \abs{\emptyparam}^{-t}}_{L^p}^2 + (p-1) \tilde\Lambda^{-1} \int_{\R^n} W^{p-2} \rho_k^2 \abs{x}^{-tp} + o(\norm{\rho_k}_\ast^2), \]
where we used that $\norm{W \abs{\emptyparam}^{-t}}_{L^p}^{p-2} = \tilde\Lambda$ by the normalization of $W$. (Here, $\tilde\Lambda$ is the best constant in \cref{eq:hardy}.)

Combining all of this and using $\norm{W}_\ast^2 = \tilde\Lambda \norm{W \abs{\emptyparam}^{-t}}_{L^p}^2$ by \cref{eq:hardy-eq}, we get 
\begin{align*}
     \norm{w_k}_\ast^2 - \tilde\Lambda \norm{w_k \abs{\emptyparam}^{-t}}_{L^p}^2 &= \norm{\rho_k}_\ast^2 - (p-1) \int_{\R^n} W^{p-2} \rho_k^2 \abs{x}^{-tp} +  o(\norm{\rho_k}_\ast^2).
\end{align*}
By the orthogonality conditions \cref{ortho conditions} satisfied by $\rho_k$, \cref{lemma spectrum nondeg} yields the decisive information that 
´\[ \norm{\rho_k}_\ast^2 \geq \mu_2 \int_{\R^n} W^{p-2} \rho_k^2 \abs{x}^{-tp}  \]
for $\mu_2 > p-1$ the second eigenvalue of the operator $\mathcal L_W$ from \cref{lemma spectrum nondeg}. 
Hence, for $k$ large enough we have 
\[ \norm{w_k}_\ast^2 - \tilde\Lambda \norm{w_k \abs{\emptyparam}^{-t}}_{L^p}^2 \geq \left(1 - \frac{p-1}{\mu_2}\right) \norm{\rho_k}_\ast^2 + o(\norm{\rho_k}_\ast^2) \gtrsim \norm{\rho_k}_\ast^2,  \]
which implies the desired statement.
\end{proof}

The last missing step is to extend the local validity of the inequality \cref{inequality BE local}, which was just proved in \cref{lemma BE local}, to a global one. To do so, still following \cite{BianchiEgnell}, the following compactness lemma is required. 

\begin{lemma}
    \label{lemma compactness lions}
    Assume \cref{parameters-hardy}. Let $(w_k)_{k\in\N} \subset \dot H^s(\R^n)$ be such that  $\norm{w_k \abs{\emptyparam}^{-t}}_{L^p} = 1$ and $\norm{w_k}_\ast^2 \to \tilde\Lambda$, where $\tilde \Lambda$ is the best constant in \cref{eq:hardy}. Then there is a sequence $(\lambda_k)_{k \in \N} \subset (0, \infty)$ such that 
    \[ \lambda_k^\frac{n-2s}{2} w_k(\lambda_k \, \cdot) \to w_\infty \]
   strongly in $\dot H^s(\R^n)$, where $w_\infty$ is a minimizer of \cref{eq:hardy}.
\end{lemma}

\Cref{lemma compactness lions} will be proved in \cref{subsection lions} below.

With the help of the above results, we can conclude the proof of \cref{theorem stability can}. 

\begin{proof}
    [Proof of \cref{theorem stability can}]
    We argue by contradiction and suppose that \cref{inequality stability thm} is false. That is, there exist $(w_k)_{k\in\N} \subseteq \dot H^s(\R^n)$, which we can take without loss of generality to satisfy $ \norm{w_k \abs{\emptyparam}^{-t}}_{L^p} = 1$, such that 
    \begin{equation}
        \label{BE contradiction ass}
        \frac{\norm{w_k}_{\ast}^2 - \tilde\Lambda}{\inf_{\tilde W\in \mathcal W} \norm{w_k-\tilde W}_{{\dot{H}^s}}^2}  \to 0
    \end{equation}  
    as $k \to \infty$. 
     Since $\inf_{\tilde W\in \mathcal W} \norm{w_k-\tilde W}_{{\dot{H}^s}}^2 \leq \norm{w_k}_{\dot H^s}^2 \leq c^2 \norm{w_k}_\ast^2$ (with $c$ as in \cref{equivalence Hs scalarprod}), \cref{BE contradiction ass} implies that $\norm{w_k}_\ast^2 \to \tilde\Lambda$. Hence, we can apply \cref{lemma compactness lions} to the sequence $(w_k)_{k\in\N}$ and deduce that $\inf_{\tilde W\in \mathcal W} \norm{w_k-\tilde W}_{{\dot{H}^s}}^2 \to 0$. Then, \cref{lemma BE local} is applicable to a rescaled subsequence of $(w_k)_{k\in\N}$ and yields a contradiction to \cref{BE contradiction ass}. 
\end{proof}

We conclude by proving \cref{corollary remainder Hardy}. Recall that the weak $L^r$ norm (which is not a norm) is defined by 
\[\norm{u}_{L^{r, \infty}} \defeq \sup_{c>0} c \cdot \abs*{\{x\in\R^n:\, \abs{u(x)}>c\}}^{\tfrac 1r} \,.
\] 

\begin{proof}[Proof of \cref{corollary remainder Hardy}]
\textbf{Step 1.}
\emph{Reduction steps.} 
    By \cref{theorem stability can} it suffices to prove that there exists $C = C(n,s,\alpha)$ such that 
\begin{equation}
    \label{weak to dist new}
     |\Omega|^{-\frac{n-2s-2\alpha}{n}} \norm{w}^2_{L^{\frac{n}{n-2s-\alpha}, \infty}} \lesssim \inf_{c \in \R, \lambda > 0} \norm{w - c W_\lambda}^2_{\dot H^s} = d(w, \mathcal W)^2
\end{equation}
(where, as usual, we have fixed a minimizer $W > 0$ that satisfies \cref{eq:hardy-eq}). 

For the rest of the proof, fix $c \in \R$ and $\lambda > 0$ such that 
\[ \norm{w - c W_\lambda}_{\dot H^s} =  d(u, \mathcal W). \]
By homogeneity of \cref{weak to dist new}, we may assume $c = 1$.

Replacing $w$ by its symmetric-decreasing rearrangement and $\Omega$ by the ball of same volume leaves invariant the left side of \cref{weak to dist new}  and decreases the right side.\footnote{To see that the right side of \cref{weak to dist new} decreases under symmetric-decreasing rearrangement, write $\norm{w -  W_\lambda}^2_{\dot H^s} = \norm{w}^2_{\dot H^s} + C(\alpha) \int_{\R^n} \frac{w^2}{|x|^{2s}}  - 2  \int_{\R^n} w W^{p-1}_\lambda |x|^{-tp} + + \norm{W_\lambda}_{\dot H^s}^2$. 
Now use \cite{AlmgrenLieb1989}, $C(\alpha) < 0$, and the fact that $W$ is symmetric-decreasing.} 

 Moreover, replacing $w$ by $w_\lambda = \lambda^\frac{n-2s}{2} w(\lambda x)$ and $\Omega$ by $\lambda^{-1} \Omega$ leaves invariant both sides of inequality \cref{remainder term ineq}. 
 
Summarizing, it suffices to prove 
\begin{equation}
    \label{goal remainder}
    \norm{w - W_\lambda}_{\dot H^s} \gtrsim \norm{w}_{{L^{\frac{n}{n-2s-\alpha}, \infty}}}. 
\end{equation}
for all $w$ supported on $B$, the ball centered in $0$ such that $|B| = 1$.

To do so, we are going to distinguish two cases depending on the value of $\lambda$. 

\textbf{Step 2.} \emph{The case when $\lambda$ is small. }
We have 
\begin{align*}
    \norm{w - W_\lambda}_{\dot H^s} & \gtrsim \norm{(w - W_\lambda)}_{L^{2^\ast_s}} \\
    &\gtrsim \norm{(W_\lambda) 
    }_{L^{2^*_s}(\R^n \setminus B)} + \norm{(w - W_\lambda) }_{L^{2^*_s}(B)} \\
    & \geq \norm{(W_\lambda)}_{L^{2^*_s}(\R^n \setminus B)} + \norm{w }_{L^{2^*_s}(B)} - \norm{W_\lambda}_{L^{2^*_s}(B)},
\end{align*}
where we used that $w$ is supported in $B$. 
If $\lambda$ is small enough, say $\lambda \leq c_0$ for some appropriate $c_0 = c_0(n,s,p,\alpha) > 0$, then 
\[ \norm{W_\lambda }_{L^{2^*_s}(\R^n \setminus B)}^{2^*_s} = \int_{\R^n \setminus \lambda B} W^{2^*_s} \geq \int_{\lambda B} W^{2^*_s} = \norm{W_\lambda}_{L^{2^*_s}(B)}^{2^*_s}.  \]
Since $\alpha<\frac{n-2s}2$, we have $\frac{n}{n-2s-\alpha}<2^*_s$. 
Thus we get
\begin{equation*}
    \norm{w }_{L^{2^*_s}}
    \ge
    \norm{w}_{L^{\frac{n}{n-2s-\alpha}}}
    \ge
    \norm{w}_{L^{\frac{n}{n-2s-\alpha}, \infty}}
\end{equation*} 
by Hölder, and \cref{goal remainder} follows in this case.

\textbf{Step 3.} \emph{The case when $\lambda$ is large. }

By Step 2, we may assume $\lambda \geq c_0$ now, for some (universal) $c_0 > 0$. To prove \cref{goal remainder} in this case, we start by writing 
\[ \norm{w}_{L^{\frac{n}{n-2s-\alpha}, \infty}} \lesssim \norm{W_\lambda \chi_B}_{L^{\frac{n}{n-2s-\alpha}, \infty}} + \norm{w - W_\lambda \chi_B}_{L^{\frac{n}{n-2s-\alpha}, \infty}}, \]
where $\chi_B$ denotes the indicator function of $B$. 
By Hölder and Sobolev, we again have
\begin{align*}
    \norm{w - W_\lambda \chi_B}_{L^{\frac{n}{n-2s-\alpha}, \infty}} &\leq \norm{w - W_\lambda \chi_B}_{L^{\frac{n}{n-2s-\alpha}}(B)} \leq  \norm{w - W_\lambda 1_B}_{L^{2^*_s}(B)}  \\
    &\leq \norm{w - W_\lambda}_{L^{2^*_s}(\R^n)}  \lesssim \norm{w - W_\lambda}_{\dot H^s}. 
\end{align*} 
The assumptions \cref{parameters-hardy} imply $-\alpha > -n+2s+\alpha$ and therefore \cref{prop:W-regularity} shows that $W(x) \lesssim |x|^{-(n-2s-\alpha)}$. Hence, 
\[ \norm{W_\lambda \chi_B}_{L^{\frac{n}{n-2s-\alpha}, \infty}} \leq \norm{W_\lambda}_{L^{\frac{n}{n-2s-\alpha}, \infty}} 
= \lambda^{\frac{-n + 2s + 2 \alpha}{2}} \norm{W}_{L^{\frac{n}{n-2s-\alpha}, \infty}}  
\lesssim 
\lambda^{\frac{-n + 2s + 2 \alpha}{2}}
.
\]

On the other hand, since $W(x) \gtrsim |x|^{-n+2s+\alpha}$ on $\R^n \setminus c_0 B$ by \cref{prop:W-regularity}, we have 
\[ \norm{w - W_\lambda}_{\dot H^s}^p \gtrsim \norm{(w - W_\lambda) |\emptyparam|^{-t}}_{L^p}^p \geq  \int_{\R^n \setminus B} W_\lambda^p |x|^{-tp} = \int_{\R^n \setminus \lambda B} W^p |x|^{-tp} \gtrsim \lambda^{(-n + 2s + \alpha -t)p + n}.  \]
In the second inequality we again used that $w$ is supported in $B$.
Recalling that $t = s - \frac{n}{2} + \frac{n}{p}$, it follows 
\[ \norm{w - W_\lambda}_{\dot H^s}  \gtrsim \lambda^{(-n + 2s + \alpha -t) + \frac{n}{p}} = \lambda^{\frac{-n+2s+2\alpha}{2}}. \]
By combining all these estimates, \cref{goal remainder} follows in this case as well. 
\end{proof}

\subsection{Proof of \texorpdfstring{\cref{lemma compactness lions}}{Lemma~\ref{lemma compactness lions}}}
\label{subsection lions}

\Cref{lemma compactness lions} would follow readily by transforming to the setting of the cylinder and using the argument in \cite[proof of Theorem 1.2 (iv)]{AoDelaTorreGonzalez2022}. Similarly, \cref{lemma compactness lions} would follow from the considerations in \cite[Section 3]{Ghoussoub2015}, using the Caffarelli--Silvestre extension.

However, since we do not use the cylindrical formulation nor the Caffarelli--Silvestre extension anywhere else in this paper, we find it of some interest to give a direct argument on $\R^n$. The argument has exactly the same structure as the one on the cylinder, only translations are replaced by dilations. 

We need the following lemma about the concentration behavior of an $\dot H^s$-bounded sequence. 

\begin{lemma}
\label{lemma Lq balls lions}
    Assume \cref{parameters-hardy}.
Let $(w_k)_{k\in\N}$ be a sequence of functions uniformly bounded in $\dot{H}^s(\R^n)$. 
Suppose that 
\begin{equation}
    \label{lions assumption local to zero}
    \sup_{r > 0} \int_{B_{2r} \setminus B_r} \frac{\abs{w_k}^p}{\abs{x}^{tp}} \to 0 \qquad \text{ as } k \to \infty.  
\end{equation} 
Then $\int_{\R^n} \frac{\abs{w_k}^p}{\abs{x}^{tp}} \to 0$.
\end{lemma}

\begin{proof}
We follow the proof of \cite[Lemma 1.21]{Willem}. 
For every $q \in (2, 2^*_s)$, we denote $t_q \defeq s - n(\frac{1}{2} - \frac{1}{q}) \in (0,s)$; thus $t = t_p$ for the fixed parameter $p$ in our notation in the rest of the paper. For every $q \in (p, 2^*_s)$, we have $t_p p > t_q q$. Therefore, by Hölder's inequality, for every open set $\Omega$ and any $w\in \dot H^s(\R^n)$, 
\[ 
\norm{w\abs{\emptyparam}^{-t_q}}_{L^q(\Omega)}^q
\le
\norm{w\abs{\emptyparam}^{-t_p}}_{L^p(\Omega)}^{\frac{t_q q}{t_p}}
\norm{w}_{L^{2^*_s}(\Omega)}^{2^*_s(1 - \frac{t_q q}{t_p p})} .
\]
Choosing $q = 2 + \frac{2sp}{n} \in (p, 2^*_s)$ we have $1 - \frac{t_q q}{t_p p} = \frac{2}{2^*_s}$ and so, by Sobolev inequality on $\Omega$, 
\begin{equation}
    \label{hölder willem}
    \norm{w\abs{\emptyparam}^{-t_q}}_{L^q(\Omega)}^q
    \le 
    S(\Omega) 
    \norm{w\abs{\emptyparam}^{-t_p}}_{L^p(\Omega)}^\theta
    \Big(\norm{w}_{\dot H^s(\Omega)}^2 + \norm{w}_{L^2(\Omega)}^2\Big).  
\end{equation}  
where $\theta \defeq \frac{t_q q}{t_p p}$ and  
\[  \norm{w}_{\dot H^s(\Omega)}^2 \defeq \iint_{\Omega \times \Omega} \frac{\abs{w(x) - w(y)}^2}{\abs{x-y}^{n+2s}} \, \mathrm dx \mathrm dy  . \]
Given $\lambda>0$, by replacing $w$ with $w(\lambda\emptyparam)$ in \cref{hölder willem}, we obtain
\begin{equation}
    \label{hölder willem 2}
    \norm{w\abs{\emptyparam}^{-t_q}}_{L^q(\lambda\Omega)}^q
    \le 
    S(\Omega) 
    \norm{w\abs{\emptyparam}^{-t_p}}_{L^p(\lambda\Omega)}^\theta
    \Big(\norm{w}_{\dot H^s(\lambda\Omega)}^2 + \lambda^{-2s}\norm{w}_{L^2(\lambda\Omega)}^2\Big).
\end{equation} 
We stress that the constant $S(\Omega)$ here does not depend on $\lambda$.

We now apply \cref{hölder willem 2} with $\Omega \defeq B_1 \setminus \overline{B_{\frac12}}$ and $\lambda = 2^i$ for $i\in\Z$.
Summing over all $i\in\Z$, we find 
\begin{equation}
\begin{aligned}
    \label{hölder willem 3}
    \norm{w\abs{\emptyparam}^{-t_q}}_{L^q}^q
    &\le
    S(B_1 \setminus \overline{B_{\frac12}})
    \Big(\sup_{r>0} \norm{w\abs{\emptyparam}^{-t_p}}_{L^p(B_{2r}\setminus B_r)}^\theta\Big)
    \Big(\norm{w}_{\dot H^s}^2 + \norm{w\abs{\emptyparam}^{-s}}_{L^2}^2\Big)\\
    &\lesssim
    \Big(\sup_{r>0} \norm{w\abs{\emptyparam}^{-t_p}}_{L^p(B_{2r}\setminus B_r)}\Big)^\theta
    \norm{w}_{\dot H^s}^2
    ,
\end{aligned}
\end{equation}
where the last step is justified by the fractional Hardy inequality.

Moreover, Hölder's inequality and the fractional Hardy inequality tell us that, for some $\vartheta \in (0,1)$, 
\begin{equation}\label{hölder willem 4}
    \norm{w\abs{\emptyparam}^{-t_p}}_{L^p} 
    \le 
    \norm{w\abs{\emptyparam}^{-s}}_{L^2}^\vartheta
    \norm{w\abs{\emptyparam}^{-t_q}}_{L^q}^{1-\vartheta}
    \lesssim
    \norm{w}_{\dot H^s}^\vartheta
    \norm{w\abs{\emptyparam}^{-t_q}}_{L^q}^{1-\vartheta}.
\end{equation}
Combining \cref{hölder willem 3,hölder willem 4}, we deduce
\begin{equation*}
    \norm{w\abs{\emptyparam}^{-t_p}}_{L^p}
    \lesssim 
    \Big(
        \sup_{r>0} \norm{w\abs{\emptyparam}^{-t_p}}_{L^p(B_{2r}\setminus B_r)}
    \Big)^{\frac{\theta(1-\vartheta)}q}
    \norm{w}_{\dot H^s}^{\vartheta + \frac{2(1-\vartheta)}q}.
\end{equation*}
The latter inequality implies the desired statement plugging in $w=w_k$ and letting $k$ go to infinity.
\end{proof}

\begin{proof}[Proof of \cref{lemma compactness lions}] 
    Let $(w_k)_{k\in\N} \subset \dot H^s(\R^n)$ be a minimizing sequence for the inequality \cref{eq:hardy}, normalized such that $\norm{w_k \abs{\emptyparam}^{-t}}_{L^p} = 1$ and $\norm{w_k}_\ast^2 \to \tilde\Lambda$ as $k \to \infty$. 

By \cref{lemma Lq balls lions}, there exists $\delta > 0$ such that, for all $k\in\N$, there is $\lambda_k > 0$ satisfying 
\[ \int_{B_{2\lambda_k} \setminus B_{\lambda_k}} \frac{\abs{w_k}^p}{\abs{x}^{tp}} \geq \delta. \]

Let $v_k(x) \defeq \lambda_k^\frac{n-2s}{2} w_k(\lambda_k x)$. Then, $\norm{v_k \abs{\emptyparam}^{-t}}_{L^p} = 1$ and $\norm{v_k}_\ast^2 \to \tilde\Lambda$ and 
\begin{equation}
    \label{v non-zero}
    \int_{B_2 \setminus B_1} \frac{\abs{v_k}^p}{\abs{x}^{tp}} =  \int_{B_{2\lambda_k} \setminus B_{\lambda_k}} \frac{\abs{w_k}^p}{\abs{x}^{tp}} \geq \delta. 
\end{equation}
Since $v_k$ is bounded in $\dot H^s$, there is $v \in \dot H^s$ such that, up to a subsequence,
\[ v_k \rightharpoonup v \qquad \text{ weakly in } \dot H^s. \]
Moreover, since $p < \frac{2n}{n-2s}$, we have the compact embedding 
\[\dot H^s(B_R \setminus B_r) \hookrightarrow L^p(B_R \setminus B_r) \simeq L^p(B_R \setminus B_r, \abs{\emptyparam}^{-tp})\]
for every annulus $B_R \setminus B_r$. By diagonal extraction along a sequence of radii $R \to \infty$ and $r \to 0$, we thus can assume that a further subsequence satisfies 
\begin{align*}
    v_k &\to v \qquad \text{ in } L^p_{\mathrm{loc}}(\R^n \setminus \{0\}, \abs{\emptyparam}^{-tp}), \\
    v_k(x) &\to v \qquad \text{ pointwise for a.e. } x \in \R^n. 
    \end{align*}
In particular, by the strong $L^p_{\mathrm{loc}}$-convergence we have 
\[ \int_{B_2 \setminus B_1} \frac{\abs{v}}{\abs{x}^{tp}}  = \lim_{k \to \infty} \int_{B_2 \setminus B_1} \frac{\abs{v_k}}{\abs{x}^{tp}} \geq \delta,  \]
so $v \not \equiv 0$. 
By weak lower semi-continuity (of the norm in the Hilbert space of functions in $\dot H^s(\R^n)$ endowed with the scalar product $\scalprod{\emptyparam,\emptyparam}_\ast$), $\norm{v}_\ast \leq \liminf_{k \to \infty} \norm{v_k}_\ast = \tilde\Lambda$.
So $v$ is a minimizer provided that we can show $\norm{v \abs{\emptyparam}^{-t}}_{L^p} = 1$. To achieve this, we write $v_k = v + \rho_k$. 
Using inequality \cref{eq:hardy} we find 
\[ \tilde\Lambda \left(\norm{v \abs{\emptyparam}^{-t}}_{L^p}^2 + \norm{\rho_k \abs{\emptyparam}^{-t}}_{L^p}^2 \right) \leq \norm{v}_\ast ^2 + \norm{\rho_k}_\ast^2. \]
Taking $\limsup$ on both sides and denoting $R \defeq \limsup_{k \to \infty} \norm{\rho_k \abs{\emptyparam}^{-t}}_{L^p}$, we find 
\begin{equation}
    \label{sobolevestimate}
    \tilde\Lambda \left(\norm{v \abs{\emptyparam}^{-t}}_{L^p}^2 + R^2 \right) \leq \norm{v}_\ast ^2 + \limsup_{k \to \infty} \norm{\rho_k}_\ast^2 = \lim_{k \to \infty} \norm{v_k}_\ast^2 = \tilde\Lambda, 
\end{equation} 
that is, 
\begin{equation}
    \label{concavity1}
    \norm{v \abs{\emptyparam}^{-t}}_{L^p}^2 + R^2  \leq 1.  
\end{equation} 
On the other hand, 
\begin{align}
\norm{v \abs{\emptyparam}^{-t}}_{L^p}^2 + R^2  &\geq \left(\norm{v \abs{\emptyparam}^{-t}}_{L^p}^p + R^p \right)^{2/p} = \left( \lim_{k \to \infty} \norm{(v + \rho_k) \abs{\emptyparam}^{-t}}_{L^p}^p \right)^{2/p} = 1. \label{concavity2}
    \end{align}
Here, for the first step, we used the concavity of $t \mapsto t^{2/p}$ and in the second we used the Brezis--Lieb lemma \cite{BrezisLieb1983} for the weighted space $L^p(\R^n, \abs{\emptyparam}^{-tp})$. 

Combining \cref{concavity1,concavity2}, we see that equality must hold in all of the above inequalities. Since $t \mapsto t^{2/p}$ is strictly concave on $(0, \infty)$, equality in the concavity inequality implies that one of the summands $\norm{v \abs{\emptyparam}^{-t}}_{L^p}^p$ or $R$ must be zero. Since $v \not \equiv 0$, we must have $\limsup_{k \to \infty} \norm{\rho_k \abs{\emptyparam}^{-t}}_{L^p} = R = 0$. Coming back to \cref{sobolevestimate}, we see $\lim_{k \to \infty} \norm{v_k}_\ast^2 = \norm{v}_\ast ^2$. Together with $v_k \rightharpoonup v$ in $\dot H^s$, this implies that $v_k \to v$ strongly in $\dot H^s$. Since $v \not \equiv 0$, $v$ is the desired minimizer.

Since minimizers do not change sign, either $v$ or $-v$ is the desired non-negative minimizer. 
 The value $c = \tilde\Lambda^{-\frac{1}{p-2}- \frac{1}{2}}$ is determined from the normalization of $v$. 
\end{proof}

\section{Symmetry for \texorpdfstring{$p$}{p} sufficiently close to \texorpdfstring{$2$}{2}}
\label{sec:symmetry-proof}

In this section, we present the proof of \Cref{thm:perturbative-symmetry-hardy}: if $p$ is taken sufficiently close to $2$, then the minimizers of \cref{eq:hardy} are radially symmetric. 

The equivalent statement for the classical case (i.e., $s=1$) was proven in \cite[Theorem 1.1]{Dolbeault2009} with the following strategy.
Consider an angular derivative $\chi_{\alpha, p}$ of a minimizer $W_{\alpha,p}$ of \cref{eq:hardy} and show that, as $p\to 2$, $(\chi_p)_{p>2}$ is a minimizing sequence for the fractional Hardy inequality. If $\chi_{\alpha, p}\not\equiv0$, this yields a contradiction because $\chi_{\alpha,p}$ is orthogonal to all radial functions but the Hardy inequality holds with an improved constant for functions orthogonal to radial ones. Thus $\chi_{\alpha,p}\equiv 0$ and therefore $W_{\alpha,p}$ is radial.

Repeating the same scheme in the fractional case presents some difficulties. First of all, the angular derivative $\chi_p$ may not be as integrable as necessary, so we consider a discrete angular derivative. 
Moreover, to find a contradiction as $p\to 2$, one needs a pointwise control on $W_{\alpha,p}$. When $s=1$, such control follows from classical elliptic regularity and a comparison with the minimizer among radial functions. 
 When $0<s<1$, we need to devise a different argument because the minimizer among radial functions is not explicit. The main new contribution is \cref{lem:asymptotic-growth}, which produces a pointwise control for subsolutions of critical fractional elliptic equations and in turn yields the desired control over $W_{\alpha,p}$. The proof is particularly technical because we need to prove a statement that is stable as $p$ and $\alpha$ change value. 

\begin{proof}[Proof of \Cref{thm:perturbative-symmetry-hardy}] 

\textbf{Step 1.} \emph{Setup.} Given the parameters set up in \cref{parameters-hardy} and further assuming $0 > \alpha > \alpha_0 $, we consider the functional
    \begin{equation*}
        \mathcal F_{\alpha,p}(w) \defeq 
        \frac{\norm{w}_{{\dot{H}^s}}^2 + C(\alpha)\norm{w\abs{\emptyparam}^{-s}}_{L^2}^2}{\norm{w\abs{\emptyparam}^{-t}}_{L^p}^2},
    \end{equation*}
    where $t=s -n\big(\frac12-\frac1p\big)$ depends on $p$.
   Let $\tilde\Lambda_{\alpha,p}$ be the infimum of $\mathcal F_{\alpha,p}$ (over non-zero functions) and let $W_{\alpha,p}$ be a non-negative minimizer, i.e., $\mathcal F_{\alpha,p}(W_{\alpha,p}) = \tilde\Lambda_{\alpha,p}$. {(Recall that such $W_{\alpha,p}$ exists by \cite[Theorem 1.1]{Ghoussoub2015}.)}
    We already know, thanks to \cref{eq:hardy-eq}, that
    \begin{equation}\label{eq:local10}
        (-\lapl)^sW_{\alpha,p} + C(\alpha)W_{\alpha,p}\abs{x}^{-2s} = W_{\alpha,p}^{p-1}\abs{x}^{-tp},
    \end{equation}
    where we have normalized $W_{\alpha,p}$ so that $\tilde\Lambda_{\alpha,p}=\norm{W_{\alpha,p}\abs{\emptyparam}^{-t}}_{L^p}^{p-2}$. 

We want to show that there exists $\varepsilon >0$ (depending on $n$, $s$, $\alpha_0$) such that, if $p \in (2,2+\varepsilon)$ and $\alpha_0\le \alpha\le 0$, then $W_{\alpha,p}$ is radially symmetric. 
We prove it by contradiction, so we assume the existence of a sequence $\alpha_k\to \alpha_\infty \ge \alpha_0$ and $p_k\to 2$ so that $W_{\alpha_k, p_k}$ is not radial.

To proceed, we would like to differentiate \cref{eq:local10} along an infinitesimal rotation. Since, a priori, we lack the needed regularity estimates on $W_{\alpha,p}$, we perform a discrete derivative instead.
Let $R:\R^n\to\R^n$ be a rotation and let \[\chi_{\alpha,p}\defeq W_{\alpha,p}\circ R-W_{\alpha,p}.\] 
We will show that $\chi_{\alpha_k,p_k}=0$ for $k$ sufficiently large.

    \textbf{Step 2.} \emph{Differentiating the Euler--Lagrange equations.}
    From \cref{eq:local10}, we deduce
    \begin{equation*}
        (-\lapl)^s\chi_{\alpha,p} + C(\alpha)\chi_{\alpha,p}\abs{x}^{-2s}
        = 
        \frac{(W_{\alpha,p}\circ R)^{p-1}-W_{\alpha,p}^{p-1}}{W_{\alpha,p}\circ R-W_{\alpha,p}}\abs{x}^{-tp+2s} \chi_{\alpha,p} \abs{x}^{-2s}
    \end{equation*}
    and integrating against $\chi_{\alpha,p}$ we obtain
    \begin{equation*}
       \mathcal F_{\alpha,2}(\chi_{\alpha,p})\norm{\chi_{\alpha,p}\abs{\emptyparam}^{-s}}_{L^2}^2
        \le 
        \norm*{\frac{(W_{\alpha,p}\circ R)^{p-1}-W_{\alpha,p}^{p-1}}{W_{\alpha,p}\circ R-W_{\alpha,p}}\abs{\emptyparam}^{-tp+2s}}_{L^\infty} \norm{\chi_{\alpha,p}\abs{\emptyparam}^{-s}}_{L^2}^2.
    \end{equation*}
    Since
    \begin{equation*}
        \norm*{\frac{(W_{\alpha,p}\circ R)^{p-1}-W_{\alpha,p}^{p-1}}{W_{\alpha,p}\circ R-W_{\alpha,p}}\abs{\emptyparam}^{-tp+2s}}_{L^\infty}
        \le
        (p-1)\norm{W_{\alpha,p}\abs{\emptyparam}^{(-tp+2s)/(p-2)}}^{p-2}_{L^\infty},
    \end{equation*}
    provided that $\chi_{\alpha,p}$ is not identically zero, we get (noticing that $\frac{-tp+2s}{p-2}=\frac n2-s$)
    \begin{equation*}
        \mathcal F_{\alpha,2}(\chi_{\alpha,p}) \le  (p-1)  \norm{W_{\alpha,p}\abs{\emptyparam}^{\frac n2-s}}^{p-2}_{L^\infty}.
    \end{equation*}
    To conclude, we need to establish three facts: 
    \begin{enumerate}[label={(F\arabic*)}]
        \item \label{it:f1} $\limsup_{k\to\infty} \tilde\Lambda_{\alpha_k,p_k}\le \tilde\Lambda_{\alpha_\infty,2}$;
        \item \label{it:f2} There exists a constant $\tilde\Lambda'_{\alpha_\infty,2}>\tilde\Lambda_{\alpha_\infty,2}$ such that, for any $\alpha$ sufficiently close to $\alpha_\infty$, $\mathcal F_{\alpha, 2}(\chi)> \tilde\Lambda'_{\alpha_\infty, 2}$ holds for all functions $\chi$ orthogonal to radial ones;
        \item \label{it:f3}  There exist $M=M(n, s)$ and $\kappa=\kappa(n,s)\in\N\setminus \{0\}$ such that
        \begin{equation*}
            \norm{W_{\alpha,p}\abs{\emptyparam}^{\frac n2-s}}_{L^\infty} 
            \le M\Big(\norm{W_{\alpha,p}\abs{\emptyparam}^{-t}}^{p-1}_{L^p} + \norm{W_{\alpha,p}\abs{\emptyparam}^{-t}}^{(p-1)^\kappa}_{L^p}\Big).
        \end{equation*}
    \end{enumerate}
    
    \textbf{Step 3.} \emph{Conclusion of the argument assuming \cref{it:f1}--\cref{it:f3}.}
    Let us take facts \cref{it:f1}, \cref{it:f2}, and \cref{it:f3} for granted. 
    
    By \cref{it:f1} and \cref{it:f3}, we have (recall that $\tilde \Lambda_{\alpha,p}=\norm{W_{\alpha,p}\abs{\emptyparam}^{-t}}_{L^p}^{p-2}$)
    \begin{equation*}
        \limsup_{k\to \infty}\norm{W_{\alpha_k,p_k}\abs{\emptyparam}^{\frac n2-s}}^{p_k-2}_{L^\infty} 
        \le \tilde\Lambda_{\alpha_\infty,2}, 
    \end{equation*}
    and, thus, 
    \begin{equation*}
        \limsup_{k\to\infty} \mathcal F_{\alpha_k, 2}(\chi_{\alpha_k,p_k})
        \le \limsup_{k\to\infty}(p_k-1)\norm{W_{\alpha_k,p_k}\abs{\emptyparam}^{\frac n2-s}}_{L^\infty}^{p_k-2}
        \le \tilde\Lambda_{\alpha_\infty,2}.
    \end{equation*}
    This is in contradiction with fact \cref{it:f2}. We deduce that $\chi_{\alpha_k,p_k}\equiv 0$ for $k$ sufficiently large and, since the rotation $R$ that defines $\chi_{\alpha,p}$ was chosen arbitrarily, we conclude that $W_{\alpha_k,p_k}$ is radial for $k$ sufficiently large.\footnote{We remark that there exists a finite set of rotations such that any $L^1_{\mathrm{loc}}$-function invariant under the action of such rotations must be radial.}

  \textbf{Step 4.} \emph{Proving \cref{it:f1}--\cref{it:f3}.} To conclude the proof, it remains to show that \cref{it:f1}, \cref{it:f2}, and \cref{it:f3} hold.  
    
    To prove \cref{it:f1}, it is sufficient to observe that $\mathcal F_{\alpha_k,p_k}\to \mathcal F_{\alpha_\infty,2}$ pointwise as $k\to \infty$ on smooth functions with compact support.

    For \cref{it:f2}, we observe that $\chi$ satisfies the Hardy inequality with a constant $C_s^{(1)}$ (in the notation of \cite{Yafaev1999}) strictly smaller than $C_s^{(0)}$ (in the notation of \cite{Yafaev1999}). That is,
    \begin{equation*}
        C_s^{(1)}\norm{\chi}_{\dot H^s}^2 \ge  \norm{\chi\abs{\emptyparam}^{-s}}_{L^2}^2,
    \end{equation*}
    whereas the optimal constant in this inequality for arbitrary functions (without the constraint of being $L^2$-orthogonal to radial ones) would be $C_s^{(0)}$ instead of $C_s^{(1)}$. 
    This result is contained in  \cite[p. 2]{Yafaev1999} (in particular, see \cite[Lemma 2.1, eq. (2.26), eq. (2.28), Theorem 2.9]{Yafaev1999}; cf. also the study of sharp constants in the Hardy inequality in \cite{Herbst77,FrankSeiringer}). {Notice that the proof of \cite[eq. (2.28)]{Yafaev1999} gives in fact the \emph{strict} inequality $C_s^{(0)} > C_s^{(1)}$ for every $s < \frac{n}{2}$.} 
    
    As a consequence, we deduce $\tilde\Lambda_{\alpha,2}= \frac1{C_s^{(0)}}+C(\alpha)$ while $\mathcal F_{\alpha, 2}(\chi)\ge \frac1{C_s^{(1)}}+C(\alpha)$ for all $\chi$ orthogonal to radial functions. The statement of \cref{it:f2} follows from the continuity of $C(\alpha)$.
    
    The estimate \cref{it:f3} is the content of \cref{lem:asymptotic-growth} below.
\end{proof} 

\begin{lemma}[Asymptotic growth for subsolutions of a nonlinear elliptic problem] \label{lem:asymptotic-growth}
    For any integer $n\ge 1$, $0<s<\min\{1, n/2\}$, and $\eps>0$, there exists a constant $C>0$ and a positive integer $\kappa$ so that the following statement holds.

    Let us assume that $1\le p < 2^*_s-\eps$.
    Let $w\in L^1_{\mathrm{loc}}(\R^n)$ be a non-negative function such that $\fint_{B_R} w\to 0$ as $R\to\infty$. 
    If $(-\lapl)^s w \le w^{p-1}\abs{x}^{-tp}$, then, for all $x\in\R^n$,
    \begin{equation*}
        w(x) \le C \big(M^{p-1} + M^{(p-1)^\kappa}\big) \frac{1}{\abs{x}^{\frac n2-s}},
    \end{equation*}
    where $M\defeq \norm{w\abs{\emptyparam}^{-t}}_{L^p}$.
\end{lemma}

\begin{remark}
We observe that $w(x)=\abs{x}^{-(\frac n2-s)}$ solves $(-\lapl)^s w = A w^{p-1}\abs{x}^{-tp}$ with $A\defeq \Big(\frac{2^s\Gamma(\frac n2+s)}{\Gamma(\frac n2-s)}\Big)^2$ (see \cite[Table 1]{MR3888401}). 
\end{remark}

\begin{proof} \, 
\textbf{Step 1.} \emph{Simplification via scaling.}
Given $r>0$, let $w_r(x) \defeq w(rx)\abs{r}^{\frac n2-s}$. We have $(-\lapl)^s w_r \le w_r^{p-1}\abs{x}^{-tp}$ and $\norm{w\abs{\emptyparam}^{-t}}_{L^p}=\norm{w_r\abs{\emptyparam}^{-t}}_{L^p}$ and  the desired statement is equivalent to 
\begin{equation*}
    w_r(x) \le C(M^{p-1}+M^{(p-1)^\kappa}) \quad\text{for all $r>0$ and all $\abs{x}=1$}.
\end{equation*}

Hence, since all the assumptions are invariant when replacing $w$ with $w_r$, it is sufficient to show
\begin{equation}\label{eq:local47}
    w(x) \le C(M^{p-1} + M^{(p-1)^\kappa}) \quad\text{ for all $x\in\R^n$ with $\abs{x}=1$.}
\end{equation}

\textbf{Step 2.} \emph{Main estimate.}
Using the differential inequality satisfied by $w$, we are going to show an inequality that controls $w$ pointwise in an annulus with a weighted integral of $w$ in a larger annulus (see \cref{eq:local-main-estimate} below).

Owing to the maximum principle stated in  \cref{thm:global-max-principle}, we have    
\begin{equation}\label{eq:int_ineq}
    w(x) \le C_{n,s} \int_{\R^n} \frac{w^{p-1}(y)\abs{y}^{-tp}}{\abs{x-y}^{n-2s}}\, \mathrm{d} y \quad\text{for almost every $x\in\R^n$}.
\end{equation}

In this proof, we will use the notation $A\lesssim B$ as a shorthand for $A\le \tilde CB$, where $\tilde C=\tilde C(n, s)>0$ is a constant depending only on $n$ and $s$. Let us fix $x\in B_2\setminus B_1$.
Given $0<\ell<1$, let us define the annulus $\annulus(\ell)\defeq B_{1+\ell}\setminus B_{1-\ell}$.

By H\"older's inequality, assuming $x\in \annulus(\ell)$ with $\ell<\frac12$, we have
\begin{align}
    \int_{B_{1-2\ell}} \frac{w^{p-1}(y)\abs{y}^{-tp}}{\abs{x-y}^{n-2s}}\, \mathrm{d} y
    &\lesssim \ell^{-(n-2s)}\norm{w^{p-1}\abs{\emptyparam}^{-t(p-1)}}_{L^{\frac p{p-1}}}
    \norm{\abs{\emptyparam}^{-t}}_{L^p(B_1)} 
    \lesssim \ell^{-(n-2s)}M^{p-1}, 
    \label{eq:local87}
    \\
    \int_{B_{1+2\ell}^{\complement}} \frac{w^{p-1}(y)\abs{y}^{-tp}}{\abs{x-y}^{n-2s}}\, \mathrm{d} y
    &\lesssim \ell^{-(n-2s)}\norm{w^{p-1}\abs{\emptyparam}^{-t(p-1)}}_{L^{\frac p{p-1}}}
    \norm{\abs{\emptyparam}^{-t-n+2s}}_{L^p(B_1^{\complement})} 
    \lesssim \ell^{-(n-2s)}M^{p-1},
    \label{eq:local88uff}
\end{align}
where we have implicitly used $t < \frac np < t+n-2s$ (and the fact that the difference between these numbers is $\frac n2-s$, so it depends only on $n$ and $s$).

By inserting \cref{eq:local87,eq:local88uff} into \cref{eq:int_ineq}, we obtain
\begin{equation}\label{eq:local-main-estimate}
    w(x) 
    \lesssim 
    \ell^{-(n-2s)}M^{p-1} 
    + \int_{\annulus(2\ell)} \frac{w^{p-1}(y)}{\abs{x-y}^{n-2s}}\, \mathrm{d} y \quad\text{for almost every $x\in \annulus(\ell)$.}
\end{equation}
From now on, we can forget about the equation satisfied by $w$ as we will only use \cref{eq:local-main-estimate}.

\textbf{Step 3.} \emph{Technical estimates for the iteration.}
We are going to prove that the $L^r$-norm of $w$ in $\annulus(\ell)$ is controlled by the $L^q$-norm of $w$ in $\annulus(2\ell)$ (i.e., the two inequalities \cref{eq:ineq-to-iterate,eq:ineq-to-conclude}). There will be a certain threshold $q_+$ so that if $q<q_+$ then $r$ will be a finite value larger than $q$, while if $q>q_+$ then $r=\infty$.
Some care is necessary to obtain uniformity in $p$ of the constants involved (recalling that the constants $C$ and $\kappa$ of \cref{eq:local47} are not allowed to depend on $p$).

Let us assume that $1\le q, r<\infty$ are such that $p-1<q$ and $n-2s=n(1+\frac 1r-\frac{p-1}{q})$. Then, the Hardy--Littlewood--Sobolev inequality (see, e.g.,~\cite{Lieb1983}) yields
\begin{equation}\label{eq:hls-application}
\begin{aligned}
    \norm*{\big(w^{p-1}\chi_{\annulus(2\ell)}\big)\ast \abs{\emptyparam}^{-(n-2s)}}_{L^r}
    &\le
    C(n-2s, \frac{q}{p-1}, r)
    \norm*{w^{p-1}\chi_{\annulus(2\ell)}}_{\frac{q}{p-1}} \\
    &=
    C_{\mathrm{HLS}}(n-2s, \frac{q}{p-1}, r)
    \norm*{w}^{p-1}_{L^q(\annulus(2\ell))}.
\end{aligned}
\end{equation}

Let us understand better the relationship between $q$ and $r$. 
Let $q_-\defeq (p-2)\frac{n}{2s}$ and $q_+\defeq (p-1)\frac{n}{2s}$.
If $p-1 < q < q_+$, the value of $r$ is given by
\begin{equation*}
    r = g(q)\defeq \left(\frac{p-1}q - \frac{2s}n\right)^{-1}.
\end{equation*} 
If $q_- < q < q_+$, then
\begin{equation}\label{eq:f-exponential-iteration}
    \frac{g(q)}q =
    \left(p-1 - \frac{2sq}n\right)^{-1}
    =
    \left(1-\frac{2s}{n}(q-q_-)\right)^{-1} > 1.
\end{equation}
The estimate \cref{eq:hls-application} together with \cref{eq:f-exponential-iteration}, will guarantee a uniform improvement of the integrability at each iteration step.

If $p-1<q<q_+$, by taking the $L^r$-norm of \cref{eq:local-main-estimate} and applying \cref{eq:hls-application}, we obtain
\begin{equation*}
    \norm{w}_{L^r(\annulus(\ell))}
    \lesssim
    \ell^{-(n-2s)}M^{p-1}
    +
    C_{\mathrm{HLS}}\big(n-2s, \frac{q}{p-1}, r\big)\norm{w}^{p-1}_{L^q(\annulus(2\ell))},
\end{equation*}
where $r=g(q)$.
Since $\annulus(\ell)$ has bounded measure, this estimate implies that, if $\max\{p-1,q_-\}\le q < q_+$, then 
\begin{equation}\label{eq:ineq-to-iterate}
    \norm{w}_{L^r(\annulus(\ell))}
    \lesssim
    \ell^{-(n-2s)}M^{p-1}
    +
    C_{\mathrm{HLS}}\big(n-2s, \frac{q}{p-1}, g(q)\big)\norm{w}^{p-1}_{L^q(\annulus(2\ell))}
\end{equation}
for all $1\le r \le g(q)$.

For $q>q_+$, we have (implicitly using $1<p<2^*_s$, $q_+>p-1$, and denoting by $\big(\frac q{p-1}\big)'$ the Hölder conjugate exponent of $\frac q{p-1}$)
\begin{align*}
    \norm{\abs{\emptyparam}^{-(n-2s)}}_{L^{\left(\frac q{p-1}\right)'}(B_3)}
    &\lesssim
    \left(n-(n-2s)\left(\frac q{p-1}\right)'\right)^{-(1-\frac{p-1}q)}
    \\ 
    & = 
    \left(\frac{q_+-(p-1)}{2s(q-q_+)}+\frac1{2s}\right)^{1-\frac{p-1}{q}}
    \lesssim
    1+\frac{1}{q-q_+}.
\end{align*}
Hence, by taking the $L^{\infty}$-norm of \cref{eq:local-main-estimate} and applying Hölder's inequality, we obtain
\begin{equation}\label{eq:ineq-to-conclude}
    \norm{w}_{L^\infty(\annulus(\ell))}
    \lesssim
    \ell^{-(n-2s)}M^{p-1}
    +
    \left(1+\frac{1}{q-q_+}\right)\norm{w}^{p-1}_{L^q(\annulus(2\ell))}
\end{equation}
for any $q>q_+$.

\textbf{Step 4.} \emph{Iteration argument uniform in $p$.}
To show \cref{eq:local47} we employ an iterative argument (with finitely many steps) over annuli that is built upon \cref{eq:ineq-to-iterate} (the last step of the iteration will use \cref{eq:ineq-to-conclude}).

The fact that $p-q_-=(2^*_s-p)\frac{n-2s}{2s}\gtrsim \eps$ and  \cref{eq:f-exponential-iteration} allow us to find $\delta=\delta(n, s,\eps)$ and $\kappa=\kappa(n, s,\eps)$ such that there exists a sequence $q_1<q_2<\cdots<q_{\kappa}$ satisfying 
\begin{alignat*}{2}
    \max\{p-1, q_-\}+\delta < &\,q_1 &\le &\,p,\\
    &\,q_{i+1}&\le &\,g(q_i), \\
    &\,q_{\kappa-1}&\le &\, q_+-\delta,\\
    &\,q_{\kappa}&\ge &\,q_++\delta
\end{alignat*}
(here notice that $q_1$ may be smaller than $1$).

By definition of $M$, we have
\begin{equation*}
    \int_{\annulus(\frac12)} w^{p}(y)\, \mathrm{d} y 
    \le 2^{tp}M^p \lesssim M^p; 
\end{equation*}
thus $\norm{w}_{L^p(\annulus(\frac12))} \lesssim M$. Since $q_1\le p$, we deduce the starting point of our iteration procedure, that is 
\begin{equation}\label{eq:iteration-first-step}
    \norm{w}_{L^{q_1}(\annulus(\frac12))}\lesssim M
\end{equation}
holds. 

For any $(p-1)+\delta<q<q_+-\delta$, the optimal constant of the Hardy--Littlewood--Sobolev inequality $C_{\mathrm{HLS}}(n-2s, \frac{q}{p-1}, g(q))$ is bounded from above by a constant that depends only on $n, s, \eps$ (as shown in \cite{Lieb1983}).
Therefore, due to \cref{eq:ineq-to-iterate}, there exists a constant $C_1= C_1(n, s, \eps)$ such that, for any $1\le i\le \kappa-1$, 
\begin{equation}\label{eq:iteration-middle-step}
    \norm{w}_{L^{q_{i+1}}(\annulus(2^{-(i+1)}))}
    \le 
    C_1 \big(M^{p-1} + \norm{w}_{L^{q_i}(\annulus(2^{-i}))}^{p-1}\big).
\end{equation}
By concatenating \cref{eq:iteration-first-step} and \cref{eq:iteration-middle-step}, we deduce that there exists a constant $C_2=C_2(n, s,\eps)$ such that
\begin{equation*}
    \norm{w}_{L^{q_\kappa}(\annulus(2^{-\kappa}))}
    \le C_2 (M^{p-1}+M^{(p-1)^{\kappa-1}}). 
\end{equation*}
Since $q_\kappa>q_++\delta$, we can conclude and obtain \cref{eq:local47} by concatenating the last inequality with \cref{eq:ineq-to-conclude}.
\end{proof}


\appendix
\section{Smoothing a singular decreasing convex profile}
\label{sec:smoothing}

For a convex decreasing profile $\varphi:(0,\infty)\to(0,\infty)$ that has a singularity at $0$, we construct a regularized function $\psi\in C^\infty(0,\infty)$ that coincides with $\varphi$ on $[1,\infty)$, is decreasing, and has the property that at each point $0<r<1$ the graph of $\psi$ can be touched from above by a suitable translation of the graph of $\varphi$ (see \textit{(2)} of \cref{lem:smoothing_argument_consequence}).
This construction is necessary in the proof of \cref{prop:Gamma_def}.

\begin{lemma}\label{lem:crucial_smoothing_argument}
    Let $\varphi:(0,\infty)\to(0,\infty)$ be a function such that:
    \begin{itemize}
        \item $\varphi\in C^\infty((0,\infty))$,
        \item $\varphi(0^+) = \infty$,
        \item $\varphi(\infty) = 0$,
        \item $\varphi' < 0$,
        \item $\varphi'' > 0$.
    \end{itemize}
    There exists a function $\psi:(0,\infty)\to(0,\infty)$ such that:
    \begin{enumerate}
        \item $\psi\in C^\infty((0,\infty))$,
        \item $\psi(r)=a-br^2$, for some $a,b>0$, for all $0<r<\frac12$,
        \item $\psi = \varphi$ in $[1,\infty)$,
        \item $\psi'<0$,
        \item If $\psi'(r_0)=\varphi'(r_1)$ for some $0 < r_0 < 1$ and $r_1>0$, then $\psi''(r_0) < \varphi''(r_1)$.
        \item There exists $\bar r>0$ such that $\frac{\psi'}{\psi}\ge \frac{\varphi'}{\varphi}$ in $[\bar r, 1]$ and $\psi''<0$ in $(0, \bar r]$.
    \end{enumerate}
\end{lemma}
\begin{proof}
    Let $\theta:\R\to[0,\infty)$ be the smooth function
    \begin{equation*}
        \theta(r)\defeq
        \begin{cases}
            0 \quad&\text{if $r<0$,} \\
            e^{-\frac1r}\quad&\text{if $r\ge 0$.}
        \end{cases}
    \end{equation*}
    Observe that $\frac{\theta'(r)}{\theta(r)}=\frac1{r^2}$ and $\frac{\theta''(r)}{\theta'(r)}=\frac1{r^2}-\frac2r$.
    
    Given $\kappa>0$, let us consider $r\mapsto \varphi_\kappa(r)\defeq \varphi(r)-\kappa\theta(1-r)$.
    Observe that
    \begin{equation*}
        \varphi'_\kappa(r) = \varphi'(r)+\kappa\theta'(1-r)
        \quad
        \varphi''_\kappa(r) = \varphi''(r) - \kappa\theta''(1-r).
    \end{equation*}
    It holds $\varphi''_\kappa(1)>0$ and, if $\kappa$ is sufficiently large, $\varphi'_\kappa(\frac12)>\varphi'_\kappa(1)=\varphi'(1)$. Thus, for $\kappa$ large enough, we can define $\frac12<\bar r_\kappa<1$ as the maximum point such that $\varphi'_\kappa(\bar r_\kappa)=\varphi'_\kappa(1)=\varphi'(1)$.
    Furthermore, $\bar r_\kappa\to 1^-$ as $\kappa\to \infty$ and, for $\kappa$ large enough, $\varphi''_\kappa(\bar r_\kappa)<0$ (because $\tfrac{\theta''(1-r)}{\theta'(1-r)}\to\infty$ as $r\to 1^-$).

    Observe that, by smoothness of $\varphi$, there exists a constant $C>0$ such that
    \begin{equation*}
        \frac{\varphi''(r_1)-\varphi''(r_0)}{\varphi'(r_1)-\varphi'(r_0)} > -C \quad\text{for all $\frac12<r_0<r_1<1$.}
    \end{equation*}
    By choosing $\kappa$ large enough we can assume that $\frac{\theta''(1-r)}{\theta'(1-r)}>C$ on $(\bar r_\kappa, 1)$. Analogously, we can also assume that $\frac{\theta'(1-r)}{\theta(1-r)}>-\frac{\varphi'(r)}{\varphi(r)}$ for $r\in(\bar r_\kappa, 1)$.

    From now on, we fix $\kappa$ large enough so that all properties mentioned above hold.
    We are going to construct a smooth function $\psi$ such that:
    \begin{itemize}
        \item $\psi(r)=a-br^2$ on $(0, \frac12)$ for some $a, b>0$,
        \item $\psi''<0$ on $[\frac12, \bar r_\kappa]$;
        \item $\psi=\varphi_\kappa$ on $[\bar r_\kappa, \infty)$.
    \end{itemize}
    In order to find the sought smooth concave extension on $[\frac12, \bar r_\kappa]$, the boundary conditions to satisfy are:
    \begin{align*}
        &\psi''(\tfrac12) < 0,\\
        &\psi''(\bar r_\kappa)<0,\\
        &\psi(\tfrac12) + (\bar r_\kappa-\tfrac12)\psi'(\tfrac12)\ge \psi(\bar r_\kappa),\\
        &\psi(\bar r_\kappa) + (\tfrac12-\bar r_\kappa)\psi'(\bar r_\kappa) \ge \psi(\tfrac12).
    \end{align*}
    The first condition holds because $b>0$; we have already checked that the second condition holds; the last two conditions hold if $a>\psi(\bar r_\kappa)$ is chosen close enough to $\psi(\bar r_\kappa)$ and then $b$ is chosen small enough.

    We shall verify that $\psi$ satisfies all the constraints mentioned in the statements. Only properties \textit{(5)} and \textit{(6)} are nontrivial to check for our choice of $\psi$.

    For \textit{(5)}, if $0<r_0\le\bar r_\kappa$, then $\psi''(r_0) < 0 < \varphi''(r_1)$. 
    For $\bar r_\kappa<r_0<1$, since $\varphi'(r_0) < \varphi'(r_0) + \kappa\theta'(1-r_0) = \psi'(r_0)<\varphi'(1)$, we deduce that $r_0<r_1<1$.
    Therefore, we have
    \begin{equation*}
        \frac{\varphi''(r_1)-\varphi''(r_0)}{\kappa\theta'(1-r_0)}
        =
        \frac{\varphi''(r_1)-\varphi''(r_0)}{\varphi'(r_1)-\varphi'(r_0)}
        > - C
        >
        -\frac{\theta''(1-r_0)}{\theta'(1-r_0)}
    \end{equation*}
    and thus $\varphi''(r_1)-\varphi''(r_0)>-\theta''(1-r_0)$, which is equivalent to $\varphi''(r_1)>\psi''(r_0)$ as desired.

    For \textit{(6)}, we set $\bar r=\bar r_\kappa$. The condition $\psi''<0$ in $(0, \bar r]$ holds by definition of $\psi$. For $\bar r<r<1$, observe that
    \begin{equation*}
        \frac{\psi'(r)}{\psi(r)} \ge \frac{\varphi'(r)}{\varphi(r)} \iff
        \frac{\theta'(1-r)}{\theta(1-r)}
        \ge
        -\frac{\varphi'(r)}{\varphi(r)},
    \end{equation*}
    which is true because of our choice of $\kappa$.
\end{proof}

\begin{lemma}\label{lem:smoothing_argument_consequence}
    Let $\varphi:(0,\infty)\to(0,\infty)$ be a function satisfying the assumptions of \cref{lem:crucial_smoothing_argument} and let $\psi:(0,\infty)\to(0,\infty)$ be a function satisfying the properties mentioned in \cref{lem:crucial_smoothing_argument}.

    Then:
    \begin{enumerate}
        \item It holds $\psi\le \varphi$ and $\varphi'\le \psi'$.
        \item Given $0<r_0<1$ there exists $r_0 < r_1$ such that the function 
        \begin{equation*}
            (0,\infty)\ni r\mapsto \varphi(r+(r_1-r_0))-\psi(r)
        \end{equation*}
        has a global minimum at $r=r_0$.
    \end{enumerate}
\end{lemma}
\begin{proof}
    Notice that $\psi'$ attains only values in $(-\infty,0)$ and $\varphi'$ is a smooth diffeomorphism between $(0,\infty)$ and $(-\infty, 0)$. Hence, the map $F\defeq (\varphi')^{-1}\circ \psi':(0,1)\to(0,\infty)$ is well defined and smooth.

    We prove that $(0,1)\ni r\mapsto F(r)-r$ is strictly decreasing. Indeed, its derivative at $0<r<1$ satisfies
    \begin{equation*}
        F'(r) = \frac{\psi''(r)}{\varphi''(F(r))}-1 < 0,
    \end{equation*}
    where we applied property \textit{(5)} of $\psi$ to justify the last inequality.
    Observe that $F(1)=1$, so $r<F(r)$ for $0<r<1$.

    Now we prove the two desired statements.
    \begin{enumerate}
        \item Observe that $\psi(\infty)=\varphi(\infty)=0$, so if we prove $\psi'\ge \varphi'$ then $\psi\le \varphi$ follows.

        Since $\psi=\varphi$ in $[1,\infty)$, then also $\psi'=\varphi'$ in such domain.
        For $0<r<1$, we have $\psi'(r) = \varphi'(F(r)) \ge \varphi'(r)$ because $F(r)>r$.
        \item Given $0<r_0<1$, let $r_1=F(r_0)$. We have already shown that $r_0<r_1$. Let $Q(r)\defeq  \varphi(r+(r_1-r_0))-\psi(r)$. This function is smooth on $(0, \infty)$, it satisfies $Q'(r_0)=0$ and $Q''(r_0) > 0$ (we are applying property \textit{(5)} of $\psi$). 
        Thus, to show that $r_0$ is a global minimum point for $Q$ it is sufficient to show that $Q'(r)=0$ implies $r=r_0$.
        For $0<r<1$, we can have $Q'(r)=0$ if and only if $F(r)=r+(r_1-r_0)$, but this can happen only if $r=r_0$ since $F(r)-r$ is strictly monotone.
        For $1\le r$, we have $Q'(r) = \varphi'(r+r_1-r_0)-\varphi'(r) > 0$ because $\varphi=\psi$ on $[1,\infty)$ and $\varphi$ is strictly convex.
    \end{enumerate}
\end{proof}

\vspace{0.5cm}
\section*{Acknowledgments}
N. De Nitti is a member of the Gruppo Nazionale per l'Analisi Matematica, la Probabilità e le loro Applicazioni (GNAMPA) of the Istituto Nazionale di Alta Matematica (INdAM). He has been partially supported by the Swiss State Secretariat for Education, Research and Innovation (SERI) under contract number MB22.00034 through the project TENSE. F. Glaudo is supported by the National Science Foundation under Grant No. DMS–1926686.

We are grateful for the hospitality of the Friedrich-Alexander-Universität Erlangen-Nürnberg, where part of this work was carried out in May 2022.

We thank Alessio Figalli for discussions that led to \cref{remark s=1} and Rupert Frank for suggesting to prove \cref{corollary remainder CKN}.
We thank Maria del Mar González, Roberta Musina, and Dino Sciunzi for some helpful comments on the topic of this work.  

\vspace{0.5cm}

\printbibliography

\vfill 

\end{document}